\documentclass[a4paper,11pt]{article}
\usepackage{amsmath,amsthm,amssymb,enumitem,xcolor}
\usepackage[hang]{footmisc}

\usepackage[nosort,nocompress,noadjust]{cite}

\usepackage[bookmarks=false,hyperfootnotes=false,colorlinks,linktoc=all,
    linkcolor={red!60!black},
    citecolor={blue!50!black},
    urlcolor={blue!80!black}]{hyperref}

\usepackage{tocloft}

\usepackage[normalem]{ulem}

\renewcommand{\eqref}[1]{\hyperref[#1]{(\ref{#1})}}

\newlist{enumlist}{enumerate}{2}
\setlist[enumlist,1]{labelindent=0cm,label=\arabic*.,ref=\arabic*,labelwidth=2.5ex,labelsep=0.5ex,leftmargin=3ex,align=left,topsep=0.5ex,itemsep=1ex,parsep=1ex}
\setlist[enumlist,2]{labelindent=0cm,label=\theenumlisti.\arabic*.,ref=\arabic*,labelwidth=5ex,labelsep=0.5ex,leftmargin=5.5ex,align=left,topsep=0.5ex,itemsep=1ex,parsep=1ex}

\newlist{itemlist}{itemize}{1}
\setlist[itemlist]{labelindent=0cm,label=$\bullet$,labelwidth=2.5ex,labelsep=0.5ex,leftmargin=3ex,align=left,topsep=0.5ex,itemsep=1ex,parsep=1ex}

\numberwithin{equation}{section}

{\theoremstyle{definition}\newtheorem{definition}{Definition}[section]
\newtheorem*{definition*}{Definition}
\newtheorem{notation}[definition]{Notation}

\newtheorem{remark}[definition]{Remark}
\newtheorem{example}[definition]{Example}
\newtheorem{assumptions}[definition]{Assumptions}
\newtheorem*{example*}{Example}
\newtheorem*{examples*}{Examples}}

\newtheorem{proposition}[definition]{Proposition}
\newtheorem{lemma}[definition]{Lemma}
\newtheorem{theorem}[definition]{Theorem}
\newtheorem{corollary}[definition]{Corollary}

\newtheorem{letterthm}{Theorem}
\newtheorem{lettercor}[letterthm]{Corollary}

{\theoremstyle{definition}}

\renewcommand{\Re}{\operatorname{Re}}

\newcommand{\C}{\mathbb{C}}
\newcommand{\cC}{\mathcal{C}}
\newcommand{\eps}{\varepsilon}
\newcommand{\al}{\alpha}
\newcommand{\be}{\beta}
\newcommand{\ot}{\otimes}
\newcommand{\recht}{\rightarrow}
\newcommand{\Z}{\mathbb{Z}}
\newcommand{\vphi}{\varphi}
\newcommand{\cO}{\mathcal{O}}
\newcommand{\id}{\mathord{\text{\rm id}}}
\newcommand{\om}{\omega}
\newcommand{\N}{\mathbb{N}}
\newcommand{\ovt}{\mathbin{\overline{\otimes}}}
\newcommand{\Tr}{\operatorname{Tr}}

\newcommand{\si}{\sigma}
\newcommand{\R}{\mathbb{R}}
\newcommand{\F}{\mathbb{F}}
\newcommand{\cH}{\mathcal{H}}
\newcommand{\cZ}{\mathcal{Z}}
\newcommand{\Ad}{\operatorname{Ad}}
\newcommand{\cG}{\mathcal{G}}
\newcommand{\cK}{\mathcal{K}}

\newcommand{\cF}{\mathcal{F}}
\newcommand{\T}{\mathbb{T}}
\newcommand{\actson}{\curvearrowright}
\newcommand{\cA}{\mathcal{A}}
\newcommand{\cB}{\mathcal{B}}

\newcommand{\cU}{\mathcal{U}}
\newcommand{\Ker}{\operatorname{Ker}}

\newcommand{\lspan}{\operatorname{span}}
\newcommand{\cN}{\mathcal{N}}
\newcommand{\cR}{\mathcal{R}}

\newcommand{\cI}{\mathcal{I}}
\newcommand{\cV}{\mathcal{V}}

\newcommand{\Aut}{\operatorname{Aut}}

\newcommand{\cS}{\mathcal{S}}

\newcommand{\cL}{\mathcal{L}}

\newcommand{\dpr}{^{\prime\prime}}

\newcommand{\Q}{\mathbb{Q}}
\newcommand{\GL}{\operatorname{GL}}
\newcommand{\SL}{\operatorname{SL}}
\newcommand{\End}{\operatorname{End}}
\newcommand{\deltabar}{\overline{\delta}}
\newcommand{\Out}{\operatorname{Out}}
\newcommand{\Emb}{\operatorname{Emb}}
\newcommand{\Inj}{\operatorname{Inj}}

\newcommand{\tr}{\operatorname{tr}}
\newcommand{\Sp}{\operatorname{Sp}}
\newcommand{\PSL}{\operatorname{PSL}}
\newcommand{\bcG}{\text{\boldmath $\mathcal{G}$}}
\newcommand{\bcT}{\text{\boldmath $\mathcal{T}$}}
\newcommand{\bN}{\text{\boldmath $N$}}

\DeclareMathOperator*\bigast{\text{\Large$\ast$}}

\newcommand{\bim}[3]{\mathord{\raisebox{-0.4ex}[0ex][0ex]{\scriptsize $#1$}{#2}\hspace{-0.25ex}\raisebox{-0.4ex}[0ex][0ex]{\scriptsize $#3$}}}

\newcommand{\emb}{\hookrightarrow}
\newcommand{\semb}{\hookrightarrow_{\text{\rm s}}}
\newcommand{\cFs}{\mathcal{F}_{\text{\rm s}}}
\newcommand{\IIone}{\text{\bf II}_{\text{\bf 1}}}
\newcommand{\scong}{\cong_{\text{\rm s}}}
\newcommand{\Ibar}{\overline{I}}
\newcommand{\fin}{\text{\rm fin}}
\newcommand{\omtil}{\widetilde{\omega}}
\newcommand{\subg}{\leqslant}
\newcommand{\Embs}{\operatorname{Emb}_{\text{\rm s}}}
\newcommand{\inv}{\text{\rm inv}}
\newcommand{\Factor}{\operatorname{Factor}}

\begin{document}

\begin{center}
{\boldmath\LARGE\bf W$^*$-rigidity paradigms for embeddings of II$_1$ factors}

\bigskip

{\sc by Sorin Popa\footnote{Mathematics Department, UCLA, Los Angeles, CA 90095-1555 (United States), popa@math.ucla.edu\\ Supported in part by NSF Grant DMS-1955812 and the Takesaki Chair in Operator Algebras.} and Stefaan Vaes\footnote{KU~Leuven, Department of Mathematics, Leuven (Belgium), stefaan.vaes@kuleuven.be\\ Supported by FWO research project G090420N of the Research Foundation Flanders and by long term structural funding~-- Methusalem grant of the Flemish Government.}}
\end{center}

\begin{abstract}\noindent
We undertake a systematic study of W$^*$-rigidity paradigms for the embeddability relation $\emb$ between separable II$_1$ factors and its stable version
$\semb$, obtaining large families of non stably isomorphic II$_1$ factors that are mutually embeddable and families of II$_1$ factors that are mutually non stably embeddable. We provide an augmentation functor $G \mapsto H_G$ from the category of groups into icc groups, so that $L(H_{G_1}) \semb L(H_{G_2})$ iff $G_1 \emb G_2$. We construct complete intervals of II$_1$ factors, including a strict chain of II$_1$ factors $(M_k)_{k\in \Z}$ with the property that if $N$ is any II$_1$ factor with $M_i\semb N$ and $N\semb M_{j}$, then $N \cong M_k^t$ for some $i \leq k \leq j$ and $t > 0$.
\end{abstract}

\section{Introduction}

The {\it embedding problem} in von Neumann algebras asks whether factors arising from certain geometric data
can be embedded one into the other or not. It is a natural companion to the isomorphism/classification problem for factors. The two problems are
particularly interesting when the factors involved are of type II$_1$.
They are parallel, and often interconnected, with similar problems in geometric group theory, orbit equivalence ergodic theory and
Lie group dynamics.

The isomorphism and embedding  problems for II$_1$ factors can by and large be divided into two types of phenomena,
the \emph{many-to-one} paradigm
and the \emph{one-to-one} (or \emph{W$^*$-rigidity}) paradigm. While opposite in nature and involving completely different techniques,
both types of results are hard to establish and they often come as a surprise.
For the isomorphism problem, a result showing that distinct geometric data $G$ (like groups, groupoids,
groups acting on spaces, etc.) give rise to the same II$_1$ factor $L(G)$, is a many-to-one paradigm, while a W$^*$-rigidity paradigm amounts
to identifying a class $\bcG$ of objects for which the functor $\bcG \ni G \mapsto L(G)\in \IIone$ is one-to-one, modulo isomorphisms in $\bcG$, respectively $\IIone$.

The typical many-to-one result  for the embedding problem singles out a class of  II$_1$ factors that
are mutually embeddable but not isomorphic.  In turn, its W$^*$-rigidity counterpart has  an array of interesting questions,
from the basic non-embeddability between specific II$_1$ factors, to obtaining families
of mutually non-embeddable factors (\emph{disjointness}), calculating all self-embeddings of factors within a certain class, etc. Viewing the embeddability
of II$_1$ factors as a preorder relation in the class $\IIone$, or alternatively as a weak subordination relation of the underlying geometric objects,
leads to a series of more nuanced problems.

Despite the diversity of interesting questions of this type, and in sharp contrast with the recent spectacular progress
on the isomorphism problem, the study of embedding phenomena has been much less developed. We undertake in this paper a systematic study of the embedding paradigms for II$_1$ factors,
using all tools presently available, especially deformation-rigidity theory, and obtaining a multitude of completely new types of results.

Thus, denoting by $\emb$ the embedding relation between II$_1$ factors and by $\semb$ its stable, weaker version involving amplifications of factors, we exhibit one-parameter families of non stably isomorphic separable II$_1$ factors $(M_t)_{t \in \R}$ of each of the following types: examples where the factors $M_t$ are mutually embeddable (a many-to-one result); examples where $M_t \semb M_r$ if and only if $t \leq r$ (a strict chain); and examples that are mutually non-embeddable (an anti-chain). We actually produce such families of II$_1$ factors $(M_i)_{i \in I}$ indexed by different types of partially ordered sets $(I,\leq)$. Moreover, we construct a large class of II$_1$ factors for which we can compute all \emph{stable self-embeddings} $M \emb M^t$, including II$_1$ factors $M$ for which all embeddings $M \emb M^t$ are inner, II$_1$ factors with a prescribed countable \emph{one-sided fundamental group}, and II$_1$ factors with prescribed \emph{outer automorphism group} from a wide range of Polish groups. Also, we construct an augmentation functor assigning to any infinite group $G$ an icc group $H_G$ such that $L(H_{G_1}) \semb L(H_{G_2})$ if and only if $G_1$ is isomorphic with a subgroup of $G_2$. Finally, we produce concrete \emph{complete intervals} of II$_1$ factors $(M_i)_{i \in I}$, indexed by a large variety of partially ordered sets, where completeness means that any intermediate II$_1$ factor $N$, with $M_i \semb N$ and $N \semb M_j$, is stably isomorphic with $M_k$ for some $i \leq k \leq j$.

Before stating these results in detail, and in order to put them in a proper perspective, we briefly review the history of the embedding problem and the main results.

Both the isomorphism  and embedding problems have been initiated by  Murray-von Neumann in \cite{MvN43}. They obtained  the first many-to-one results,
by showing that all ``locally finite'' geometric objects $G$ give rise to the {\it hyperfinite} II$_1$ factor, $R:=\overline{\otimes}_n (M_2(\C), \tr)_n$,
which in turn can be embedded into any other II$_1$ factor. Interestingly, they
also construct two non stably isomorphic, but mutually embeddable II$_1$ factors in \cite[Appendix]{MvN43} and they comment on \cite[page 717]{MvN43}:
\emph{``the possibility exists that any factor in the case {\rm II}$_1$ is isomorphic to a sub-ring of any other such factor.''} In other words, Murray and von Neumann did not exclude the possibility that
no matter the geometric data they may come from, all separable II$_1$ factors look alike under the very weak equivalence relation given by mutual embeddability!

It took more than two decades to prove that this is not the case. Thus, work in \cite{Sch63,HT67} provided
the right notion of amenability for II$_1$ factors, showed that it is inherited by embeddable factors
and that $L(G)$ is amenable as a II$_1$ factor iff $G$ is amenable (with its initial structure), leading
to the first non-embeddability result: if $G$ is amenable and $G_0$ is not, then $L(G_0) \not\emb L(G)$, e.g. $L(\F_2) \not\emb L(S_\infty)$, where $\F_2$ is the free group on two generators and $S_\infty$ is the group of finite permutations of $\N$.
Connes' fundamental theorem \cite{Con76}  completed the picture: all II$_1$ factors $L(G)$ arising from amenable
geometric objects $G$, as well as all II$_1$ subfactors of $R$, are isomorphic to $R$. This reduced both the isomorphism and the embedding problems to
II$_1$ factors and initial data  that are nonamenable.

Two decades later, Connes and Jones obtained in \cite{CJ83} a landmark non-embeddability result for nonamenable II$_1$ factors,
showing that if $G_0$ is a group with Kazhdan's property (T), e.g.\ $G_0=\SL(3, \Z)$, and $G$ is a group having the Haagerup property, e.g.\ $G=\F_2$, then $L(G_0) \not\semb L(G)$. A few years later, Cowling and Haagerup proved in \cite{CH88} the striking, deep result that for $G_n = \Sp(1,n)_\Z$ with $n \geq 3$, one has $L(G_n)\semb L(G_m)$ iff $n \leq m$. Note that both of these W$^*$-rigidity statements
say that more rigid objects cannot be embedded into less rigid ones.

During that same period of time, several embedding W$^*$-rigidity results modulo countable classes
were obtained in \cite{Pop86}, using a separability pigeon-hole principle trick, showing
for instance that, within the class $\bcT$ of property (T) II$_1$ factors (as defined in \cite{CJ83}), the equivalence relation given
by mutual stable embedding with finite index (also called {\it virtual isomorphism}) is countable-to-one with respect to actual isomorphism.
The separability trick together with the abundance of simple property (T) groups in (see \cite{Gro87}) was then used in \cite{Oza02} to show
that the embedding preorder relation does not have an upper bound in the class of separable II$_1$ factors.

Starting in 2001, deformation/rigidity and intertwining-by-bimodules techniques allowed a systematic insight into W$^*$-rigidity phenomena.
While particularly revealing for the isomorphism problem, this led to several non-embeddability results as well.
For instance, it is shown in \cite{IPP05} that if $n>m$ then the II$_1$ factor arising from a free product of $n$ property (T) infinite groups
cannot be embedded into one that comes from the free product of $m$ property (T) groups. Also, it is shown in \cite{OP07}
that one cannot embed a nonamenable II$_1$ factor with Cartan subalgebra into a free group factor. On the other hand,
a non-embeddability result  based on amenability of boundary actions was obtained in \cite{Oza03}, showing
that if $N$ is a nonamenable II$_1$ factor that either has property Gamma of \cite{MvN43}, or
is a tensor product of II$_1$ factors, then it cannot be embedded into a free group factor.
A number of results were obtained about the non-embeddability of a tensor product of $n$ factors of type II$_1$ into a tensor product of a strictly smaller number of factors from a certain class, like free group factors, see e.g.\ \cite{OP03}.

The main focus of deformation/rigidity theory has been on the isomorphism and classification problem for families of II$_1$ factors. Nevertheless, several of these articles contained, both implicitly and explicitly,
partial results on the structure of embeddings between specific group II$_1$ factors and group measure space II$_1$ factors, most notably for Bernoulli crossed products, see e.g. \cite{Pop03,Pop06,CH08,PV09,Ioa10,CS11,KV15}.

But overall, these recent tools have been very little used in a systematic study of the embedding problem. In particular, the two most prominent such problems, going back to
\cite{MvN43}, remained open: finding large classes of mutually non stably embeddable II$_1$ factors; finding natural classes of mutually embeddable but non stably isomorphic II$_1$ factors.
While we were solving these problems, stated as Theorem \ref{thm.main-embedding} and Corollary \ref{cor.family-with-a} below, a large number of additional questions naturally emerged and led to Theorems~\ref{thm.one-sided-fundamental}~--~\ref{thm.complete-intervals}, attesting to the striking richness of such W$^*$-rigidity phenomena.

To state in details our main results, we need to fix some notations. As above, when $N$ and $M$ are II$_1$ factors, we write $N \emb M$ if there exists a unital normal $*$-homomorphism from $N$ to $M$. We write $N \semb M$, to be read as $N$ stably embeds into $M$, if $N \emb M^t$ for some $t > 0$ or, equivalently, $N^t \emb M$ for some $t > 0$. We write $M \scong N$ when $M$ and $N$ are \emph{stably isomorphic} II$_1$ factors, meaning that $M \cong N^t$ for some $t > 0$. Recall that a normal $*$-homomorphism between II$_1$ factors is necessarily a trace preserving embedding.

Our first result provides a family of II$_1$ factors indexed by tracial amenable von Neumann algebras and a precise characterization of when one can be (stably) embedded in the other.

\begin{letterthm}\label{thm.main-embedding}
Let $\Gamma = \F_n$ be the free group with $2 \leq n \leq +\infty$ generators. For every amenable tracial von Neumann algebra $(A_0,\tau_0)$ with $A_0 \neq \C 1$, we define the II$_1$ factor $M(n,A_0,\tau_0)$ as the left-right Bernoulli crossed product
\begin{equation}\label{eq.II1-base}
M(n,A_0,\tau_0) = (A_0,\tau_0)^{\overline{\ot} \Gamma} \rtimes (\Gamma \times \Gamma) \; .
\end{equation}
We have $M(n_0,A_0,\tau_0) \semb M(n_1,A_1,\tau_1)$ iff $M(n_0,A_0,\tau_0) \emb M(n_1,A_1,\tau_1)$ iff there exists a trace preserving unital embedding $(A_0,\tau_0) \emb (A_1,\tau_1)$.
\end{letterthm}

A few particular choices of $A_0$ lead to the following interesting examples, providing concrete families of II$_1$ factors that are mutually incomparable for $\semb$, as well as chains of II$_1$ factors for the relation $\semb$ and families of II$_1$ factors that are mutually embeddable, yet nonisomorphic and even non virtually isomorphic, meaning that it is impossible to embed one as a finite index subfactor of an amplification of the other (see Remark \ref{rem.virtually-isom}).

\begin{lettercor}\label{cor.family-with-a}
Fix $2 \leq n \leq +\infty$. Write $M(A_0,\tau_0) = M(n,A_0,\tau_0)$, using the notation \eqref{eq.II1-base}.
\begin{enumlist}
\item For $a \in (0,1/2]$, write $A_a = \C \oplus \C$ with $\tau_a(x \oplus y) = a x + (1-a)y$. Then $M_a = M(A_a,\tau_a)$ is an anti-chain for the preorder relations $\semb$ and $\emb$, i.e.\ $M_a \semb M_b$ iff $M_a \emb M_b$ iff $a = b$.

\item For $a \in (0,1]$, denote by $\eta_a$ the probability measure on $[0,a] \cup \{1\}$ given by the Lebesgue measure on $[0,a]$ and the atom $\eta_a(1) = 1-a$. Writing $(B_a,\tau_a) = L^\infty([0,a] \cup \{1\},\eta_a)$, the II$_1$ factors $P_a  = M(B_a,\tau_a)$ form a strict chain, i.e.\  $P_a \semb P_b$ iff $P_a \emb P_b$ iff $a \leq b$.

\item Let $R$ be the hyperfinite II$_1$ factor. For $a \in (0,1/2]$, define the trace $\tau_a$ on $R \oplus R$ by $\tau_a(x \oplus y) = a \tau(x) + (1-a) \tau(y)$. Write $Q_a = M(R \oplus R,\tau_a)$.
We have $Q_a \emb M(R,\tau) \emb Q_b$ for all $a,b \in (0,1/2]$, while $Q_a \cong Q_b$ iff $Q_a \scong Q_b$ iff $Q_a$, $Q_b$ are virtually isomorphic iff $a = b$.
\end{enumlist}
\end{lettercor}

Following \cite[Definition 10.4]{Ioa10}, one defines the \emph{one-sided fundamental group} $\cFs(M)$ of a II$_1$ factor $M$ as the set of $t > 0$ such that $M \emb M^t$. Note that $\cFs(M)$ also is the set of possible finite values of the right dimension $\dim (H_M)$ of a Hilbert $M$-bimodule $H$. Also note that $\cFs(M) \subseteq \R^*_+$ is closed under addition and multiplication. We always have $\N \subseteq \cFs(M)$ and the following dichotomy holds: either $\cFs(M) \cap (0,1) = \emptyset$, or $\cF_s(M) = \R^*_+$ (see \cite[Section 10]{Ioa10}).

The only existing computations of $\cFs(M)$ yield $\N$ or $\R^*_+$. Combining our results in Theorem \ref{thm.main-embedding} with the methods of \cite[Corollary 6.5]{IPP05}, we realize any countable semiring $\N \subseteq \cF \subseteq [1,+\infty)$ as one-sided fundamental group.

\begin{letterthm}\label{thm.one-sided-fundamental}
Fix one of the II$_1$ factors $M = M_a$ given in the first point of Corollary \ref{cor.family-with-a}. Let $\cF \subset [1,+\infty)$ be a countable subset satisfying $\cF + \cF \subseteq \cF$, $\cF \cF \subseteq \cF$ and $\N \subseteq \cF$. Define the free product II$_1$ factor
$$P = \bigast_{t \in \cF} M^{1/t} \; .$$
Then, $\cFs(P) = \cF$.
\end{letterthm}

We then turn to a systematic study of \emph{stable self-embeddings} $\theta : M \recht M^d$. We say that such a self-embedding is \emph{trivial} if $d \in \N$ and if $\theta$ is unitarily conjugate to $M \emb M_d(\C) \ot M : x \mapsto 1 \ot x$. It was conjectured in \cite[Section 10]{Ioa10} that there exist II$_1$ factors $M$ such that \emph{every} stable self-embedding $M \semb M$ is trivial. This conjecture was proven in \cite{Dep13}, with an extremely involved construction and the paper remaining unpublished. Our methods provide the following simpler example.

Denote by $G$ the group of all finite, even permutations of $\N$ and let $G_1 < G$ be the subgroup of permutations fixing $1 \in \N$. Put $\Gamma = G \ast G$ and $\Gamma_0 = G_1 \ast G$. Denote by $\cA \subset G_1$ the subset of elements of order $2$ and by $\cB \subset G$ the subset of elements of order $3$. View $\cA \subset \Gamma_0$ as elements in the first copy free product factor $G_1$ and view $\cB \subset \Gamma_0$ as elements in the second free product factor $G$.
For all $(a,b) \in \cA \times \cB$, define the subgroup $\Lambda_{a,b} < \Gamma_0 \times \Gamma_0$ generated by $(a,a)$ and $(b,b)$. Note that $\Lambda_{a,b} \cong \Z/2\Z \ast \Z/3\Z$.

\begin{letterthm}\label{thm.all-embed-trivial}
For all $(a,b) \in \cA \times \cB$, choose distinct probability measures $\mu_{a,b}$ on $\{0,1\}$ with $\mu_{a,b}(0) \in (0,1/2)$.
Consider the generalized Bernoulli action
$$\Gamma_0 \times \Gamma \actson (X,\mu) = \prod_{(a,b) \in \cA \times \cB} \bigl(\{0,1\},\mu_{a,b}\bigr)^{(\Gamma_0 \times \Gamma)/\Lambda_{a,b}}$$
and define $M = L^\infty(X,\mu) \rtimes (\Gamma \times \Gamma)$. Then every embedding $M \emb M^t$ is trivial.
\end{letterthm}

One may view Theorem \ref{thm.main-embedding} as a functorial construction that takes a base algebra $(A_0,\tau)$ as input and produces a II$_1$ factor such that embeddability holds if and only if the initial data are embeddable. We develop this approach much further and provide the following concrete augmentation functor assigning to any infinite group $\Gamma$ an icc group $H_\Gamma$, through a generalized wreath product construction, such that $L(H_\Gamma) \semb L(H_\Lambda)$ iff $L(H_\Gamma) \emb L(H_\Lambda)$ iff $\Gamma$ is isomorphic with a subgroup of $\Lambda$.

\begin{letterthm}\label{thm.family-with-groups}
Let $\Gamma$ be an infinite group (not necessarily countable). Denote by $G_\Gamma \cong \F_{1+|\Gamma|}$ the free group with free generators $a_0$ and $(a_g)_{g \in \Gamma}$. Define the surjective group homomorphism $\pi_\Gamma : G_\Gamma \recht \Z * \Gamma$ by $\pi_\Gamma(a_0) = 1 \in  \Z$ and $\pi_\Gamma(a_g) = g$ for all $g \in \Gamma$. Define the subgroup $N_\Gamma < G_\Gamma \times G_\Gamma$ consisting of the elements $(g,g)$ with $\pi_\Gamma(g) \in \Gamma$. Consider the generalized wreath product group
$$H_\Gamma = \bigl(\Z/2\Z\bigr)^{((G_\Gamma \times G_\Gamma)/N_\Gamma)} \rtimes (G_\Gamma \times G_\Gamma) \; .$$
If $\Gamma$ and $\Lambda$ are arbitrary infinite groups, we have $L(H_\Gamma) \semb L(H_\Lambda)$ iff $L(H_\Gamma) \emb L(H_\Lambda)$ iff $\Gamma$ is isomorphic with a subgroup of $\Lambda$.

The II$_1$ factors $L(H_\Gamma)$ and $L(H_\Lambda)$ are virtually isomorphic iff they are stably isomorphic iff they are isomorphic iff $\Gamma \cong \Lambda$.
\end{letterthm}

Both $\emb$ and $\semb$ define preorder relations on the class of II$_1$ factors. We similarly have a preorder relation $\emb$ on the class of groups defined by $\Gamma \emb \Lambda$ iff $\Gamma$ is isomorphic to a subgroup of $\Lambda$. Combining Theorem \ref{thm.family-with-groups} with the wildness of the relation $\emb$ on countable groups, we may concretely realize numerous partial orders and total orders $(I,\leq)$ inside $(\IIone,\emb)$ and $(\IIone,\mathord{\semb})$, see Corollary \ref{cor.examples-orders}.

More precisely, given a partially ordered set $(I,\leq)$, we want to construct a family of II$_1$ factors $(M_i)_{i \in I}$ such that $M_i \emb M_j$ (resp.\ $M_i \semb M_j$) if and only if $i \leq j$. It is natural to ask which partial orders can be realized in this way within II$_1$ factors having separable predual.

We provide the following concrete interesting examples. The meaning of ``concrete'' will become clear in the proof of Corollary \ref{cor.examples-orders}, which is entirely constructive. In the formulation of the corollary, we say that a subset $I_0 \subseteq I$ of a partially ordered set $(I,\leq)$ is \emph{sup-dense} if every element of $I$ is the supremum of a subset of $I_0$. We say that $(I,\leq)$ is separable if it admits a countable sup-dense subset. More generally, we say that $(I,\leq)$ has density character at most $\kappa$ if $\kappa$ is an infinite cardinal number and $(I,\leq)$ admits a sup-dense subset of cardinality at most $\kappa$. We similarly say that a II$_1$ factor $M$ has density character at most $\kappa$ if $M$ admits a $\|\,\cdot\,\|_2$-dense subset of cardinality at most $\kappa$.

\begin{lettercor}\label{cor.examples-orders}
Let $(I,\leq)$ be a partially ordered set with density character at most $\kappa$. There exists a concrete family of II$_1$ factors $(M_i)_{i \in I}$ with density character at most $\kappa$ such that $M_i \emb M_j$ iff $M_i \semb M_j$ iff $i \leq j$.

There are concrete chains of II$_1$ factors with separable predual, both in $(\IIone,\emb)$ and $(\IIone,\mathord{\semb})$, of order type $(\R,\leq)$ and of order type $(\om_1,\leq)$, where $\om_1$ is the first uncountable ordinal.
\end{lettercor}

In specific cases, we may even determine complete intervals in $(\IIone,\mathord{\semb})$, meaning that we not only provide families of II$_1$ factors $(M_i)_{i \in I}$ indexed by a partially ordered set $(I,\leq)$ with the property $M_i \semb M_j$ iff $i \leq j$, but also having the property that no other II$_1$ factors can sit between these $M_i$~: whenever $N$ is a II$_1$ factor and $M_i \semb N$ and $N \semb M_j$, there exists a $k \in I$ such that $i \leq k \leq j$ and $N \scong M_k$.

It is quite challenging to provide such complete intervals of II$_1$ factors, since we should at the same time control all possible embeddings $M_i \emb M_j^t$ \emph{and} determine all intermediate subfactors $M_i \emb N \emb M_j^t$. An appropriate variant of the construction in Theorem \ref{thm.all-embed-trivial} allows to do all this at once.

We can then provide a concrete construction of such complete intervals of II$_1$ factors indexed by several kinds of \emph{lattices} (i.e.\ partially ordered sets in which every pair of elements has an infimum and a supremum).

Our result covers all ``discrete'' lattices like $(\Z,\leq)$, arbitrary finite lattices or, more generally, any lattice $(I,\leq)$ with the property that $\{i \in I \mid a \leq i \leq b\}$ is finite for all $a,b \in I$.
Our result actually covers many \emph{complete} lattices, i.e.\ partially ordered sets in which every subset has an infimum and a supremum, like $(\Z \cup \{-\infty,+\infty\},\leq)$. In its most general form, our theorem realizes every separable \emph{algebraic\footnote{An algebraic lattice is a complete lattice $(I,\leq)$ in which every element can be written as a supremum of compact elements. An element $a \in I$ is called compact if whenever $a \leq \sup J$ for some $J \subseteq I$, then $a \leq \sup J_0$ for a finite subset $J_0 \subseteq J$.}} lattice as a complete interval of II$_1$ factors w.r.t.\ $\semb$.

Separable algebraic lattices are abundant. Any partially ordered set can be canonically completed into the algebraic lattice $(\Ibar,\leq)$, defined as the set of all downward closed subsets\footnote{A subset $J \subseteq I$ is called downward closed if for all $j \in J$ and all $k \in I$ with $k \leq j$, we have $k \in J$. Identifying $i \in I$ with the subset $\{j \in I \mid j \leq i\}$ of $I$, we embed $(I,\leq)$ in its completion $(\Ibar,\leq)$.} of $I$, ordered by inclusion. If $I$ is countable, $(\Ibar,\leq)$ is separable. Whenever $\Lambda$ is a group and $\Lambda_0 \subg \Lambda$ is a subgroup, the lattice of intermediate subgroups ordered by inclusion is an algebraic lattice. It is separable if $\Lambda$ is a countable group. By \cite{Tum86}, every algebraic lattice arises in this way. For our proof of Theorem \ref{thm.complete-intervals}, we will actually reprove this characterization with an extra control on the intermediate subgroups, like not having nontrivial finite dimensional unitary representations (see Theorem \ref{thm.realizing-lattices}). We do this by adapting the proof of \cite{Rep04}.

\begin{letterthm}\label{thm.complete-intervals}
Let $(I,\leq)$ be any separable algebraic lattice. There exists a family of separable II$_1$ factors $(M_i)_{i \in I}$ with the following properties:
\begin{itemlist}
\item we have $M_i \semb M_j$ iff $i \leq j$ iff $M_i \emb M_j$,
\item if $N$ is any II$_1$ factor such that $M_i \semb N$ and $N \semb M_j$ for some $i, j \in I$, there exists a $k \in I$ with $i \leq k \leq j$ and $N \scong M_k$,
\item if $N$ is any II$_1$ factor such that $M_i \emb N$ and $N \emb M_j$ for some $i, j \in I$, there exists a $k \in I$ with $i \leq k \leq j$ and $N \cong M_k$.
\end{itemlist}
\end{letterthm}

The level of concreteness of the family $(M_i)_{i \in I}$ depends on the nature of $(I,\leq)$. When $(I,\leq)$ is the completion $(\overline{I_0},\leq)$ of a countable partially ordered set, we give a separate proof of Theorem \ref{thm.complete-intervals} that is entirely constructive and essentially ``functorial'' in $(I,\leq)$: see Proposition \ref{prop.complete-interval-concrete-for-completion}. Note that this thus provides a concrete family of II$_1$ factors realizing complete intervals of the form $(\Z,\leq)$, or of the form $([0,\lambda],\leq)$, where $\lambda$ is any countable ordinal. For arbitrary separable algebraic lattices, our proof depends on a much less explicit inductive procedure.

We also provide the following example to illustrate how much weaker is the equivalence relation given by mutual embedding
``$N \emb M$ and $M \emb N$''  from isomorphism $N \cong M$ or stable isomorphism $N \scong M$.

Indeed, while the third point of Corollary \ref{cor.family-with-a} already provides an uncountable family of II$_1$ factors that are mutually embeddable but pairwise nonisomorphic, things can become considerably worse.  Thus, denote by $\cA$ the set of nonzero subgroups of $\Lambda_\infty = \Q^{(\N)}$, the direct sum of countably many copies of the additive group $\Q$. For every $\Lambda \in \cA$, using the notation of Theorem \ref{thm.family-with-groups}, we define $P_\Lambda = L(H_{\Lambda \times \Lambda_\infty})$. By construction, $P_\Lambda \emb P_{\Lambda'}$ for all $\Lambda, \Lambda' \in \cA$. On the other hand, $P_\Lambda \cong P_{\Lambda'}$ iff $P_\Lambda \scong P_{\Lambda'}$ iff $\Lambda \times \Lambda_\infty \cong \Lambda' \times \Lambda_\infty$. By \cite[Theorem 2.1]{DM07}, this is a \emph{complete analytic} equivalence relation on the standard Borel space $\cA$, i.e.\ an equivalence relation that is as complicated as it can be in this context.

Finally, variants of our method to prove Theorem \ref{thm.all-embed-trivial} also lead to constructions of II$_1$ factors $M$ for which the semiring of stable self-embeddings $M \semb M$ can be explicitly determined. We consider such stable embeddings up to unitary conjugacy. Since we can take direct sums and compositions of embedings, we obtain a semiring that we denote as $\Embs(M)$. We prove in Theorem \ref{thm.prescribed-embeddings-semiring} that for \emph{any} countable structure $\cG$ (like a countable group or a countable field) with semigroup of self-embeddings $\cS = \Emb(\cG)$, there exists a II$_1$ factor $M$ such that $\Embs(M) \cong \N[\cS]$, the semiring of formal sums of elements in $\cS$. The construction of $M$ is entirely explicit.

This even provides new results on the possible outer automorphism groups $\Out(M)$ of II$_1$ factors. We prove in Corollary \ref{cor.outer-automorphism-groups} that all of the following Polish groups can be realized as outer automorphism groups of full II$_1$ factors: any closed subgroup of the Polish group of all permutations of $\N$, the unitary group $\cU(H)$ of a separable Hilbert space and the unitary group $\cU(N)$ of any von Neumann algebra with separable predual. As we discuss in Section~\ref{sec.self-embed}, this encompasses all known realizations of outer automorphism groups and this is the first systematic result to realize a large class of non locally compact groups as $\Out(M)$.

We conclude this introduction with a few comments on our methods and proofs. Theorem \ref{thm.main-embedding} is a special case of Theorem \ref{thm.main-thm-embedding-bernoulli} below, which deals with a larger family of groups $\Gamma$ and a possibly nontrivial action $\Gamma \actson (A_0,\tau)$ as initial data for the construction. To analyze all possible embeddings $\theta : B \rtimes (\Lambda \times \Lambda) \recht (A \rtimes (\Gamma \times \Gamma))^d$ between such generalized Bernoulli actions, we use several results from deformation/rigidity theory for Bernoulli crossed products, see \cite{Pop03,Ioa10,IPV10,KV15}. This leads to a point where we know that after a unitary conjugacy, $\theta$ maps $L(\Lambda \times \Lambda)$ into $L(\Gamma \times \Gamma)^d$ such that the \emph{height} of $\theta(\Lambda \times \Lambda)$ inside $L(\Gamma \times \Gamma)^d$ is positive.

In \cite{IPV10}, it was proven that for \emph{isomorphisms} $\theta : L(\Lambda) \recht L(\Gamma)$, with moreover $d = 1$, positive height implies that $\theta$ can be unitarily conjugated to an isomorphism that comes from an isomorphism between the groups $\Lambda$ and $\Gamma$. In Section \ref{sec.embedding-group-vNalg}, we prove a generalization of this result for arbitrary embeddings $\theta : L(\Lambda) \recht L(\Gamma)^d$, also allowing for an arbitrary amplification. This unavoidably includes the appearance of finite index subgroups and finite-dimensional unitary representations.

Applying this result to our initial embedding of Bernoulli crossed products, the restriction of $\theta$ to $L(\Lambda \times \Lambda)$ must be group-like. From that point onward, it is easy to completely describe all possible embeddings, in terms of embeddings between the initial data $\Lambda \actson (B_0,\tau)$ and $\Gamma \actson (A_0,\tau)$.

One may thus summarize the approach to Theorems \ref{thm.main-embedding} and \ref{thm.main-thm-embedding-bernoulli} as follows: the context of Bernoulli crossed products is so rigid that any embedding must come from an embedding of the ``base spaces''. Once Theorem \ref{thm.main-embedding} is proven, Corollary \ref{cor.family-with-a} results from variants of the trivial observation that a base space with $3$ points cannot be embedded into a base space with $2$ points. The proof can be found on page \pageref{proof.corollary.B}. Our results on II$_1$ factors with prescribed embeddings semiring, or prescribed outer automorphism group, arise in the same way, by constructing initial data with prescribed symmetries. The proofs of Theorems~\ref{thm.one-sided-fundamental}~--~\ref{thm.complete-intervals} can be found using the table of contents at the end of this introduction.

The main focus and the main novelty of this paper are to consider arbitrary embeddings between arbitrary II$_1$ factors $N,  M$ and their amplifications by any $t>0$.
This allows us to obtain new types of results even in the case $N=M$, where Ioana
obtained in \cite{Ioa10} numerous rigidity results on self-embeddings $M \recht pMp$ between Bernoulli crossed product factors of product groups $\Gamma \times \Gamma$, including cases where this forces $p$ to be $1$ and the embedding to be canonical, but where
the cases $t > 1$ were left open (cf.\ Theorems \ref{thm.one-sided-fundamental} and \ref{thm.all-embed-trivial}).

Above we also highlighted how our construction provides large families of II$_1$ factors that are mutually embeddable, but not (stably) isomorphic. Deformation/rigidity theory has provided a wealth of classification theorems up to stable isomorphisms. They can be exploited to give many such families of mutually embeddable, but not (stably) isomorphic II$_1$ factors. For instance, whenever the fundamental group $\cF(M)$ is nontrivial, this fundamental group contains arbitrarily small $t > 0$. Taking direct sums of such $M \cong M^t$, we find that $M \emb M^s$ for all $s > 0$. Taking representatives $s > 0$ for the quotient $\R^*_+ / \cF(M)$, the II$_1$ factors $M^s$ are mutually embeddable, but not isomorphic. Of course, they are stably isomorphic by construction. One can obtain mutually embeddable, but not stably isomorphic II$_1$ factors by applying \cite{Ioa10,IPV10} to Bernoulli actions of mutually embeddable, nonisomorphic groups. Also the factor actions of a fixed Bernoulli action, as studied in \cite{Pop04}, lead to such examples since $1$-cohomology provides a method to distinguish these II$_1$ factors up to stable isomorphism.

As we already mentioned, our aim in this paper is to reveal unexpected W$^*$-rigidity paradigms and illustrate them with classes of examples that have intrinsic interest. But there are famous embedding problems that remain wide open. For instance, a natural embedding version of Connes' rigidity conjecture asks whether for $G_n=\PSL(n, \Z)$, $n\geq 3$, one has $L(G_n) \emb L(G_m)$ iff $n \leq m$.
Also, it is not known whether the class of nonamenable II$_1$ factors admits a least element for the preorder relations $\emb$ and $\semb$. A well known conjecture predicts that $L(\F_2)$ is such a least element, i.e.\ a nonamenable II$_1$ factor that embeds into any other nonamenable II$_1$ factor.
This can be viewed as the W$^*$-version of the von Neumann-Day conjecture, that any nonamenable group contains a copy of $\F_2$, shown to be false in \cite{Ols80}. But this W$^*$-version is believed to be true. Since the free group  factors $L(\F_n)$, $2\leq n \leq \infty$, are mutually embeddable, they would all be least elements in $(\IIone,\emb)$ and define a single class $\bN$ under the mutual embedding equivalence relation. How large that class $\bN$ would be, with respect to plain isomorphism (respectively stable isomorphism) of II$_1$ factors $N\in \bN$, depends on the answer to two other famous problems. On the one hand, it has been speculated (see e.g.\ \cite{PS18}) that if $N$ is nonamenable and $N\semb L(\F_2)$, then $N$ must be an interpolated free group factor $L(\F_t)$, for some $1< t\leq \infty$. Assuming this to be true as well,  $\bN$ would thus consist of all interpolated free group factors and its further description would depend on the answer to the free group factor problem. Recall in this respect that Voiculescu's free probability theory has been used in \cite{Dyk92,Rad92} to show that $L(\F_n)$, $2\leq n \leq \infty$, are either all isomorphic, or all nonisomorphic. It was also shown that $L(\F_n)$,  $2\leq n < \infty$, are stably isomorphic and that their nonisomorphism implies that $L(\F_\infty)$ is not stably isomorphic to any of them.

\begin{center}
\begin{minipage}{330pt}
\tableofcontents
\end{minipage}
\end{center}

\section{Embeddings of group von Neumann algebras}\label{sec.embedding-group-vNalg}

Let $\Gamma$ be a countable group and $\cG \subg \cU(p (M_n(\C) \ot L(\Gamma))p)$ a subgroup. As in \cite[Section 4]{Ioa10} and \cite[Section 3]{IPV10}, we say that $\cG$ has \emph{zero height}, denoted as $h_\Gamma(\cG) = 0$, if there exists a sequence $v_n \in \cG$ such that
$$\sup_{g \in \Gamma} \|(\id \ot \tau)(v_n (1 \ot u_g^*))\| \recht 0 \quad\text{when $n \recht +\infty$.}$$
More intuitively, $h_\Gamma(\cG) = 0$ if and only if $\cG$ contains a sequence whose Fourier coefficients tend to zero uniformly in $g \in \Gamma$.

Recall that \cite[Theorem 3.1]{IPV10} is saying the following: if $\Gamma$ is an icc group and $\pi : L(\Lambda) \recht L(\Gamma)$ is a $*$-isomorphism, then the following two statements are equivalent.
\begin{enumlist}
\item We have $h_\Gamma(\pi(\Lambda)) > 0$.
\item There exists a unitary $w \in \cU(L(\Gamma))$ such that $w \pi( \T \Lambda) w^* = \T \Gamma$.
\end{enumlist}
As mentioned in \cite[Remark 3.2]{IPV10}, it is unclear whether a similar result holds when $\pi : L(\Lambda) \recht L(\Gamma)$ is an (irreducible) embedding rather than an isomorphism. It is even less clear what can be said if $\pi : L(\Lambda) \recht L(\Gamma)^t$ is an embedding into an amplification of $L(\Gamma)$.

When $t \leq 1$ and the embedding $\pi : L(\Lambda) \recht L(\Gamma)^t$ is \emph{weakly mixing}, it was proven in \cite[Theorem 4.1]{KV15} that the condition $h_\Gamma(\pi(\Lambda)) > 0$ is equivalent with $t = 1$ and the embedding being standard as above.

In the main result of this section, we prove that when $\Gamma$ is an icc group in \emph{Ozawa's class ($\cS$)} (see \cite{Oza04} and \cite[Definition 15.1.2]{BO08}), one can handle arbitrary $t > 0$ and one may remove the weak mixing hypothesis. This will be a key tool in our rigidity proofs for embeddings of II$_1$ factors into arbitrary amplifications.

Let $\Gamma$ be a countable group and let $\cG$ be any group. We start by describing the standard/trivial ways to obtain group homomorphisms from $\cG$ to the unitary group of an amplification of $L(\Gamma)$, by combining group homomorphisms $\cG \recht \Gamma$, finite-dimensional unitary representations and finite index considerations. First, whenever $\delta: \cG \recht \Gamma$ is a group homomorphism and $\gamma : \cG \recht \cU(n)$ is a finite dimensional unitary representation, we may consider
$$\pi_{\gamma,\delta} : \cG \recht \cU(M_n(\C) \ot L(\Gamma)) : \pi_{\gamma,\delta}(v) = \gamma(v) \ot u_{\delta(v)} \; .$$
Secondly, if $\cG_0 \subg \cG$ is a finite index subgroup and $\pi_0 : \cG_0 \recht \cU(N)$ is a group homomorphism to the unitary group of a von Neumann algebra $N$, we write $m = [\cG:\cG_0]$, choose representatives $v_1,\ldots,v_m \in \cG$ for the cosets $\cG/\cG_0$ and get a natural induction
$$\pi : \cG \recht \cU(M_m(\C) \ovt N) : \pi(v)_{ij} = \begin{cases} \pi_0(w) &\;\;\text{if $vv_j = v_i w$ for some $w \in \cG_0$,}\\ 0 &\;\;\text{if $vv_j \not \in v_i \cG_0$}.\end{cases}$$

\begin{definition}\label{def.standard-embedding}
Let $\Gamma$ be an icc group and $t > 0$. Let $\cG$ be a group. We say that a homomorphism $\pi : \cG \recht \cU(L(\Gamma)^t)$ is \emph{standard} if $t \in \N$ and if $\pi$ is unitarily conjugate to a finite direct sum of inductions to $\cG$ of homomorphisms of the form $\pi_{\gamma,\delta} : \cG_0 \recht \cU(M_n(\C) \ot L(\Gamma))$ where $\cG_0 \subg \cG$ is a finite index subgroup, $\delta: \cG_0 \recht \Gamma$ is a group homomorphism and  $\gamma: \cG_0 \recht \cU(n)$ is an $n$-dimensional unitary representation.
\end{definition}

\begin{remark}
Definition \ref{def.standard-embedding} of a standard homomorphism of a group to $\cU(L(\Gamma)^t)$ may sound a bit cumbersome, but this is unavoidable since one may take inductions and direct sums of arbitrary homomorphisms. Note that if $\pi : \cG \recht \cU(L(\Gamma)^t)$ is a standard homomorphism, then $t = n \in \N$ and there exists a finite index subgroup $\cG_0 \subg \cG$, a unitary representation $\gamma : \cG_0 \recht \cU(\C^n)$ and a group homomorphism $\delta : \cG_0 \recht \Gamma$ such that the restriction $\pi|_{\cG_0}$ is unitarily conjugate to $\pi_{\gamma,\delta} : \cG_0 \recht \cU(M_n(\C) \ot L(\Gamma))$.
\end{remark}

Clearly, if $\pi : \cG \recht \cU(L(\Gamma)^t)$ is a standard homomorphism, we have $h_\Gamma(\pi(\cG)) > 0$. The main result of this section shows that the converse holds if $\Gamma$ belongs to Ozawa's class ($\cS$) and $\pi(\cG)\dpr$ is not too small. For our applications in this paper, the important advantage of Theorem \ref{thm.standard-embedding} compared to \cite[Theorem 4.1]{KV15} is to make no assumptions at all on the nature of the homomorphism $\pi$.

Since we allow arbitrary amplifications, the formulation of the theorem has to take into account direct sums of homomorphisms. The theorem then says that an arbitrary homomorphism is the direct sum of a standard homomorphism, a homomorphism whose image has zero height and a homomorphism whose image is small (i.e.\ amenable).

\begin{theorem}\label{thm.standard-embedding}
Let $\Gamma$ be a countable group in Ozawa's class ($\cS$). Assume that the centralizer of every element $g \in \Gamma \setminus \{e\}$ is amenable.

If $\cG$ is a group, $t > 0$ and $\pi : \cG \recht \cU(L(\Gamma)^t)$ is a group homomorphism, there exist projections $p_1,p_2,p_3 \in L(\Gamma)^t \cap \pi(\cG)'$ with $p_1 + p_2 + p_3 = 1$ such that
\begin{itemlist}
\item The homomorphism $\cG \ni v \mapsto \pi(v) p_1$ is standard.
\item We have $h_\Gamma(\pi(\cG) p_2) = 0$.
\item We have that $\pi(\cG)\dpr p_3$ is amenable.
\end{itemlist}
\end{theorem}

If $\Gamma$ belongs to class ($\cS$), the centralizer of every infinite subgroup of $\Gamma$ is amenable. So, Theorem \ref{thm.standard-embedding} applies to all torsion free groups in class ($\cS$).

We start by proving the following lemma, generalizing \cite[Theorem 4.1]{KV15} to arbitrary amplifications, with a very similar proof. For every element $g$ in a group $\Gamma$, we denote by $C_\Gamma(g) = \{h \in \Gamma \mid gh = h g\}$ its centralizer.

\begin{lemma}\label{lem.embedding-weakly-mixing}
Let $\Gamma$ be an icc group and $t > 0$. Write $P = L(\Gamma)^t$. Let $\cG \subg \cU(P)$ be a subgroup such that the unitary representation $(\Ad v)_{v \in \cG}$ on $L^2(P) \ominus \C 1$ is weakly mixing. Assume that for all $g \in \Gamma \setminus \{e\}$, we have $\cG\dpr \not\prec L(C_\Gamma(g))$.

If $h_\Gamma(\cG) > 0$, then $t = 1$ and there exists a unitary $W \in L(\Gamma)$ such that $W \cG W^* \subseteq \T \Gamma$.
\end{lemma}
\begin{proof}
Take an integer $n \geq t$, write $M = M_n(\C)$, $Q = L(\Gamma)$ and realize $P = p (M \ot Q) p$, where $p$ is a projection with $(\Tr \ot \tau)(p) = t$. We use in this proof multiple tensor products of $M$ and $Q$. We use the tensor leg numbering notation, where e.g.\ $(a \ot b \ot c \ot d)_{2143} = b \ot a \ot d \ot c$.  When $v \in M \ot Q$, we denote by $v_{ij} \in Q$ the matrix coefficients. When $a \in Q = L(\Gamma)$, we write $(a)_g = \tau(a u_g^*)$ for all $g \in \Gamma$.

By the assumption that $h_\Gamma(\cG) > 0$, there exists a $\delta > 0$ such that
$$\sup_{g \in \Gamma} \sum_{k,l=1}^n |(v_{kl})_g|^2 \geq \delta \quad\text{for all $v \in \cG$.}$$
Denote by $\Delta : L(\Gamma) \recht L(\Gamma) \ovt L(\Gamma) : \Delta(u_g) = u_g \ot u_g$ the comultiplication. Note that for all $v \in \cG \subseteq M \ovt Q$,
\begin{align*}
(\Tr \ot \Tr \ot \tau \ot \tau \ot \tau) &\, \bigl( (v \ot (\id \ot \Delta)(v))_{13245} \, ((\id \ot \Delta)(v^*) \ot v^*)_{13425}\bigr) \\
& = \sum_{i,j,k,l = 1}^n (\tau \ot \tau \ot \tau)\, \bigl( (v_{ij} \ot \Delta(v_{kl})) \, (\Delta(v_{ij}^*) \ot v_{kl}^*) \bigr) \\
& = \sum_{g \in \Gamma} \sum_{i,j,k,l = 1}^n |(v_{ij})_g|^2 \, |(v_{kl})_g|^2 \geq \Bigl(\sup_{g \in \Gamma} \sum_{k,l=1}^n |(v_{kl})_g|^2\Bigr)^2 \geq \delta^2 > 0 \; \; .
\end{align*}
Defining $X \in M \ovt M \ovt Q \ovt Q \ovt Q$ as the unique element of minimal $\|\,\cdot\,\|_2$ in the closed linear span of
$$\bigl\{ (v \ot (\id \ot \Delta)(v))_{13245} \, ((\id \ot \Delta)(v^*) \ot v^*)_{13425} \bigm| v \in \cG \bigr\} \; ,$$
the estimate above implies that
$$(\Tr \ot \Tr \ot \tau \ot \tau \ot \tau)(X) \geq \delta^2 \; ,$$
so that $X \neq 0$. Write $q = (\id \ot \Delta)(p)$. By construction and by the uniqueness of $X$,
\begin{align*}
& (p \ot q)_{13245}\, X = X = X \, (q \ot p)_{13425} \quad\text{and}\\
& (v \ot (\id \ot \Delta)(v))_{13245}\, X = X \, ((\id \ot \Delta)(v) \ot v)_{13425} \quad\text{for all $v \in \cG$.}
\end{align*}

By our assumption that $\cG\dpr \not\prec L(C_\Gamma(g))$ if $g \in \Gamma \setminus \{e\}$ and by \cite[Proposition 7.2(3)]{IPV10}, it follows that the unitary representation $(\Ad (\id \ot \Delta)(v))_{v \in \cG}$ of $\cG$ on the orthogonal complement of $(\id \ot \Delta)(p (M \ovt Q)p)$ inside $L^2(q(M \ovt Q \ovt Q)q)$ is weakly mixing. We also assumed that the unitary representation $(\Ad v)_{v \in \cG}$ on $L^2(p (M \ovt Q)p) \ominus \C p$ is weakly mixing. Applying $\Delta$, it follows that $(\Ad (\id \ot \Delta)(v))_{v \in \cG}$ is weakly mixing on $L^2(q(M \ovt Q \ovt Q)q) \ominus \C q$. Taking the tensor product with the representation $(\Ad v)_{v \in \cG}$, it follows that the unitary representation
$$(\Ad (v \ot (\id \ot \Delta)(v)))_{v \in \cG} \quad\text{on}\quad L^2(p(M \ovt Q)p \ovt q(M \ovt Q \ovt Q)q) \ominus \C(p \ot q)$$
is weakly mixing. Since the left support of $X$ is invariant under this representation (after reshuffling the tensor factors), we conclude that after multiplying $X$ with a scalar, $X$ is a partial isometry with left support $(p \ot q)_{13245}$. Making the same reasoning on the right, the right support of $X$ then equals $(q \ot p)_{13425}$.

Since the left and right supports match correctly, the element
$$Y \in (M \ovt M \ovt M) \ovt (Q \ovt Q \ovt Q \ovt Q) \quad\text{given by}\quad Y = X_{23567} \, X_{12456}$$
is a partial isometry with left support $(p \ot p \ot q)_{1425367}$ and right support $(q \ot p \ot p)_{1452637}$ and we have
$$(v \ot v \ot (\id \ot \Delta)(v))_{1425367} \, Y = Y \, ((\id \ot \Delta)(v) \ot v \ot v)_{1452637} \quad\text{for all $v \in \cG$.}$$
Consider the Hilbert space
$$\cK = (p \ot p)_{1324} \, \bigl((\C^n \ot \C^n) (\C^n)^* \ot L^2(Q \ovt Q)\bigr) \, q$$
and the unitary representation
$$\zeta : \cG \recht \cU(\cK) : \zeta(v)(T) = (v \ot v)_{1324} \, T \, (\id \ot \Delta)(v^*) \; .$$
We can view $Y$ as a nonzero invariant vector of the unitary representation $\zeta \ot \overline{\zeta}$. Therefore, $\zeta$ is not a weakly mixing representation. We thus find an integer $k \in \N$, an irreducible unitary representation $\rho : \cG \recht \cU(\C^k)$ and a nonzero element
$$Z \in (p \ot p)_{1324} \, \bigl((\C^n \ot \C^n) (\C^k \ot \C^n)^* \ot L^2(Q \ovt Q)\bigr) \, (1 \ot q)$$
satisfying $(v \ot v)_{1324} Z = Z (\rho(v) \ot (\id \ot \Delta)(v))$ for all $v \in \cG$. By weak mixing of both $\Ad v$ and $\Ad (\id \ot \Delta)(v)$, we find that $ZZ^*$ is a multiple of $(p \ot p)_{1324}$ in the matrix algebra over $L^1(Q \ovt Q)$, while $Z^* Z$ is a multiple of $1 \ot q$. We may thus assume that $ZZ^* = (p \ot p)_{1324}$ and $Z^* Z = 1 \ot q$. In particular, $Z(M_k(\C) \ot 1 \ot 1 \ot 1)Z^*$ is a finite dimensional subspace of $(M \ovt M)_{1324}$ that is globally invariant under $\Ad (v \ot v)_{1324}$. It follows that $k = 1$. It then follows that $p \ot p$ and $q$ have the same trace. Hence, $t= 1$ and we may assume that $n=1$.

We are now in precisely the same situation as in the last paragraph of the proof of \cite[Theorem 4.1]{KV15}. Invoking \cite[Theorem 3,3]{IPV10}, we find a unitary $W \in L(\Gamma)$ such that $W \cG W^* \subseteq \T \Gamma$.
\end{proof}

\begin{proof}[{Proof of Theorem \ref{thm.standard-embedding}}]
Define $p_3$ as the maximal projection in $L(\Gamma)^t \cap \pi(\cG)'$ such that $\pi(\cG)\dpr p_3$ is amenable. Of course, $p_3$ could be zero. By construction, $\pi(\cG)\dpr (1-p_3)$ has no amenable direct summand. Let $(q_i)_{i \in I}$ be a maximal orthogonal family of projections in $L(\Gamma)^t \cap \pi(\cG)'$ with the properties that $q_i \leq 1-p_3$ and $h_\Gamma(\pi(\cG) q_i) = 0$ for all $i \in I$. Define $p_2 = \sum_{i \in I} q_i$. One checks that $h_\Gamma(\pi(\cG) p_2) = 0$. Write $p_1 = 1-p_2-p_3$. By construction, $\pi(\cG)\dpr p_1$ has no amenable direct summand and for every nonzero projection $p \in \pi(\cG)' \cap p_1 L(\Gamma)^t p_1$, we have that $h_\Gamma(\pi(\cG) p) > 0$.

For the rest of the proof, we may thus assume that $\pi : \cG \recht \cU(L(\Gamma)^t)$ is a group homomorphism with the properties that $\pi(\cG)\dpr$ has no amenable direct summand and that $h_\Gamma(\pi(\cG) p) > 0$ for every nonzero projection $p \in L(\Gamma)^t \cap \pi(\cG)'$. We have to prove that $\pi$ is standard.

Define $A_0 \subseteq L(\Gamma)^t$ as the subset of elements $a \in L(\Gamma)^t$ with the property that $\lspan \{\pi(v) a \pi(v)^* \mid v \in \cG\}$ is finite dimensional. Note that $A_0$ is a unital $*$-subalgebra of $L(\Gamma)^t$. Denote by $A$ the weak closure of $A_0$. Let $z \in \cZ(A)$ be the maximal central projection such that $A z$ is diffuse. We prove that $z = 0$, so that $A$ is discrete.

By construction, $\pi(\cG)$ is a subgroup of the normalizer of $A$ inside $L(\Gamma)^t$. Define $\al : \cG \recht \Aut(A) : \al(v)(a) = \pi(v) a \pi(v)^*$. By construction, the action $\al$ is compact, meaning that the closure of $\al(\cG)$ in $\Aut(A)$ is a compact group. Since $\pi(\cG)\dpr$ has no amenable direct summand, a fortiori $\pi(\cG)\dpr$ is diffuse. We can thus pick a sequence $v_n \in \cG$ such that $\pi(v_n) \recht 0$ weakly. Passing to a subsequence, we may assume that the sequence $(\al(v_n))_n$ is convergent in $\Aut(A)$. We then find among the elements $v_m^{-1} v_n$ a sequence $w_n \in \cG$ such that $\pi(w_n) \recht 0$ weakly and $\|\pi(w_n) a - a \pi(w_n)\|_2 \recht 0$ for every $a \in A$. Since $\Gamma$ belongs to Ozawa's class ($\cS$), by \cite[Section 4]{Oza10}, the II$_1$ factor $L(\Gamma)^t$ is $\omega$-solid, implying that $A$ is amenable. It then follows from \cite[Proposition 3.2]{OP07} that the action $\al$ of $\cG$ on $A$ is weakly compact.

Since $\pi(\cG)$ normalizes $A$, we get that $z$ commutes with $\pi(\cG)$. So, $A z \subseteq (A \cup \pi(\cG))\dpr z$ is a regular and weakly compact inclusion. Since $\Gamma$ belongs to Ozawa's class ($\cS$), if $z \neq 0$, it follows from \cite[Theorem 4.1]{CS11} that $(A \cup \pi(\cG))\dpr z$ is amenable, contradicting the fact that $\pi(\cG)\dpr$ has no amenable direct summand. So we have proven that $z=0$ and that $A$ is discrete.

The center $\cZ(A)$ is atomic and $\al$ restricts to an action of $\cG$ on $\cZ(A)$. We then find nonzero projections $(z_n)_{n \in I}$ in $\cZ(A)$ such that the following holds.
\begin{itemlist}
\item $\sum_{n \in I} z_n = 1$.
\item $\al(v)(z_n) = z_n$ for all $v \in \cG$ and $n \in I$.
\item For every $n \in I$, the algebra $\cZ(A) z_n$ is finite dimensional and the restriction of $\al$ to $\cZ(A) z_n$ has trivial fixed point algebra $\C z_n$.
\end{itemlist}
It suffices to prove that for every $n \in I$, the homomorphism $v \mapsto \pi(v)z_n$ is standard. To prove this, we multiply all data with $z_n$ and may thus assume that $\cZ(A)$ is finite dimensional with $\cZ(A)^\al = \C 1$.

Let $z \in \cZ(A)$ be a minimal projection. Define the finite index subgroup $\cG_0 \subg \cG$ by $\cG_0 = \{v \in \cG \mid \al(v)(z) = z\}$. We find that $1 = \sum_{v \in \cG/\cG_0} \al(v)(z)$ and that $\pi$ is induced from the homomorphism $\pi_0 : \cG_0 \recht L(\Gamma)^t z : \pi_0(v) = \pi(v) z$. Since $\cG_0 \subg \cG$ has finite index, we still have that $h_\Gamma(\pi_0(\cG_0)) > 0$ and that $\pi_0(\cG_0)\dpr$ has no amenable direct summand.

By construction, $Az \cong M_k(\C)$ is a matrix algebra. We may thus realize $L(\Gamma)^t z$ as $M_k(\C) \ot L(\Gamma)^s$ in such a way that $A z$ corresponds to $M_k(\C) \ot 1$. For every $v \in \cG_0$, we have that $\al(v)$ restricts to an automorphism of $M_k(\C)$. We can thus choose unitaries $\gamma(v) \in M_k(\C)$ and $\pi_1(v) \in L(\Gamma)^s$ such that $\pi_0(v) = \gamma(v) \ot \pi_1(v)$. Since the unitaries $\gamma(v)$ are uniquely determined up to a scalar, we find that $\cG_1 := \T \pi_1(\cG_0)$ is a subgroup of $\cU(L(\Gamma)^s)$. It also follows that $h_\Gamma(\cG_1) > 0$ and that $\cG_1\dpr$ has no amenable direct summand. For every $g \in \Gamma \setminus \{e\}$, we have that $C_\Gamma(g)$ is amenable and thus, $\cG_1\dpr \not\prec L(C_\Gamma(g))$.

Write $P = L(\Gamma)^s$. We claim that the unitary representation $(\Ad v)_{v \in \cG_1}$ on $L^2(P) \ominus \C 1$ is weakly mixing. If not, we find a nonzero finite dimensional subspace $B_0 \subseteq P \ominus \C 1$ that is globally invariant under $(\Ad v)_{v \in \cG_1}$. Define $B_1 = M_k(\C) \ot B_0$. Then, $B_1$ is globally invariant under $(\Ad \pi_0(v))_{v \in \cG_0}$. Since $\cG_0 \subg \cG$ has finite index, also
$$B_2 := \lspan\{\al(v)(b) \mid v \in \cG, b \in B_1\}$$
is finite dimensional. By definition, $B_2 \subseteq A$. Hence, $B_1 \subseteq Az$. By construction, $B_1$ is orthogonal to $Az$. This contradiction implies that the unitary representation $(\Ad v)_{v \in \cG_1}$ on $L^2(P) \ominus \C 1$ is weakly mixing.

By Lemma \ref{lem.embedding-weakly-mixing}, $s=1$ and after a unitary conjugacy $\cG_1 \subseteq \T \Gamma$. We have thus realized $L(\Gamma)^t z$ as $M_k(\C) \ot L(\Gamma)$ in such a way that $\pi_0(v) = \gamma(v) \ot u_{\delta(g)}$ for all $v \in \cG_0$. This forces $\gamma: \cG_0 \recht \cU(\C^k)$ to be a unitary representation and $\delta: \cG_0 \recht \Gamma$ to be a group homomorphism. Since $\pi$ is the induction of $\pi_0$, the theorem is proven.
\end{proof}

\section{Embeddings of generalized Bernoulli crossed products}\label{sec.embeddings-bernoulli}

Whenever $(A_0,\tau)$ is a tracial von Neumann algebra and $I$ is a countable set, we denote by $(A_0,\tau)^I$ the tensor product von Neumann algebra $\bigotimes_{i \in I} (A_0,\tau)$. For every $i \in I$, we denote by $\pi_i : (A_0,\tau) \recht (A_0,\tau)^I$ the embedding in the $i$'th tensor factor.

Let $\Gamma$ be a countable group with infinite subgroup $\Gamma_0 \subg \Gamma$ and let $\Gamma_0 \actson^\al (A_0,\tau)$ be a trace preserving action on an amenable tracial von Neumann algebra $(A_0,\tau)$. We build the II$_1$ factor
\begin{equation}\label{eq.our-factors}
M(\Gamma_0,\Gamma,\al) = (A_0,\tau)^\Gamma \rtimes_\sigma (\Gamma_0 \times \Gamma) \quad\text{where}\quad \sigma_{(g,h)}(\pi_k(a)) = \pi_{gkh^{-1}}(\al_g(a))
\end{equation}
for all $g \in \Gamma_0$, $h,k \in \Gamma$ and $a \in A_0$.

We will require that $\Gamma$ belongs to the following well studied class of groups with a rank one behavior.

\begin{definition}\label{def.family-C}
We say that a countable group $\Gamma$ belongs to the family $\cC$ if $\Gamma$ is nonamenable, weakly amenable, in Ozawa's class $(\cS)$ and if every nontrivial element $g \in \Gamma \setminus \{e\}$ has an amenable centralizer.
\end{definition}

Note that all free groups $\F_n$ with $2 \leq n \leq +\infty$ and all free products $\Gamma_1 * \Gamma_2$ of amenable groups with $|\Gamma_1| \geq 2$ and $|\Gamma_2| \geq 3$ belong to the family $\cC$. Since groups $\Gamma$ in the family $\cC$ have amenable centralizers, every nonamenable subgroup $\Gamma' \subg \Gamma$ is relatively icc, meaning that $\{g h g^{-1} \mid g \in \Gamma'\}$ is infinite for every $h \in \Gamma \setminus \{e\}$. In particular, $\Gamma$ is itself icc. Finally note that if $\Gamma \in \cC$, then every nonamenable subgroup $\Gamma' \subg \Gamma$ still belongs to $\cC$.

We now return to the construction in \eqref{eq.our-factors}. In most cases, we consider $\Gamma_0 = \Gamma$ and denote the II$_1$ factor as $M(\Gamma,\al)$. When $\Lambda$ and $\Gamma$ belong to the family $\cC$ and when $\Gamma \actson^\al (A_0,\tau)$ and $\Lambda \actson^\be (B_0,\tau)$ are arbitrary trace preserving actions with nontrivial kernel, we describe all possible embeddings of $M(\Lambda,\be)$ into amplifications of $M(\Gamma,\al)$. In particular, we establish the following result that provides numerous families of II$_1$ factors that cannot be embedded one into the other.

Note that Theorem \ref{thm.main-embedding} is a special case of Theorem \ref{thm.main-thm-embedding-bernoulli} by considering the trivial actions $\Gamma \actson (A_0,\tau)$ and $\Lambda \actson (B_0,\tau)$. We explicitly state this special case as Corollary \ref{cor.special-case} below.

\begin{theorem}\label{thm.main-thm-embedding-bernoulli}
Let $\Lambda$ be a nonamenable icc group and let $\Gamma$ be a group in the family $\cC$ of Definition \ref{def.family-C}. Let $\Gamma \actson^\al (A_0,\tau)$ and $\Lambda \actson^\be (B_0,\tau)$ be trace preserving actions on amenable tracial von Neumann algebras such that $A_0 \neq \C 1 \neq B_0$ and such that $\Ker \be$ is nontrivial.
Then the following two statements are equivalent.
\begin{enumlist}
\item There exists a $t > 0$ such that $M(\Lambda,\be) \emb M(\Gamma,\al)^t$.
\item There exists an injective group homomorphism $\delta : \Lambda \recht \Gamma$ and a trace preserving unital $*$-homomorphism $\psi : (B_0,\tau) \recht (A_0,\tau)$ such that $\psi \circ \beta_g = \al_{\delta(g)} \circ \psi$ for all $g$ belonging to a finite index subgroup $\Lambda_1 \subg \Lambda$.
\end{enumlist}
Moreover, if these statements hold and if $t > 0$, we have $M(\Lambda,\be) \emb M(\Gamma,\al)^t$ if and only if $t$ is an integer that can be written as a sum of integers of the form $[\Lambda:\Lambda_1]$ with $\Lambda_1 \subg \Lambda$ as in 2.
\end{theorem}

Applying Theorem \ref{thm.main-thm-embedding-bernoulli} in the case $t=1$, we get in particular that $M(\Lambda,\be) \emb M(\Gamma,\al)$ if and only if there exists an injective group homomorphism $\delta : \Lambda \recht \Gamma$ and a trace preserving unital $*$-homomorphism $\psi : (B_0,\tau) \recht (A_0,\tau)$ such that $\psi \circ \beta_g = \al_{\delta(g)} \circ \psi$ for all $g \in \Lambda$.

Theorem \ref{thm.main-thm-embedding-bernoulli} applies in particular to the trivial action $\Gamma \actson (A_0,\tau)$. The resulting II$_1$ factor is the generalized Bernoulli crossed product of the left-right action of $\Gamma \times \Gamma$ on $\Gamma$ and base algebra $(A_0,\tau)$~:
$$M(\Gamma,A_0,\tau) = (A_0,\tau)^\Gamma \rtimes (\Gamma \times \Gamma) \; .$$

\begin{corollary}\label{cor.special-case}
Let $\Lambda$ be a nonamenable icc group and let $\Gamma$ be a group in the family $\cC$. Let $(A_0,\tau)$, $(B_0,\tau)$ be nontrivial amenable tracial von Neumann algebras. Then the following statements are equivalent.
\begin{enumlist}
\item There exists a $t > 0$ and an embedding $M(\Lambda,B_0,\tau) \hookrightarrow M(\Gamma,A_0,\tau)^t$.
\item There exists an embedding $M(\Lambda,B_0,\tau) \hookrightarrow M(\Gamma,A_0,\tau)$.
\item There exists an injective group homomorphism $\delta : \Lambda \recht \Gamma$ and a trace preserving unital $*$-homomorphism $\psi : (B_0,\tau) \recht (A_0,\tau)$.
\end{enumlist}
Moreover, if these statements hold, we have $M(\Lambda,B_0,\tau) \hookrightarrow M(\Gamma,A_0,\tau)^t$ if and only if $t \in \N$.
\end{corollary}

\begin{remark}\label{rem.virtually-isom}
We say that the II$_1$ factors $M$ and $N$ are \emph{virtually isomorphic} if $M$ can be embedded as a finite index subfactor of $N^t$ for some $t > 0$. Note that $M$ and $N$ are virtually isomorphic if and only if there exists a Hilbert $M$-$N$-bimodule $\bim{M}{\cH}{N}$ with $\dim_{-N}(\cH) < +\infty$ and $\dim_{M-}(\cH) < +\infty$. Virtual isomorphism is thus an equivalence relation on the class of II$_1$ factors.

Our proof of Theorem \ref{thm.main-thm-embedding-bernoulli} and Corollary \ref{cor.special-case} also implies the following. In the setting of Theorem \ref{thm.main-thm-embedding-bernoulli} and given $t>0$, we have that the following statements are equivalent.
\begin{itemlist}
\item $M(\Lambda,\be) \cong M(\Gamma,\al)^t$.
\item $t=1$ and there exists a group isomorphism $\delta : \Lambda \recht \Gamma$ and a trace preserving unital $*$-isomorphism $\psi : (B_0,\tau) \recht (A_0,\tau)$ such that $\psi \circ \beta_g = \al_{\delta(g)} \circ \psi$ for all $g \in \Lambda$.
\item $M(\Lambda,\be)$ and $M(\Gamma,\al)$ are virtually isomorphic.
\end{itemlist}
In particular, the II$_1$ factors $M(\Gamma,\al)$ have trivial fundamental group. In the setting of Corollary \ref{cor.special-case}, we similarly have that the following statements are equivalent.
\begin{itemlist}
\item $M(\Lambda,B_0,\tau) \cong M(\Gamma,A_0,\tau)^t$.
\item $t=1$, $\Lambda \cong \Gamma$ and there exists a trace preserving isomorphism $(B_0,\tau) \cong (A_0,\tau)$.
\item $M(\Lambda,B_0,\tau)$ and $M(\Gamma,A_0,\tau)$ are virtually isomorphic.
\end{itemlist}%
\end{remark}

Before proving Theorem \ref{thm.main-thm-embedding-bernoulli}, we need a few lemmas. For later use (for instance, in the proof of Theorem \ref{thm.all-embed-trivial} below), we state and prove the first lemma \emph{without} assuming that $\Ker \be$ is nontrivial and by allowing the left acting group $\Lambda_0$ to be a proper subgroup of $\Lambda$.

\begin{lemma}\label{lem.step1}
Let $\Lambda_0 \subg \Lambda$ and $\Gamma_0 \subg \Gamma$ be nonamenable icc groups. Assume that $\Gamma \in \cC$. Let $\Gamma_0 \actson^\al (A_0,\tau)$ and $\Lambda_0 \actson^\be (B_0,\tau)$ be trace preserving actions on amenable tracial von Neumann algebras such that $A_0 \neq \C 1 \neq B_0$. Write $(A,\tau) = (A_0,\tau)^\Gamma$ and $(B,\tau) = (B_0,\tau)^\Lambda$. Let $t > 0$ and let $\theta : M(\Lambda_0,\Lambda,\be) \recht M(\Gamma_0,\Gamma,\al)^t$ be a normal unital $*$-homomorphism.

Then $t \in \N$ and there exist finite index subgroups $\Lambda_1 \subg \Lambda_0$ and $\Lambda_2 \subg \Lambda$ such that the restriction of $\theta$ to $B \rtimes (\Lambda_1 \times \Lambda_2)$ is unitarily conjugate to a finite direct sum of embeddings $\theta_i$ of the form
$$\theta_i : B \rtimes (\Lambda_1 \times \Lambda_2) \recht M_{n_i}(\C) \ot (A \rtimes (\Gamma_0 \times \Gamma)) : \begin{cases} \theta_i(B) \subseteq M_{n_i}(\C) \ot A \; ,\\
 \theta_i(u_{(g,h)}) = \gamma_i(g,h) \ot u_{\pi_i(g,h)} \; ,\end{cases}$$
with $\gamma_i : \Lambda_1  \times \Lambda_2 \recht \cU(\C^{n_i})$ a unitary representation and $\pi_i : \Lambda_1 \times \Lambda_2 \recht \Gamma_0 \times \Gamma$ an injective group homomorphism that is either of the form $\pi_i(g,h) = (\eta_i(g),\delta_i(h))$ or $\pi_i(g,h) = (\delta_i(h),\eta_i(g))$.

Moreover, $\Lambda_1 \subg \Lambda_0$ and $\Lambda_2 \subg \Lambda$ can be chosen such that (crudely) $[\Lambda_0 : \Lambda_1] \, [\Lambda : \Lambda_2] \leq \exp(t)$.
\end{lemma}

\begin{proof}
Write $Q_0 = \theta(L(\Lambda_0 \times \{e\}))$ and $Q = \theta(L(\{e\} \times \Lambda))$. With some abuse of notation, we write $Q_0 \ovt Q = \theta(L(\Lambda_0 \times \Lambda))$.
Similarly write $P_0 = L(\Gamma_0 \times \{e\})$ and $P = L(\{e\} \times \Gamma)$, as well as $P_0 \ovt P = L(\Gamma_0 \times \Gamma)$. We also write
$M = M(\Gamma_0,\Gamma,\al)$.

By \cite[Lemma 5.3]{KV15}, after a unitary conjugacy, we may assume that $Q_0 \ovt Q \subseteq (P_0 \ovt P)^t$. By \cite[Theorem 7]{OP03}, there exists a nonempty countable set $I = I_0 \sqcup I_1$, minimal projections $(p_i)_{i \in I}$ in $(Q_0 \ovt Q)' \cap (P_0 \ovt P)^t$ and $r_i,s_i > 0$ such that $\sum_{i \in I} p_i = 1$ and such that
\begin{equation}\label{eq.my-embeddings-nice}
\begin{split}
&\text{For every $i \in I_0$, after unitary conjugacy, $p_i Q_0 \subseteq P_0^{r_i} \ot 1$ and $p_i Q \subseteq 1 \ot P^{s_i}$.}\\
&\text{For every $i \in I_1$, after unitary conjugacy, $p_i Q_0 \subseteq 1 \ot P^{s_i}$ and $p_i Q \subseteq P_0^{r_i} \ot 1$.}
\end{split}
\end{equation}
Since both $\Lambda_0$ and $\Lambda$ are nonamenable icc groups, we have that $Q_0$ and $Q$ are nonamenable factors. Since $\sum_{i \in I} p_i = 1$, it then follows from \eqref{eq.my-embeddings-nice} that for every nonzero projection $p \in (Q_0 \ovt Q)' \cap M^t$, the subalgebra $p (Q_0 \ovt Q)$ is not amenable relative to $A \rtimes (\Gamma_0 \times \{e\})$ and nonamenable relative to $A \rtimes (\{e\} \times \Gamma)$. Since the normalizer of $\theta(B)$ contains $Q_0 \ovt Q$ and using the standard notation of full intertwining-by-bimodules, it follows from \cite[Theorem 1.4]{PV12} that $\theta(B) \prec_f A \rtimes (\Gamma_0 \times \{e\})$ and $\theta(B) \prec_f A \rtimes (\{e\} \times \Gamma)$. So, $\theta(B) \prec_f A$.

Fix $i \in I_0$. Take a unitary conjugacy such that $p_i \theta(u_{(g,h)}) = v_g \ot w_h$ for all $(g,h) \in \Lambda_0 \times \Lambda$, where $\cG_0 = \{v_g \mid g \in \Lambda_0\}$ is a subgroup of $\cU(P_0^{r_i})$ and $\cG = \{w_h \mid h \in \Lambda\}$ is a subgroup of $\cU(P^{s_i})$. By \cite[Lemma 5.10]{KV15} and using that $\theta(B)$ is diffuse and $\theta(B) \prec_f A$, we find that $h_{\Gamma_0}(\cG_0)>0$ and that $h_{\Gamma}(\cG) > 0$. By Theorem \ref{thm.standard-embedding}, the embeddings $\Lambda_0 \ni g \mapsto v_g$ and $\Lambda \ni h \mapsto w_h$ are standard. In particular, $r_i,s_i \in \N$.

The same reasoning holds for $i \in I_1$. In particular, all the amplifications $r_i,s_i$ are integers and there thus are only finitely many of them. It also follows that $t \in \N$.

We have proven that the restriction of $\theta$ to $L(\Lambda_0 \times \Lambda)$ is a finite direct sum of embeddings arising as the tensor product of standard embeddings $L(\Lambda_0) \hookrightarrow L(\Gamma_0)^{r_i}$ and $L(\Lambda) \hookrightarrow L(\Gamma)^{s_i}$, or the other way around. By the definition of a standard embedding, for each $i$, we find finite index subgroups $\Lambda_{1,i} \subg \Lambda_0$ and $\Lambda_{2,i} \subg \Lambda$ such that the embeddings are induced from embeddings given by group homomorphisms and unitary representations defined on $\Lambda_{1,i}$ and $\Lambda_{2,i}$. By construction,
$$\sum_i [\Lambda_0 : \Lambda_{1,i}] \, [\Lambda : \Lambda_{2,i}] \leq t \; .$$
Define $\Lambda_1 = \bigcap_i \Lambda_{1,i}$ and $\Lambda_2 = \bigcap_i \Lambda_{2,i}$. We have $[\Lambda_0 : \Lambda_1] \, [\Lambda : \Lambda_2] \leq \exp(t)$. We find unitary representations $\gamma_i : \Lambda_1 \times \Lambda_2 \recht \cU(\C^{n_i})$ and group homomorphisms $\pi_i : \Lambda_1 \times \Lambda_2 \recht \Gamma_0 \times \Gamma$ such that the restriction of $\theta$ to $L(\Lambda_1 \times \Lambda_2)$ is given by the direct sum of
\begin{equation}\label{eq.partial-embed}
L(\Lambda_1 \times \Lambda_2) \recht M_{n_i}(\C) \ot L(\Gamma_0 \times \Gamma) : u_{(g,h)} \recht \gamma_i(g,h) \ot u_{\pi_i(g,h)}
\end{equation}
with $i \in \{1,\ldots,n\}$. Moreover, each $\pi_i$ is either of the form $\pi_i(g,h) = (\eta_i(g),\delta_i(h))$ or $\pi_i(g,h) = (\delta_i(h),\eta_i(g))$. Since each of the embeddings in \eqref{eq.partial-embed} is an embedding of II$_1$ factors, each $\pi_i$ has finite kernel. Since the groups $\Lambda_1$ and $\Lambda_2$ are icc, it follows that the homomorphisms $\pi_i$ are injective.

Whenever $\pi_i,\pi_j : \Lambda_1 \times \Lambda_2 \recht \Gamma_0 \times \Gamma$ are conjugate by an element in $\Gamma_0 \times \Gamma$, we can perform a unitary conjugacy and regroup the two direct summands. We may thus assume that for $i \neq j$, the homomorphisms $\pi_i$ and $\pi_j$ are not conjugate.

Denote by $p_i \in M_t(\C) \ot 1$ the projection that corresponds to the direct summand in \eqref{eq.partial-embed}. We claim that for $i \neq j$, we have $p_i \theta(B) p_j = \{0\}$. Since $\pi_i$ and $\pi_j$ are not conjugate and since $\pi_i(\Lambda_1 \times \Lambda_2)$ is relatively icc inside $\Gamma_0 \times \Gamma$, it follows that $\{\pi_i(g,h) (a,b) \pi_j(g,h)^{-1} \mid (g,h) \in \Lambda_1 \times \Lambda_2\}$ is infinite for every $(a,b) \in \Gamma_0 \times \Gamma$. We can then take $(g_n,h_n) \in \Lambda_1 \times \Lambda_2$ such that $\pi_i(g_n,h_n) (a,b) \pi_j(g_n,h_n)^{-1} \recht \infty$ for all $(a,b) \in \Gamma_0 \times \Gamma$.

For every finite subset $S \subset \Gamma_0 \times \Gamma$, denote by
$$P_S : M \recht M : P_S(X) = \sum_{(g,h) \in S} E_A(X u_{(g,h)}^*) \, u_{(g,h)}$$
the orthogonal projection of $M$ onto $\lspan \{A u_{(g,h)} \mid (g,h) \in S\}$. Note that
\begin{equation}\label{eq.interm}
\lim_{n \recht \infty} \bigl\| (\id \ot P_S)(p_i \, \theta(u_{(g_n,h_n)}) \, X \, \theta(u_{(g_n,h_n)})^* p_j)\bigr\|_2 = 0
\end{equation}
for all $X \in M_t(\C) \ot M$ and all finite subsets $S \subset \Gamma_0 \times \Gamma$.

Fix $b \in \cU(B)$ and write $b_n = \si_{(g_n,h_n)}(b)$. Fix $\eps > 0$. Since $\theta(B) \prec_f A$, we can take a finite subset $S \subset \Gamma_0 \times \Gamma$ such that
$$\bigl\| p_i \theta(b_n) p_j - (\id \ot P_S)(p_i \theta(b_n) p_j)\bigr\|_2 < \eps \quad\text{for all $n$.}$$
Note that $p_i \theta(b_n) p_j = p_i \theta(u_{(g_n,h_n)}) \, \theta(b) \, \theta(u_{(g_n,h_n)})^* p_j$. So by \eqref{eq.interm}, we have that
$$\bigl\| (\id \ot P_S)(p_i \theta(b_n) p_j)\bigr\|_2 \recht 0 \; .$$
It follows that $\limsup_n \bigl\| p_i \theta(b_n) p_j \bigr\|_2 \leq \eps$, for all $\eps > 0$, so that $\lim_n \bigl\| p_i \theta(b_n) p_j \bigr\|_2 = 0$. Since
$$\bigl\| p_i \theta(b) p_j \bigr\|_2 = \bigl\| \theta(u_{(g_n,h_n)}) \, p_i \theta(b) p_j \, \theta(u_{(g_n,h_n)})^* \bigr\|_2 = \bigl\| p_i \theta(b_n) p_j\|_2 \; ,$$
it follows that $p_i \theta(b) p_j = 0$ for all $b \in \cU(B)$. So our claim is proven.

By this claim, we get that $\theta(B) \subseteq \bigoplus_{i=1}^n (M_{n_i}(\C) \ot M)$. To conclude the proof of the lemma, we show that $p_i \theta(B) \subseteq M_{n_i}(\C) \ot A$. Since $\pi_i(\Lambda_1 \times \Lambda_2)$ is relatively icc in $\Gamma_0 \times \Gamma$, we can take $(g_n,h_n) \in \Lambda_1 \times \Lambda_2$ such that $\pi_i(g_n,h_n) (a,b) \pi_i(g_n,h_n)^{-1} \recht \infty$ for all $(a,b) \in \Gamma_0 \times \Gamma$ with $(a,b) \neq (e,e)$. It follows that
$$\lim_{n \recht \infty} \bigl\| (\id \ot P_S)(p_i \, \theta(u_{(g_n,h_n)}) \, X \, \theta(u_{(g_n,h_n)})^* p_i)\bigr\|_2 = 0$$
for all $X \in M_t(\C) \ot (M \ominus A)$ and all finite subsets $S \subset \Gamma_0 \times \Gamma$. When $S \subset \Gamma_0 \times \Gamma$ is a finite subset with $(e,e) \in S$ and when $b \in \cU(B)$, we define $a = (\id \ot E_A)(\theta(b))$ and $X = \theta(b) - a$. We again write $b_n = \si_{(g_n,h_n)}(b)$. It follows that
$$(\id \ot P_S)(p_i \, \theta(b_n) \, p_i) = p_i \, \theta(u_{(g_n,h_n)}) \, a \, \theta(u_{(g_n,h_n)})^* \, p_i + (\id \ot P_S)(p_i \, \theta(u_{(g_n,h_n)}) \, X \, \theta(u_{(g_n,h_n)})^* p_i) \; .$$
The first term at the right hand side lies in $M_{n_i}(\C) \ot A$ and the second term tends to zero in $\|\,\cdot\,\|_2$. As in the previous paragraph, the left hand side lies uniformly close in $\|\,\cdot\,\|_2$ to $p_i \, \theta(b_n) \, p_i$. We have thus proven that
$$\lim_{n \recht +\infty} \bigl\| p_i \, \theta(b_n) \, p_i - p_i \, \theta(u_{(g_n,h_n)}) \, a \, \theta(u_{(g_n,h_n)})^* \, p_i \bigr\|_2 = 0 \; .$$
Conjugating with $\theta(u_{(g_n,h_n)})$, it follows that $p_i \theta(b) \pi_i = p_i a p_i \in M_{n_i}(\C) \ot A$. This concludes the proof of the lemma.
\end{proof}

\begin{lemma}\label{lem.step2}
Make the same assumptions as in Lemma \ref{lem.step1}. Assume further that $\Lambda_0 \subg \Lambda$ is relatively icc and that $\Ker \be$ is nontrivial. There then exist finite index subgroups $\Lambda_1 \subg \Lambda_0$ and $\Lambda_2 \subg \Lambda$ such that $\Lambda_1 \subg \Lambda_2$ and such that the restriction of $\theta$ to $B \rtimes (\Lambda_1 \times \Lambda_2)$ is unitarily conjugate to a finite direct sum of embeddings $\theta_i : B \rtimes (\Lambda_1 \times \Lambda_2) \recht M_{n_i}(\C) \ot (A \rtimes (\Gamma_0 \times \Gamma))$ satisfying one of the following conditions labeled \eqref{eq.form1} and \eqref{eq.form2}.
\begin{equation}\label{eq.form1}
\begin{split}
\theta_i(\pi_g(B_0)) \subseteq 1 \ot \pi_{\rho_i(g)}(A_0) &\quad\text{for all $g \in \Lambda$, and}\\
\theta_i(u_{(g,h)}) = \gamma_i(g,h) \ot u_{(\delta_i(g),\delta_i(h))} &\quad\text{for all $(g,h) \in \Lambda_1 \times \Lambda_2$,}
\end{split}
\end{equation}
where $\gamma_i : \Lambda_1 \times \Lambda_2 \recht \cU(\C^{n_i})$ is a unitary representation, $\delta_i : \Lambda_2 \recht \Gamma$ is a injective group homomorphism satisfying $\delta_i(\Lambda_1) \subseteq \Gamma_0$ and $\rho_i : \Lambda \recht \Gamma$ is an injective map satisfying $\rho_i(e) =e$ and $\rho_i(g k h) = \delta_i(g) \rho_i(k) \delta_i(h)$ for all $g \in \Lambda_1$, $k \in \Lambda$, $h \in \Lambda_2$.
\begin{equation}\label{eq.form2}
\begin{split}
\theta_i(\pi_g(B_0)) \subseteq 1 \ot \pi_{\rho_i(g)^{-1}}(A_0) &\quad\text{for all $g \in \Lambda$, and}\\
\theta_i(u_{(g,h)}) = \gamma_i(g,h) \ot u_{(\delta_i(h),\delta_i(g))} &\quad\text{for all $(g,h) \in \Lambda_1 \times \Lambda_2$,}
\end{split}
\end{equation}
where $\gamma_i : \Lambda_1 \times \Lambda_2 \recht \cU(\C^{n_i})$ is a unitary representation, $\delta_i : \Lambda_2 \recht \Gamma_0$ is a injective group homomorphism and $\rho_i : \Lambda \recht \Gamma$ is an injective map satisfying $\rho_i(e) =e$ and $\rho_i(g k h) = \delta_i(g) \rho_i(k) \delta_i(h)$ for all $g \in \Lambda_1$, $k \in \Lambda$, $h \in \Lambda_2$.
\end{lemma}

\begin{proof}
Since $\Ker \be$ is a nontrivial normal subgroup of the icc group $\Lambda_0$, we get that $\Ker \be$ is an infinite normal subgroup of $\Lambda_0$. We use the notation and conclusion of Lemma \ref{lem.step1}. Since the finite index subgroup $\Lambda_1$ of $\Lambda_0$ can be embedded into $\Gamma$ and since $\Gamma$ belongs to $\cC$, the group $\Lambda_0$ has no infinite amenable normal subgroups. It follows that $\Ker \be$ is nonamenable. Replacing $\Lambda_2$ by a smaller finite index subgroup, we may assume that $\Lambda_2$ is normal in $\Lambda$. Then replacing $\Lambda_1$ by $\Lambda_1 \cap \Lambda_2$, we may also assume that $\Lambda_1 \subg \Lambda_2$. Define $\Lambda_3 = \Lambda_1 \cap \Ker \be$. Since $\Lambda_3 \subg \Ker \be$ has finite index, also $\Lambda_3$ is nonamenable.

To prove the lemma, we consider each of the direct summands $\theta_i$ in Lemma \ref{lem.step1} separately and thus drop the index $i$. First assume that $\pi(g,h) = (\eta(g),\delta(h))$ for all $(g,h) \in \Lambda_1 \times \Lambda_2$. We claim that there exists a $k \in \Gamma$ such that $k \delta(h) k^{-1} = \eta(h)$ for all $h \in \Lambda_3$. Assume the contrary. Since $\eta(\Lambda_3)$ is relatively icc in $\Gamma$, it follows that $\{\eta(h) k \delta(h)^{-1} \mid h \in \Lambda_3\}$ is infinite for every $k \in \Gamma$. We can then take $h_n \in \Lambda_3$ such that $\eta(h_n) k \delta(h_n)^{-1} \recht \infty$ for all $k \in \Gamma$. Given the form of $\theta(u_{(g,h)})$, it follows that
$$\theta(u_{(h_n,h_n)}) X \theta(u_{(h_n,h_n)})^* \recht 0  \quad\text{weakly,}$$
for all $X \in M_{n}(\C) \ot (A \ominus \C 1)$. Since $u_{(h_n,h_n)}$ commutes with $\pi_e(B_0)$, it follows that $\theta(\pi_e(B_0)) \subseteq M_{n}(\C) \ot 1$. Conjugating with $u_{(e,g)}$, $g \in \Lambda_2$, it follows that $\theta(B_0^{\Lambda_2}) \subseteq M_{n}(\C) \ot 1$, which is absurd because $B_0^{\Lambda_2}$ is diffuse. So the claim is proven and after a further unitary conjugacy with $1 \ot u_{(e,k)}$, we may assume that $\eta(h) = \delta(h)$ for all $h \in \Lambda_3$. If now $g \in \Lambda_1$, since $\Lambda_3$ is a normal subgroup of $\Lambda_1$, we get that $\eta(g)\delta(g)^{-1}$ commutes with $\eta(h)$ for all $h \in \Lambda_3$. Since $\eta(\Lambda_3)$ is relatively icc in $\Gamma$, it follows that $\eta(g) = \delta(g)$ for all $g \in \Lambda_1$.

The same argument of the previous paragraph now also implies that $\theta(\pi_e(B_0)) \subseteq M_n(\C) \ot \pi_e(A_0)$ and hence $\theta(\pi_g(B_0)) \subseteq M_n(\C) \ot \pi_{\delta(g)}(A_0)$ for all $g \in \Lambda_2$.

Now let $g \in \Lambda$ be arbitrary, not necessarily belonging to $\Lambda_2$. We use that $\Lambda_2 \lhd \Lambda$ is normal, so that $g^{-1}\Lambda_2 g = \Lambda_2$. We have that $\pi_g(B_0)$ commutes with $u_{(h,g^{-1}hg)}$ for all $h \in \Lambda_3$. We claim that there exists a unique element $\rho(g) \in \Gamma$ such that $\delta(h) \, \rho(g) \, \delta(g^{-1} h^{-1} g) = \rho(g)$ for all $h \in \Lambda_3$. The existence of $\rho(g)$ follows because otherwise, the same reasoning as above leads to $\theta(\pi_g(B_0)) \subseteq M_n(\C) \ot 1$ and a contradiction. The uniqueness of $\rho(g)$ follows from the relative icc property of $\delta(\Lambda_3)$ in $\Gamma$. By uniqueness of $\rho(g)$ and because $\Lambda_3$ is normal in $\Lambda_1$, it follows that $\rho(h g k) = \delta(h) \rho(g) \delta(k)$ for all $h \in \Lambda_1$, $g \in \Lambda$ and $k \in \Lambda_2$. By construction, $\rho : \Lambda \recht \Gamma$ extends $\delta : \Lambda_2 \recht \Gamma$, we have $\rho(e) = e$ and $\theta(\pi_g(B_0)) \subseteq M_n(\C) \ot \pi_{\rho(g)}(A_0)$.

We claim that $\rho : \Lambda \recht \Gamma$ is injective. For every $g \in \Lambda$, we have that $\delta(h) \rho(g) = \rho(g) \delta(g^{-1}hg)$ for all $h \in \Lambda_1$. If $g,k \in \Lambda$ and $\rho(g) = \rho(k)$, it follows that $\delta(g^{-1} h g) = \delta(k^{-1} h k)$ for all $h \in \Lambda_1$. Since $\delta: \Lambda_1 \recht \Gamma$ is injective, we conclude that $gk^{-1}$ commutes with $\Lambda_1$. Since $\Lambda_0 \subg \Lambda$ is relatively icc, we get that $g = k$.

For every $k \in \Lambda$, define the von Neumann algebra $D_k \subseteq M_n(\C)$ by
$$D_k = \bigl\{ (\id \ot \om)(\theta(\pi_k(b))) \bigm| \om \in A_*, b \in B_0 \bigr\}\dpr \; .$$
Note that $D_{gkh^{-1}} = \gamma(g,h) D_k \gamma(g,h)^*$ for all $k \in \Lambda$ and $(g,h) \in \Lambda_1 \times \Lambda_2$.

The subalgebras $\theta(\pi_k(B_0))$, $k \in \Lambda$, commute among each other. Since $\theta(\pi_k(B_0)) \subseteq M_n(\C) \ot \pi_{\rho(k)}(A_0)$ and since the map $\rho : \Lambda \recht \Gamma$ is injective, the subalgebras $D_k \subseteq M_n(\C)$, $k \in \Lambda$, also commute among each other. For a given $k \in \Lambda$, we have that $D_{gk} = \gamma(g,e) D_k \gamma(g,e)^* \cong D_k$ for all $g \in \Lambda_1$. Since $M_n(\C)$ has no room for infinitely many commuting subalgebras that are all nonabelian, we conclude that all $D_k$, $k \in \Lambda$, are abelian. So it follows that $D = \bigl(\bigcup_{k \in \Lambda} D_k\bigr)\dpr$ is an abelian von Neumann subalgebra of $M_n(\C)$. Since $D_{gkh^{-1}} = \gamma(g,h) D_k \gamma(g,h)^*$, the unitaries $\gamma(g,h)$ normalize $D$. Since $D$ is finite dimensional and abelian, the subgroups
$$\{g \in \Lambda_1 \mid \gamma(g,e) \in D' \cap M_n(\C)\} \subg \Lambda_1 \quad\text{and}\quad \{h \in \Lambda_2 \mid \gamma(e,h) \in  D' \cap M_n(\C)\} \subg \Lambda_2$$
have finite index. Replacing $\Lambda_1$ and $\Lambda_2$ by these finite index subgroups and reducing the embedding $\theta$ by minimal projections in $D$, the conclusion of the lemma holds.

In the case where $\pi(g,h) = (\delta(h),\eta(g))$, we reason analogously.
\end{proof}

We now turn to the assumptions of Theorem \ref{thm.main-thm-embedding-bernoulli}. Before proving Theorem \ref{thm.main-thm-embedding-bernoulli}, we will actually provide a complete description of all possible embeddings $M(\Lambda,\be) \hookrightarrow M(\Gamma,\al)^t$, which is of course of independent interest.

To formulate the result, we first describe the canonical irreducible embeddings. Let $G \subg \Lambda \times \Lambda$ be a finite index subgroup, $\gamma : G \recht \cU(\C^n)$ an irreducible unitary representation and $\delta : \Lambda \recht \Gamma$ an injective group homomorphism. For every $k \in \Lambda$, let $\psi_k : (B_0,\tau) \recht (A_0,\tau)$ be a unital trace preserving $*$-homomorphism such that
\begin{equation}\label{eq.my-equivariance}
\psi_{gkh^{-1}} = \al_{\delta(g)} \circ \psi_k \circ \beta_g^{-1} \quad\text{for all $(g,h) \in G$ and $k \in \Lambda$.}
\end{equation}
Define $\theta_0 : B \rtimes G \recht M_n(\C) \ot (A \rtimes (\Gamma \times \Gamma))$ as the unique normal $*$-homomorphism satisfying
$$\theta_0(\pi_k(b)) = 1 \ot \pi_{\delta(k)}(\psi_k(b)) \quad , \quad \theta(u_{(g,h)}) = \gamma(g,h) \ot u_{(\delta(g),\delta(h))}$$
for all $(g,h) \in G$, $k \in \Lambda$, $b \in B_0$. Then, $\theta_0$ is an irreducible embedding and we define $\theta : M(\Lambda,\be) \hookrightarrow M(\Gamma,\al)^{mn}$ with $m=[\Lambda \times \Lambda: G]$ as the induction of $\theta_0$. Then also $\theta$ is irreducible.

Finally, denote by $\zeta$ the flip automorphism of $M(\Lambda,\be)$ given by $\zeta \circ \pi_k = \pi_{k^{-1}} \circ \be_k^{-1}$ and $\zeta(u_{(g,h)}) = u_{(h,g)}$ for all $g,h,k \in \Lambda$.

\begin{proposition}\label{prop.all-embeddings}
Under the same assumptions as in Theorem \ref{thm.main-thm-embedding-bernoulli}, each embedding $M(\Lambda,\be) \hookrightarrow M(\Gamma,\al)^t$ is a finite direct sum of irreducible embeddings and each irreducible embedding is unitarily conjugate to either $\theta$ or $\theta \circ \zeta$, with $\theta$ and $\zeta$ being as above.
\end{proposition}

\begin{proof}
Let $\theta : M(\Lambda,\be) \recht M(\Gamma,\al)^t$ be a unital normal $*$-homomorphism. We apply Lemma \ref{lem.step2}. Since $\Lambda_0 = \Lambda$, we may assume that $\Lambda_1 = \Lambda_2$ and that $\Lambda_1$ is a finite index normal subgroup of $\Lambda$. We claim that the injective maps $\rho : \Lambda \recht \Gamma$ appearing in \eqref{eq.form1} and \eqref{eq.form2} automatically are group homomorphisms. Since $\Lambda_1$ is a normal subgroup of $\Lambda$, we have that $\delta(g) \rho(k) = \rho(k) \delta(k^{-1}g k)$ for all $g \in \Lambda_1$ and $k \in \Lambda$. Applying this equality twice, we get that
$$\delta(g) \, \rho(k) \rho(h) = \rho(k) \rho(h) \, \delta((kh)^{-1} g kh))$$
for all $g \in \Lambda_1$ and $k,h \in \Lambda$. We also have that $\delta(g) \, \rho(kh) = \rho(kh) \, \delta((kh)^{-1} g kh))$. It follows that $\rho(k)\rho(h)\rho(kh)^{-1}$ commutes with $\delta(\Lambda_1)$. Since $\delta(\Lambda_1) \subg \Gamma$ is relatively icc, we find that $\rho(k) \rho(h) = \rho(kh)$ for all $k,h \in \Lambda$. So, $\rho: \Lambda \recht \Gamma$ is an injective group homomorphism.

We can thus reformulate the conclusion of Lemma \ref{lem.step2} in the following way. We have $t \in \N$ and we find a finite index normal subgroup $\Lambda_1 \lhd \Lambda$, an $n \in \N$ and for all $i \in \{1,\ldots,n\}$, an injective group homomorphism $\delta_i : \Lambda \recht \Gamma$, a unitary representation $\gamma_i : \Lambda_1 \times \Lambda_1 \recht \cU(\C^{n_i})$ and an embedding
$$\theta_i : B \rtimes (\Lambda_1 \times \Lambda_1) \recht M_{n_i}(\C) \ot (A \rtimes (\Gamma \times \Gamma))$$
satisfying
$$\theta_i(\pi_k(B_0)) \subseteq 1 \ot \pi_{\delta_i(k)}(A_0) \quad\text{and}\quad \theta_i(u_{(g,h)}) = \gamma_i(g,h) \ot u_{(\delta_i(g),\delta_i(h))}$$
for all $g,h \in \Lambda_1$, $k \in \Lambda$, $b \in B_0$, such that after a unitary conjugacy, the restriction of $\theta$ to $B \rtimes (\Lambda_1 \times \Lambda_1)$ is given by
$$B \rtimes (\Lambda_1 \times \Lambda_1) \recht M_d(\C) \ot (A \rtimes (\Gamma \times \Gamma)) : \theta|_{B \rtimes (\Lambda_1 \times \Lambda_1)} = \bigoplus_{i=1}^m \theta_i \oplus \bigoplus_{i=m+1}^n (\theta_i \circ \zeta) \; ,$$
with $t = d = \sum_{i=1}^n n_i$. After a further unitary conjugacy, we may assume that for all $1 \leq i,j \leq m$ and for all $m+1 \leq i,j \leq n$, either $\delta_i = \delta_j$ or $\delta_i$ and $\delta_j$ are not conjugate. Denote by $p_i \in M_d(\C) \ot 1$ the projection onto the $i$'th direct summand.

Let $1 \leq i < j \leq m$ such that $\delta_i$ and $\delta_j$ are not conjugate. Take $(a,b) \in \Lambda \times \Lambda$. The element
$$X = p_i \, \theta(u_{(a,b)}) \, p_j \, (1 \ot u_{(\delta_j(a),\delta_j(b))}^*)$$
satisfies
$$\bigl( \gamma_i(g,h) \ot u_{(\delta_i(g),\delta_i(h))} \bigr) \, X = X \, \bigl( \gamma_j(a^{-1}ga,b^{-1}hb) \ot u_{(\delta_j(g),\delta_j(h))} \bigr)$$
for all $(g,h) \in \Lambda_1 \times \Lambda_1$. Since $\delta_i$ and $\delta_j$ are not conjugate and since $\delta_i(\Lambda_1) \subg \Gamma$ is relatively icc, there exists a sequence $g_n \in \Lambda_1$ such that $\delta_i(g_n) k \delta_j(g_n)^{-1} \recht \infty$ for all $k \in \Gamma$. It follows that $X = 0$. Denoting the entire crossed product as $M = B \rtimes (\Lambda \times \Lambda)$, we conclude that $p_i \theta(M) p_j = \{0\}$. A similar reasoning holds for $m+1 \leq i < j \leq n$ with $\delta_i,\delta_j$ not conjugate, and also for all $1 \leq i \leq m$ and $m+1 \leq j \leq n$.

This means that the original embedding can be decomposed as a direct sum of embeddings and for each of these direct summands, in the description above, we have that all $\delta_i$ are equal and that either $m=n$, or $m=0$. For the rest of the proof, we may consider each of these direct summands separately. Composing the embedding with the flip $\zeta$ if necessary, we may thus assume that all $\delta_i$ are equal and that $m = n$. Write $\delta = \delta_i$. This means that the restriction of $\theta$ to $B \rtimes (\Lambda_1 \times \Lambda_1)$ satisfies
$$\theta(\pi_k(B_0)) \subseteq D_n(\C) \ot \pi_{\delta(k)}(A_0) \quad\text{and}\quad \theta(u_{(g,h)}) = \gamma(g,h) \ot u_{(\delta(g),\delta(h))}$$
for all $k \in \Lambda$, $g,h \in \Lambda_1$, where $D_n(\C) \subset M_n(\C)$ is the subalgebra of diagonal matrices and $\gamma : \Lambda_1 \times \Lambda_1 \recht \cU(\C^n)$ is a unitary representation.

The same argument as above then forces that for all $(a,b) \in \Lambda \times \Lambda$, the unitary $\theta(u_{(a,b)})$ is of the form $\gamma(a,b) \ot u_{(\delta(a),\delta(b))}$ for some $\gamma(a,b) \in \cU(\C^n)$.

Define $D \subseteq D_n(\C)$ as the smallest von Neumann subalgebra satisfying $\theta(\pi_k(B_0)) \subseteq D \ot \pi_{\delta(k)}(A_0)$ for all $k \in \Lambda$. It follows that $\gamma(a,b) D \gamma(a,b)^* = D$ for all $(a,b) \in \Lambda \times \Lambda$. Whenever $p \in D$ is a projection that commutes with $\gamma(\Lambda \times \Lambda)$, we have that $p \ot 1$ commutes with $\theta(M)$. So after a further decomposition into direct summands, we may assume that $\gamma(\Lambda \times \Lambda)' \cap D = \C 1$. Taking a minimal projection $p_0 \in D$ and defining the finite index subgroup $G \subg \Lambda \times \Lambda$ by
$$G = \{(a,b) \in \Lambda \times \Lambda \mid \gamma(a,b) p_0 \gamma(a,b)^* = p_0\} \; ,$$
this means that $\theta$ is induced from the embedding
$$\theta_0 : B \rtimes G \recht p_0 M_n(\C) p_0 \ot (A \rtimes (\Gamma \times \Gamma)) : \theta_0(x) = \theta(x) (p_0 \ot 1)\; .$$
Since $p_0$ is a minimal projection in $D$, we find that $\theta_0(\pi_k(b)) = p_0 \ot \pi_{\delta(k)}(\psi_k(b))$, where $\psi_k : (B_0,\tau) \recht (A_0,\tau)$ is a trace preserving unital $*$-homomorphism. So, $\theta_0$ is precisely of the form as described before the proposition. Write $\gamma_0 : G \recht \cU(p_0 M_n(\C) p_0) : \gamma_0(a,b) = \gamma(a,b)p_0$. To conclude the proof of the proposition, it suffices to decompose $\gamma_0$ as a direct sum of irreducible representations.
\end{proof}

We are finally ready to prove Theorem \ref{thm.main-thm-embedding-bernoulli}.

\begin{proof}[{Proof of Theorem \ref{thm.main-thm-embedding-bernoulli}}]
Assume that the first statement holds and let $\theta : M(\Lambda,\be) \hookrightarrow M(\Gamma,\al)^t$ be an embedding. By Proposition \ref{prop.all-embeddings}, $\theta$ is a finite direct sum of irreducible embeddings and each irreducible embedding has a concrete description given before Proposition \ref{prop.all-embeddings}. In particular, we find for each of these irreducible embeddings a finite index subgroup $G \subg \Lambda \times \Lambda$, an injective group homomorphism $\delta : \Lambda \recht \Gamma$ and a family $(\psi_k)_{k \in \Lambda}$ of trace preserving unital $*$-homomorphisms $\psi_k : (B_0,\tau) \recht (A_0,\tau)$ such that \eqref{eq.my-equivariance} holds. Note that $t$ equals a sum of multiples of $[\Lambda \times \Lambda : G]$.

Denote by $\Delta : \Lambda \recht \Lambda \times \Lambda$ the diagonal embedding. Defining the finite index subgroup $\Lambda_1 \subg \Lambda$ such that $\Delta(\Lambda_1) = G \cap \Delta(\Lambda)$, it follows from \eqref{eq.my-equivariance} that $\al_{\delta(g)} \circ \psi_e = \psi_e \circ \beta_g$ for all $g \in \Lambda_1$. So, the second statement of the theorem holds.

In view of the ``moreover'' statement in the theorem, consider the left action of $\Delta(\Lambda)$ on the coset space $(\Lambda \times \Lambda)/G$. Let $(g_j,h_j)G$ be representatives for the orbits. Define the finite index subgroups $\Lambda_j \subg \Lambda$ such that $\Delta(\Lambda_j)$ is the stabilizer of $(g_j,h_j)G$. Note that we could take $(g_1,h_1) = (e,e)$ and then $\Lambda_1$ equals the subgroup $\Lambda_1$ in the previous paragraph. By definition,
$$\Lambda_j = \bigl\{k \in \Lambda \bigm| (g_j^{-1} k g_j , h_j^{-1} k h_j) \in G \bigr\} \; .$$
It thus follows from \eqref{eq.my-equivariance} that
$$\al_{\delta(g_j^{-1} k g_j)} \circ \psi_{g_j^{-1} h_j} = \psi_{g_j^{-1} h_j} \circ \beta_{g_j^{-1} k g_j} \quad\text{for all $k \in \Lambda_j$.}$$
Defining $\psi_j = \al_{\delta(g_j)} \circ \psi_{g_j^{-1} h_j} \circ \beta_{g_j^{-1}}$, this means that $\al_{\delta(k)} \circ \psi_j = \psi_j \circ \beta_k$ for all $k \in \Lambda_j$. Since
$$[\Lambda \times \Lambda : G] = \bigl| (\Lambda \times \Lambda)/G \bigr| = \sum_j [\Lambda : \Lambda_j] \; ,$$
we have indeed written $t$ as a sum of indices $[\Lambda : \Lambda_1]$, where $\Lambda_1 \subg \Lambda$ is a finite index subgroup as in the second statement of the theorem.

Conversely, whenever the second statement of the theorem holds, we have a canonical embedding
$$(B_0,\tau)^\Lambda \rtimes (\Lambda_1 \times \Lambda) \hookrightarrow (A_0,\tau)^\Gamma \rtimes (\Gamma \times \Gamma) \; ,$$
which can be induced to an embedding $M(\Lambda,\be) \emb M(\Gamma,\al)^n$ with $n = [\Lambda : \Lambda_1]$.
\end{proof}

Combining Corollary \ref{cor.special-case} and Remark \ref{rem.virtually-isom}, it is easy to prove Corollary \ref{cor.family-with-a}.

\label{proof.corollary.B}
\begin{proof}[{Proof of Corollary \ref{cor.family-with-a}}]
1.\ It suffices to note that there exists a trace preserving unital embedding $(A_a,\tau_a) \emb (A_b,\tau_b)$ iff $a = b$.

2.\ Since an embedding $B_a \emb B_b$ must map the diffuse part of $B_a$ into the diffuse part of $B_b$, there exists a trace preserving unital embedding $(B_a,\tau_a) \emb (B_b,\tau_b)$ iff $a \leq b$.

3.\ Fix $a \in (0,1/2]$ and choose a projection $p_a \in R$ with $\tau(p_a) = a$. Choose isomorphisms $\theta_1 : R \to p_a R p_a$ and $\theta_2 : R \to (1-p_a)R(1-p_a)$. Then, $\theta(x) = \theta_1(x) + \theta_2(x)$ defines a trace preserving unital embedding $(R \oplus R,\tau_a) \emb (R,\tau)$. Therefore, $Q_a \emb M(R,\tau)$. The map $x \mapsto x \oplus x$ is a trace preserving unital embedding of $(R,\tau)$ into $(R \oplus R,\tau_a)$. Thus, $M(R,\tau) \emb Q_a$.

By Remark \ref{rem.virtually-isom}, we have $Q_a \scong Q_b$ iff $Q_a \cong Q_b$ iff $Q_a,Q_b$ are virtually isomorphic iff there exists a trace preserving isomorphism $(R \oplus R,\tau_a) \cong (R \oplus R,\tau_b)$. This is equivalent with $a = b$ since such a trace preserving isomorphism restricts to a trace preserving isomorphism between the two-dimensional centers and $a,b \in (0,1/2]$.
\end{proof}

As mentioned above, we are formulating Lemmas \ref{lem.step1} and \ref{lem.step2} in a more general context, with possibly proper subgroups $\Lambda_0 \subg \Lambda$ and $\Gamma_0 \subg \Gamma$. This will be crucial to prove Theorem \ref{thm.all-embed-trivial}, where the canonical flip automorphism $\zeta$ has to be avoided. To make the picture complete, we also describe when such asymmetric crossed products embed one into the other. Compared to the formulation of Theorem \ref{thm.main-thm-embedding-bernoulli}, the formulation of Theorem \ref{thm.embedding-bernoulli-distinct-left-right} may sound a bit cumbersome, but this is unavoidable. In Remark \ref{rem.finite-index-subtleness}, we discuss the finite index subtleties that may arise.

\begin{theorem}\label{thm.embedding-bernoulli-distinct-left-right}
Let $\Lambda_0 \subg \Lambda$ and $\Gamma_0 \subg \Gamma$ be nonamenable groups such that $\Lambda_0 \subg \Lambda$ is relatively icc and $\Gamma \in \cC$. Let $\Gamma \actson^\al (A_0,\tau)$ and $\Lambda \actson^\be (B_0,\tau)$ be trace preserving actions on amenable tracial von Neumann algebras such that $A_0 \neq \C 1 \neq B_0$ and such that $\Ker \be$ is nontrivial. Then the following statements are equivalent.
\begin{enumlist}
\item There exists a $t > 0$ and an embedding $M(\Lambda_0,\Lambda,\be) \hookrightarrow M(\Gamma_0,\Gamma,\al)^t$.
\item There exist finite index subgroups $\Lambda_1 \subg \Lambda_0$ and $\Lambda_2 \subg \Lambda$, an injective group homomorphism $\delta : \Lambda_2 \recht \Gamma$, an injective map $\rho : \Lambda \recht \Gamma$ and a unital trace preserving $*$-homomorphism $\psi : (B_0,\tau) \recht (A_0,\tau)$ such that
    \begin{equation}\label{eq.data}
    \Lambda_1 \subseteq \Lambda_2 \quad , \quad \delta(\Lambda_1) \subseteq \Gamma_0 \quad , \quad \rho(gkh) = \delta(g) \rho(k) \delta(h) \quad\text{and}\quad \psi \circ \be_g = \al_{\delta(g)} \circ \psi
    \end{equation}
    for all $g \in \Lambda_1$, $k \in \Lambda$, $h \in \Lambda_2$.
\end{enumlist}
Moreover, if an embedding as in 1 exists, then $t \in \N$.
\end{theorem}

\begin{proof}
1 $\Rightarrow$ 2. By Lemma \ref{lem.step2}, we find that $t \in \N$ and we find finite index subgroups $\Lambda_1 \subg \Lambda_0$ and $\Lambda_2 \subg \Lambda$ with $\Lambda_1 \subseteq \Lambda_2$, and an embedding $\theta$ of $B \rtimes (\Lambda_1 \times \Lambda_2)$ into $M_n(\C) \ot (A \rtimes (\Gamma_0 \times \Gamma))$ that is either of the form \eqref{eq.form1} or \eqref{eq.form2}. In both cases, we define $\psi : (B_0,\tau) \recht (A_0,\tau)$ such that $\theta(\pi_e(b)) = 1 \ot \pi_e(\psi(b))$ for all $b \in B_0$ and check that 2 holds.

2 $\Rightarrow$ 1. Assume that we are given the data of \eqref{eq.data}. There then is a unique normal unital $*$-homomorphism $\theta : B \rtimes (\Lambda_1 \times \Lambda_2) \recht A \rtimes (\Gamma_0 \times \Gamma)$ satisfying
$$\theta(\pi_k(b)) = \pi_{\rho(k)}(\psi(b)) \quad\text{and}\quad \theta(u_{(g,h)}) = u_{(\delta(g),\delta(h))} \quad\text{for all $k \in \Lambda,g \in \Lambda_1, h \in \Lambda_2$ and $b \in B_0$.}$$
This embedding can now be induced to an embedding of $M(\Lambda_0,\Lambda,\be)$ into $M(\Gamma_0,\Gamma,\al)^n$ with $n = [\Lambda_0:\Lambda_1] \, [\Lambda:\Lambda_2]$.
\end{proof}

\begin{remark}\label{rem.finite-index-subtleness}
In light of Theorem \ref{thm.main-thm-embedding-bernoulli}, it is tempting to try to formulate \eqref{eq.data} better, in terms of an injective group homomorphism $\delta : \Lambda \recht \Gamma$. This is however impossible, as the following example illustrates. Take $\Lambda_0 = \F_2$, freely generated by elements $a$ and $b$. Then put $\Lambda = \Lambda_0 * \Z/2\Z$, freely generated by $\Lambda_0$ and an element $c$ of order $2$. Let $(A_0,\tau)$ be any nontrivial tracial amenable von Neumann algebra, e.g.\ $A_0 = \C^2$ with $\tau(1,0) = 1/2$.

Then the left-right Bernoulli crossed product $M(\Lambda_0,\Lambda,A_0) = (A_0,\tau)^\Lambda \rtimes (\Lambda_0 \times \Lambda)$ embeds into an amplification of $M(\F_2,A_0) = (A_0,\tau)^{\F_2} \rtimes (\F_2 \times \F_2)$, although there is no injective group homomorphism $\Lambda \hookrightarrow \F_2$ because $\F_2$ is torsion free. The corresponding map $\eta : \Lambda \recht \F_2$ can be constructed as follows. Let $\F_2$ be freely generated by elements $u,v$. Define the index $2$ subgroup $\Lambda_2 \subg \Lambda$ by $\Lambda_2 = \langle \Lambda_0, c \Lambda_0 c \rangle$. Define the injective group homomorphism $\delta : \Lambda_2 \recht \F_2$ by
$$\delta(a) = u \;\; , \;\; \delta(b) = v u v^{-1} \;\;,\;\; \delta(cac) = v^2 u v^{-2} \;\; , \;\; \delta(cbc) = v^3 u v^{-3} \; .$$
Then define the injective map $\eta : \Lambda \recht \F_2$ by $\eta(g) = \delta(g)$  and $\eta(c g) = v^{-2} \delta(g)$ for all $g \in \Lambda_2$. One checks that $\eta(gkh) = \delta(g)\eta(k)\delta(h)$ for all $g \in \Lambda_0$, $k \in \Lambda$ and $h \in \Lambda_2$.
\end{remark}

\begin{remark}
The W$^*$-rigidity paradigms for embeddings of II$_1$ factors introduced in this article have natural analogues in measurable group theory, for countable probability measure preserving (pmp) equivalence relations.

Already as such, most of the results in this paper have an immediate orbit equivalence counterpart. Given countable pmp equivalence relations
$\cR$ on $(X,\mu)$ and $\cS$ on $(Y,\nu)$, we write $\cR \emb \cS$ if $\cR$ admits an extension that is isomorphic with a subequivalence relation of $\cS$ (cf.\ \cite[Definition 1.4.2]{Pop05}, \cite[Definition 1.6]{Fur06} and \cite[Definition 2.7]{AP15}). In terms of the associated Cartan inclusions $A = L^\infty(X) \subseteq L(\cR) = M$ and $B = L^\infty(Y) \subseteq L(\cS) = N$, this amounts to the existence of an embedding $\theta : M \recht N$ satisfying $\theta(A) \subseteq B$ and $\theta(\cN_M(A)) \subseteq \cN_N(B)$.
When $\cS$ is ergodic, we similarly define the stable embedding relation $\cR \semb \cS$.

The generic construction \eqref{eq.our-factors} has the following analogue on the level of equivalence relations. Given $\Gamma_0 \subg \Gamma$, a countable pmp equivalence relation $\cR_0$ on $(X_0,\mu_0)$ and a measure preserving action $\al : \Gamma_0 \recht \Aut(\cR_0)$, we define the countable pmp equivalence relation $\cR$ on $(X,\mu) = (X_0,\mu_0)^\Gamma$ by $(x,y) \in \cR$ iff there exists $g \in \Gamma_0$ and $h \in \Gamma$ such that $(x_{gkh},\al_g(y_k)) \in \cR_0$ for all $k \in \Gamma$. Then, $L(\cR) \cong M(\Gamma_0,\Gamma,\al)$, where $\Gamma_0 \actson^\al L(\cR_0)$ is the canonical trace preserving action.

We can thus apply Theorem \ref{thm.main-thm-embedding-bernoulli} and provide in this way orbit equivalence variants of the main results stated in the introduction.
\end{remark}

\section{One-sided fundamental groups: proof of Theorem \ref{thm.one-sided-fundamental}}

\begin{proof}[{Proof of Theorem \ref{thm.one-sided-fundamental}}]
We use the following properties of the II$_1$ factor $M$ that appears in the initial data. First, by Theorem \ref{thm.main-thm-embedding-bernoulli}, if $M \emb M^t$ for some $t > 0$, then $t \in \N$. Second, $M$ contains commuting nonamenable subfactors $Q_1 \subset M$ and $Q_2 \subset M$ given by the two subgroups $\Gamma \times \{e\}$ and $\{e\} \times \Gamma$, and the subfactor $Q \subset Q_1 \ovt Q_2$ given by the diagonal embedding of $\Gamma$ into $\Gamma \times \Gamma$, such that $M$ is generated by
$$M = \bigl( Q_1 \cup (Q_1' \cap M) \cup (Q' \cap M)\bigr)\dpr \; .$$

For any countably infinite subset $\cH \subset (0,+\infty)$, we write
$$P(\cH) = \bigast_{r \in \cH} M^r \; .$$
So by definition, $P = P(\cF^{-1})$. By \cite[Theorem 1.5]{DR99}, we have $P(\cH)^t \cong P(t \cH)$ for every $t > 0$. By construction, $P(\cH_1) \emb P(\cH_2)$ if $\cH_1 \subseteq \cH_2$. Whenever $t \in \cF$, we have $t \cF \subseteq \cF$ and thus $\cF^{-1} \subseteq t \cF^{-1}$. We conclude that $P \emb P^t$ so that $\cF \subseteq \cFs(P)$.

Conversely, assume that $t \in \cFs(P)$. Since $1 \in \cF$, we have that $M \emb P$ and thus, $M \emb P^t \cong P(t \cF^{-1})$. Fix such an embedding $\psi : M \recht P(t \cF^{-1})$. Combining \cite[Theorem 4.3]{HU15} with the control of relative commutants provided by \cite[Theorem 1.1]{IPP05}, we find a countable family $(p_i)_{i \in I}$ of nonzero projections $p_i \in P(t \cF^{-1}) \cap \psi(Q_1)'$ with $\sum_i p_i = 1$ and distinct $s_i \in \cF$ such that for every $i$, the subfactor $\psi(Q_1) p_i$ can be unitarily conjugated into the corner $q_i M^{t s_i^{-1}} q_i$ of the canonical copy of $M^{t s_i^{-1}}$ in the free product $P(t \cF^{-1})$. We have $\tau(p_i) = \tau(q_i)$ and we denote $r_i = \tau(q_i)$. Since $\sum_i p_i = 1$, we have that $\sum_i r_i = 1$.

If $a \in Q_1' \cap M$ and $i \neq j$, we must have that $p_i \psi(a) p_j = 0$, since otherwise we can create a nonzero conjugacy between isomorphic copies of $Q_1$ in different factors of the free product $P(t \cF^{-1})$. We conclude that $p_i$ commutes with $\psi(Q_2)$. Using again \cite[Theorem 1.1]{IPP05}, it follows that $\psi(Q_1 \ovt Q_2) p_i$ can be unitarily conjugated into the corner $q_i M^{t s_i^{-1}} q_i$. Then reasoning similarly with $\psi(Q)$, we find that $p_i \in \psi(M)' \cap P^t$ and that $\psi(M) p_i$ can be unitarily conjugated into the corner $q_i M^{t s_i^{-1}} q_i$. This means that $M \emb M^{t r_i s_i^{-1}}$ for every $i$. So, $t r_i s_i^{-1} = n_i \in \N$. It follows that $t r_i = n_i s_i \geq 1$ for every $i$. Since $\sum_i r_i = 1$, we conclude that $I$ is a finite set and that $t = \sum_i n_i s_i$ belongs to $\cF$.
\end{proof}

\section{No nontrivial self-embeddings: proof of Theorem \ref{thm.all-embed-trivial}}

Before proving Theorem \ref{thm.all-embed-trivial}, we build a more abstract context and prove a more general result that we will also use in the proof of Theorem \ref{thm.complete-intervals}.

\begin{notation}
When $\Lambda_1,\Lambda_2$ are subgroups of a group $\Gamma$, we write $\Lambda_1 \prec_\Gamma \Lambda_2$ if a finite index subgroup of $\Lambda_1$ can be conjugated into $\Lambda_2$, i.e.\ if there exists a $g \in \Gamma$ such that $\Lambda_1 \cap g \Lambda_2 g^{-1}$ has finite index in $\Lambda_1$.
\end{notation}

Fix a group $\Gamma$ in the class $\cC$ of Definition \ref{def.family-C}. Let $\Lambda_n < \Gamma_0 < \Gamma_1 \subg \Gamma$ be subgroups and make the following assumptions.

\begin{assumptions}\label{my-assum}
\begin{enumlist}
\item\label{one} We have that $\Gamma_0$ is a proper subgroup of $\Gamma_1$ and that $\Gamma_1$ is a normal subgroup of $\Gamma$. We have that all $\Lambda_n$ are nonamenable.
\item\label{two} The groups $\Gamma_0$ and $\Gamma_1$ have no nontrivial finite dimensional unitary representations.
\item\label{three} If $n \neq m$ and $g \in \Gamma$, the group $g \Lambda_n g^{-1} \cap \Lambda_m$ is amenable.
\item\label{four} If $g \in \Gamma \setminus \Lambda_n$, the group $g \Lambda_n g^{-1} \cap \Lambda_n$ is amenable.
\item\label{five} If $i \in \{0,1\}$ and $\delta : \Gamma_i \recht \Gamma$ is an injective group homomorphism such that $\delta(\Lambda_n) \prec_{\Gamma} \Lambda_n$ for all $n$, there exists a $g \in \Gamma$ such that $\delta(k) = g k g^{-1}$ for all $k \in \Gamma_i$.
\end{enumlist}
\end{assumptions}

\begin{example}\label{example-of-assumptions}
To prove Theorem \ref{thm.all-embed-trivial}, we will use $\Gamma = G \ast G$, where $G = A_\infty$ is the group of finite even permutations of $\N$. We take $\Gamma_1 = \Gamma$ and $\Gamma_0 = G_1 \ast G$, where $G_1$ is the subgroup of finite even permutations of $\N$ that fix $1 \in \N$.
We denote by $\cA \subset G_1$ the set of all elements of order $2$ and we denote by $\cB \subset G$ the set of all elements of order $3$. For all $(a,b) \in \cA \times \cB$, define the subgroup $\Lambda_{a,b} = \langle a \rangle \ast \langle b \rangle$ of $\Gamma_0 = G_1 \ast G$.
In Lemma \ref{lem.first-group-lemma}, we prove that (an enumeration of) the groups $\Lambda_{a,b}$ satisfy the assumptions in \ref{my-assum}.

To prove Theorem \ref{thm.complete-intervals}, we will use a similar family of groups $\Lambda_i < \Gamma_0 < \Gamma_1 \subg \Gamma$ that we introduce in Section \ref{sec.lattices-of-II1-factors}.

In all these cases, the assumptions \ref{three}, \ref{four} and \ref{five} in \ref{my-assum} follow from an analysis of reduced words in free product groups (see Lemma \ref{lem.playing-with-words}) combined with the Kurosh theorem.
\end{example}

Given the assumptions in \ref{my-assum}, we choose probability measures $\mu_n$ on the two point set $\{0,1\}$ such that $\mu_n(0)$ takes distinct values in $(0,1/2)$. We let $\Delta : \Gamma_0 \recht \Gamma_0 \times \Gamma_0$ be the diagonal embedding and consider the generalized Bernoulli action
$$\Gamma_0 \times \Gamma \actson (X,\mu) = \prod_{n \in \N} \bigl(\{0,1\},\mu_n\bigr)^{(\Gamma_0 \times \Gamma)/\Delta(\Lambda_n)} \; .$$
We denote by $M = L^\infty(X,\mu) \rtimes (\Gamma_0 \times \Gamma)$ the crossed product and consider the subfactor $M_1 = L^\infty(X,\mu) \rtimes (\Gamma_0 \times \Gamma_1)$. The following is the main technical result.

\begin{lemma}\label{lem.all-embedding-trivial-general}
Let $\Gamma$ be a countable group in the class $\cC$ of Definition \ref{def.family-C}. Let $\Lambda_n < \Gamma_0 < \Gamma_1 \subg \Gamma$ be subgroups satisfying the assumptions in \ref{my-assum}. Define $M_1$ and $M$ as above.

If $d > 0$ and $\theta : M_1 \recht M^d$ is an embedding, then $d \in \N$ and $\theta$ is unitarily conjugate to the trivial embedding $M_1 \recht M_n(\C) \ot M : a \mapsto 1 \ot a$.
\end{lemma}
\begin{proof}
Before starting the actual proof of the lemma, we deduce a few other group theoretic properties from the assumptions in \ref{my-assum}. For convenience, we continue the numbering of \ref{my-assum}.
\begin{enumlist}[start=6]
\item\label{six} If $\delta : \Gamma_0 \recht \Gamma_0$ is an injective group homomorphism and $\delta(\Lambda_n) \prec_{\Gamma_0} \Lambda_n$ for all $n$, then $\delta$ is conjugate to the identity. Indeed, by assumption \ref{five}, we find $g \in \Gamma$ such that $\delta(k) = g k g^{-1}$ for all $k \in \Gamma_0$. We have to prove that $g \in \Gamma_0$. Fix an $n \in \N$. Take $h \in \Gamma_0$ such that a finite index subgroup of $h \delta(\Lambda_n) h^{-1}$ is contained in $\Lambda_n$. So, $hg$ conjugates a finite index subgroup of $\Lambda_n$ into $\Lambda_n$. By assumption \ref{four}, we have that $hg \in \Lambda_n < \Gamma_0$. Thus $g \in \Gamma_0$.
\item\label{seven} There is no injective group homomorphism $\delta : \Gamma_1 \recht \Gamma_0$ such that $\delta(\Lambda_n) \prec_{\Gamma_0} \Lambda_n$ for all $n$. Indeed, by the previous point, the restriction of $\delta$ to $\Gamma_0$ would be an inner automorphism of $\Gamma_0$ and, in particular, surjective. Since $\Gamma_0 < \Gamma_1$ is a proper subgroup, this is in contradiction with the injectivity of $\delta : \Gamma_1 \recht \Gamma_0$.
\end{enumlist}

For each $n \in \N$ and $(g,h) \in \Gamma_0 \times \Gamma$, we denote by $A_n(g,h) \cong \C^2$ the algebra $L^\infty(\{0,1\},\mu_n) \subset L^\infty(X,\mu)$ viewed in coordinate $(g,h)\Delta(\Lambda_n) \in (\Gamma_0 \times \Gamma)/\Delta(\Lambda_n)$. Defining $(A_0,\tau)$ as the von Neumann algebra generated by $A_n(g,g)$ for all $n \in \N$ and $g \in \Gamma_0$, we have the natural action $\Gamma_0 \actson^\al (A_0,\tau)$ implemented by the diagonal embedding $\Delta : \Gamma_0 \recht \Gamma_0 \times \Gamma_0$. By definition, and using the notation \eqref{eq.our-factors}, we have that $M = M(\Gamma_0,\Gamma,\al)$. Writing $M_0 = M(\Gamma_0,\Gamma_1,\al)$, we canonically have $M_0 \subseteq M_1$.

Denote $A = L^\infty(X,\mu)$ and define $B \subseteq A$ as the von Neumann subalgebra generated by $A_n(g,h)$, $n \in \N$, $(g,h) \in \Gamma_0 \times \Gamma_1$. Note that $M_0 = B \rtimes (\Gamma_0 \times \Gamma_1)$, while $M_1 = A \rtimes (\Gamma_0 \times \Gamma_1)$.

Denote by $\theta_0$ the restriction of $\theta$ to $M_0$, so that $\theta_0 : M_0 \recht M^d$ is an embedding that fits into the setting of Lemma \ref{lem.step1}. We prove that $d \in \N$ and that after a unitary conjugacy $\theta_0(u_{(g,h)}) = u_{(g,h)}$ for all $(g,h) \in \Gamma_0 \times \Gamma_1$. Because $\Gamma_0$ and $\Gamma_1$ have no nontrivial finite dimensional unitary representations, it follows from Lemma \ref{lem.step1} that $d \in \N$ and that, after a unitary conjugacy, $\theta_0$ is a direct sum of finitely many embeddings $\theta_i : M_0 \recht M_{d_i}(\C) \ot M$ satisfying
$$\theta_i(B) \subset M_{d_i}(\C) \ot A \quad\text{and}\quad \theta_i(u_{(g,h)}) = 1 \ot u_{\pi_i(g,h)} \quad\text{for all $(g,h) \in \Gamma_0 \times \Gamma_1$,}$$
where $\pi_i : \Gamma_0 \times \Gamma_1 \recht \Gamma_0 \times \Gamma$ is an injective group homomorphism that is either of the form $\pi_i(g,h) = (\eta_i(g),\delta_i(h))$ or of the form $\pi_i(g,h) = (\delta_i(h),\eta_i(g))$.

We analyze each $\theta_i$ separately and thus drop the index $i$. Fix $n \in \N$. We prove that $\pi(\Delta(\Lambda_n)) \prec_{\Gamma_0 \times \Gamma} \Delta(\Lambda_n)$. We first prove that $\pi(\Delta(\Lambda_n)) \prec_{\Gamma_0 \times \Gamma} \Delta(\Lambda_m)$ for some $m \in \N$. Indeed, if this is not the case, since $A_n(e,e) \subset B$ commutes with $L(\Delta(\Lambda_n))$, it would follow that $\theta(A_n(e,e)) \subseteq M_d(\C) \ot 1$ and thus $\theta(A_n(g,h)) \subseteq M_d(\C) \ot 1$ for all $(g,h) \in \Gamma_0 \times \Gamma_1$, which is absurd.

Fix $m \in \N$ such that $\pi(\Delta(\Lambda_n)) \prec_{\Gamma_0 \times \Gamma} \Delta(\Lambda_m)$. We now use assumptions \ref{three} and \ref{four}, and the fact that every nontrivial element of $\Gamma$ has an amenable centralizer. Since $A_n(e,e)$ commutes with $L(\Delta(\Lambda_n))$, it follows that there exists $(g,h) \in \Gamma_0 \times \Gamma$ such that $\theta(A_n(e,e)) \subseteq M_d(\C) \ot A_m(g,h)$. Denote by $\pi_{m,(g,h)} : \C^2 \recht A_m(g,h)$ the canonical isomorphism. Define the von Neumann subalgebra $D \subseteq M_d(\C) \ot \C^2$ such that $\theta(A_n(e,e)) = (\id \ot \pi_{m,(g,h)})(D)$. Let $k \in \Gamma_0 \setminus \{e\}$ be any element. We have $\theta(A_n(k,e)) = (\id \ot \pi_{m,\pi(k,e)(g,h)})(D)$. Note that $\theta(A_n(k,e))$ commutes with $\theta(A_n(e,e))$. Also note that given the special form of $\pi$, we must have $\pi(k,e)(g,h)\Delta(\Lambda_m) \neq (g,h)\Delta(\Lambda_m)$. The commutation thus forces $D \subseteq D_0 \ot \C^2$, where $D_0 \subseteq M_d(\C)$ is an abelian von Neumann subalgebra.

Defining $C_n$ as the von Neumann algebra generated by $A_n(g,h)$, $(g,h) \in \Gamma_0 \times \Gamma_1$ and $N_n = C_n \rtimes (\Gamma_0 \times \Gamma_1)$, we have proven that $\theta(N_n) \subseteq D_0 \ot M$. Let $q \in D_0$ be any minimal projection. Since $N_n$ and $M$ are II$_1$ factors, the map $N_n \recht q \ot M : a \mapsto \theta(a) (q \ot 1)$ is (normalized) trace preserving. Therefore, $a \mapsto \theta(a)(q \ot 1)$ induces a trace preserving embedding of $A_n(e,e)$ into $A_m(g,h)$. When $n \neq m$, we have that $\mu_n(0)$ and $\mu_m(0)$ are distinct elements of $(0,1/2)$ and such an embedding does not exist. If $n = m$, because $\mu_n(0) < 1/2$, such an embedding must be the canonical ``identity'' map. We have thus proven that $n=m$ and that
$$\theta(\pi_{n,(e,e)}(a))(q \ot 1) = q \ot \pi_{n,(g,h)}(a)$$
for all $a \in \C^2$ and all minimal projections $q \in D_0$. Writing $1 \in D_0$ as a sum of minimal projections, we conclude that $\theta(\pi_{n,(e,e)}(a)) = 1 \ot \pi_{n,(g,h)}(a)$ for all $a \in \C^2$.

We now return to the full setting where we had decomposed $\theta_0$ as a direct sum of embeddings $\theta_i$. We have in particular shown that
$$\pi_i(\Delta(\Lambda_n)) \prec_{\Gamma_0 \times \Gamma} \Delta(\Lambda_n)$$
for all $i$ and all $n$. If $\pi_i(g,h) = (\delta_i(h),\eta_i(g))$, we get that $\delta_i : \Gamma_1 \recht \Gamma_0$ is an injective group homomorphism satisfying $\delta_i(\Lambda_n) \prec_{\Gamma_0} \Lambda_n$ for all $n$. By property \ref{seven}, this is impossible. Thus, $\pi_i(g,h) = (\eta_i(g),\delta_i(h))$. We find that $\eta_i : \Gamma_0 \recht \Gamma_0$ is an injective group homomorphism satisfying $\eta_i(\Lambda_n) \prec_{\Gamma_0} \Lambda_n$ for all $n$, while $\delta_i : \Gamma_1 \recht \Gamma$ is an injective group homomorphism satisfying $\delta_i(\Lambda_n) \prec_{\Gamma} \Lambda_n$ for all $n$. Using property \ref{six} and assumption \ref{five}, we get after a unitary conjugacy that $\pi_i(g,h) = (g,h)$ for all $(g,h) \in \Gamma_0 \times \Gamma_1$.

We finally return to the initial embedding $\theta : M_1 \recht M^d$. We have proven that $d \in \N$ and that, after a unitary conjugacy, $\theta(u_{(g,h)}) = 1 \ot u_{(g,h)}$ for all $(g,h) \in \Gamma_0 \times \Gamma_1$. Fix $n \in \N$ and $(a,b) \in \Gamma_0 \times \Gamma$. Since $\Gamma_1 \lhd \Gamma$ is normal, we have that $(a,b)\Delta(\Lambda_n)(a,b)^{-1}$ is a nonamenable subgroup of $\Gamma_0 \times \Gamma_1$. This subgroup acts trivially on $A_n(a,b)$. Analyzing commutants as above, it follows that $\theta(A_n(a,b)) \subseteq M_d(\C) \ot A_n(a,b)$. Making the same reasoning as above, we get that
$$\theta(\pi_{n,(a,b)}(x)) = 1 \ot \pi_{n,(a,b)}(x)$$
for all $x \in \C^2$. So we have shown that $\theta(x) = 1  \ot x$ for all $x \in M_1$.
\end{proof}

We start by proving the following well known result.

\begin{lemma}\label{lem.long-words}
Let $S \subg G \ast K$ be a subgroup of a free product group. If the word length function is bounded on $S$, then $S$ is conjugate to a subgroup of $G$ or $K$.
\end{lemma}
\begin{proof}
We may assume that $S \neq \{e\}$. We distinguish two cases. First assume that $S$ contains an element $s \in S$ that is conjugate to a cyclically reduced word of length at least two, i.e.\ $s = w g w^{-1}$. Then the word length of $s^n = w g^n w^{-1}$ tends to infinity, because $|g^n| = n \, |g|$.

If we are not in the first case, every element of $S$ is conjugate to an element of $G$ or $K$. Conjugating $S$, we may assume that $S$ contains a nontrivial element of $G$ or $K$. By symmetry, we may assume that $S \cap G \neq \{e\}$ and take $s \in S \cap G$, $s \neq e$. If $S \subg G$, the lemma is proven. Otherwise, we find $r \in S \setminus G$. Write $r = w g w^{-1}$ with $g \in G \cup K$, $g \neq e$, and with the concatenation $w g w^{-1}$ being a reduced word. If the first letter of $w$ belongs to $G$, write $w = w_0 w_1$ with $w_0 \in S$ being this first letter. Otherwise, take $w_0=e$ and $w_1 = w$. Consider the element $sr \in S$. Then, $sr = w_0 ((w_0^{-1} s w_0) w_1 g w_1^{-1}) w_0^{-1}$ and $(w_0^{-1} s w_0) w_1 g w_1^{-1}$ is a cyclically reduced word. We are thus back in the first case and the lemma is proven.
\end{proof}

\begin{lemma}\label{lem.playing-with-words}
Let $\Gamma = G \ast K$ be a free product group. Assume that $a_1 \in G \cup K$ is an element of order $2$ in either $G$ or $K$. Assume that $b_1 \in G \cup K$ is an element of order $3$ in either $G$ or $K$. Let $w \in G \ast K$ be such that $a_1$ and $w b_1 w^{-1}$ are free and denote by $\Lambda_1 \subg G \ast K$ the subgroup generated by these two elements.
\begin{enumlist}
\item Let $a \in G$ be of order $2$, $b \in K$ of order $3$ and write $\Lambda = \langle a \rangle \ast \langle b \rangle$.

If $z \in \Gamma \setminus \Lambda$, then $z \Lambda z^{-1} \cap \Lambda$ is finite.

If there exists a $z \in \Gamma$ such that $z \Lambda_1 z^{-1} \cap \Lambda$ is infinite, then $a_1 \in G$, $b_1 \in K$ and $w$ can be written as $w = u^{-1} w_0 v$ where $u \in G$, $v \in K$ and $w_0 \in \Lambda$ such that $u a_1 u^{-1} = a$, $v b_1 v^{-1} = b^{\pm 1}$ and such that $w_0 \in \langle a \rangle \ast \langle b \rangle$ is either equal to $e$ or starting with the letter $b^{\pm 1}$ and ending with the letter $a$.

\item Let $a \in G$ be of order $2$ and $b \in G$ of order $3$. Let $k \in K \setminus \{e\}$ and write $\Lambda = \langle a \rangle \ast k \langle b \rangle k^{-1}$.

If $z \in \Gamma \setminus \Lambda$, then $z \Lambda z^{-1} \cap \Lambda$ is finite.

If there exists a $z \in \Gamma$ such that $z \Lambda_1 z^{-1} \cap \Lambda$ is infinite, then $a_1,b_1 \in G$ and $w$ can be written as $w = u^{-1} w_1 k v$ where $u,v \in G$ and $w_1 \in \Lambda$ such that $u a_1 u^{-1} = a$, $v b_1 v^{-1} = b^{\pm 1}$ and such that $w_1 \in \langle a \rangle \ast \langle kbk^{-1} \rangle$ is either equal to $e$ or starting with the letter $k b^{\pm 1}k^{-1}$ and ending with the letter $a$.
\end{enumlist}
\end{lemma}
\begin{proof}
Throughout the proof of this lemma, we use the following property: if two cyclically reduced words in $G \ast K$ are conjugate as elements of $G \ast K$, then they are cyclic permutations of each other. It follows in particular that in both settings 1 and 2, $z G z^{-1} \cap \Lambda$ and $z K z^{-1} \cap \Lambda$ are finite groups for all $z \in \Gamma$.

For the following reason, we have in both settings 1 and 2 that $z \in \Lambda$ whenever $z \Lambda z^{-1} \cap \Lambda$ is infinite. Indeed, if $z \Lambda z^{-1} \cap \Lambda$ is infinite, by Lemma \ref{lem.long-words}, we can take an arbitrarily long word $x \in \Lambda$ such that $z x z^{-1} \in \Lambda$. When the word length of $x$ is large enough, the middle letters of $x$ remain untouched when reducing $z x z^{-1}$. Since this reduction must belong to $\Lambda$, we align these middle letters with the canonical reduced words defining elements of $\Lambda$ and conclude that $z \in \Lambda$.

To prove the second part of statements 1 and 2, write $L = G$ or $L = K$ so that $a_1 \in L$. Denote by $L'$ the other group. Similarly denote $R = G$ or $R = K$ so that $b_1 \in R$ and $R'$ is the other group. Assume that an infinite subgroup of $\Lambda_1$ can be conjugated into $\Lambda$. The first paragraph of the proof now implies that $\Lambda_1 \not\subseteq G$ and $\Lambda_1 \not\subseteq K$. So if $L = R$, we must have $w \not\in L$.

Uniquely write $w = u^{-1} w_0 v$ with $u \in L$, $v \in R$ and with $w_0$ either equal to $e$ (which can only occur if $L \neq R$), or $w_0$ being a reduced word that starts with a letter in $L' \setminus \{e\}$ and ends with a letter from $R' \setminus \{e\}$. Put $a_2 = u a_1 u^{-1}$ and $b_2 = v b_1 v^{-1}$. So, an infinite subgroup of $\langle a_2 \rangle \ast \langle w_0 b_2 w_0^{-1} \rangle$ can be conjugated into $\Lambda$. By the second paragraph of the proof, we can take an element of the form
\begin{equation}\label{eq.long-word-x}
x = a_2 \, w_0 b_2^{\pm 1} w_0^{-1} \, a_2 \, \cdots \, w_0 b_2^{\pm 1} w_0^{-1}
\end{equation}
having arbitrarily large length such that $x$ can be conjugated into $\Lambda$.

In setting~1, with $\Lambda = \langle a \rangle \ast \langle b \rangle$ and $a \in G$, $b \in K$, it follows that all letters appearing in $x$ must be either $a$ or $b^{\pm 1}$. Since $a_2$, $a$ have order $2$, while $b_2$, $b$ have order $3$, we must have $a_2 = a$, $b_2 = b^{\pm 1}$ and $w_0 \in \Lambda$. In particular, $L = G$ and $R = K$. So all conclusions of 1 indeed hold.

In setting~2, with $\Lambda = \langle a \rangle \ast \langle k b k^{-1} \rangle$ and $a,b \in G$, $k \in K \setminus \{e\}$, we make the following two observations: any cyclically reduced word in $a$ and $kb^{\pm 1} k^{-1}$ only contains the letters $a$, $k^{\pm 1}$, $b^{\pm 1}$ and has the property that between any two occurrences of the letter $a$, there sits a word that belongs to $\Lambda$. We know that the word $x$ in \eqref{eq.long-word-x} can be cyclically permuted to such a cyclically reduced word in $a$ and $kb^{\pm 1} k^{-1}$. So, only the letters $a$, $k^{\pm 1}$, $b^{\pm 1}$ appear in the word $x$. Since $a_2$ has order $2$, while $b$ has order $3$, we must have that $a_2 \in \{a,k, k^{-1}\}$. Similarly, $b_2 \in \{b,b^{-1},k,k^{-1}\}$.

First assume that $a_2 = a$. By the previous paragraph, between any two occurrences of $a_2$ in $x$, there sits a word that belongs to $\Lambda$. It follows that $w_0 b_2 w_0^{-1} \in \Lambda$. The only elements in $\Lambda$ of order $3$ are conjugate (inside $\Lambda$) to $k b^{\pm 1} k^{-1}$. It follows that $b_2 = b^{\pm 1}$ and $w_0 = w_1 k$ with $w_1 \in \Lambda$. So all conclusions of 2 hold in this case.

Next assume that $a_2 \neq a$, which should lead to a contradiction. Then $a_2 \in \{k,k^{-1}\}$. In particular, $k$ has order $2$. Since $b_2 \in \{b,b^{-1},k,k^{-1}\}$ and $b_2$ has order $3$, it follows that $b_2 = b^{\pm 1}$.

So, $k$ has order $2$ and $a_2 = k$. In any cyclically reduced word in $a$ and $kb^{\pm 1} k$, the letter $k$ appears twice as much as the letters $b^{\pm 1}$, and every occurrence of $b^{\pm 1}$ is preceded and followed by $k$. Comparing with the description of $x$ in \eqref{eq.long-word-x}, we must have that $k$ appears in $w_0$ and that thus, $w_0$ ends with the letter $k$. Write $w_0 = w_1 k$. Since $a_2 = k$, we have that $L = K$, so that $w_0$ must start with a letter in $G \setminus \{e\}$. It follows that $w_1 \neq e$ and that $w_1$ starts and ends with one of the letters $a,b,b^{-1}$. We have written the word $x$ in \eqref{eq.long-word-x} as
$$x = (k \, w_1 \, kb^{\pm 1}k \, w_1^{-1})^n \quad\text{for some $n \in \N$.}$$
We know that this word can be cyclically permuted so as to become a cyclically reduced word in $a$ and $kb^{\pm 1} k$. Reading $x$ from left to right, the only possibilities are $w_1 \in b^{\pm 1} k \Lambda$ and $w_1 \in \Lambda$. In the first case, $x$ is a reduced word in $kb^{\pm 1} k$ and $a$ followed by the ``remainder'' $k b^{\mp 1}$. In the second case $kxk$ is a reduced word in $a$ and $kb^{\pm 1} k$ followed by the ``remainder'' $k$. In both cases, we reached the required contradiction.
\end{proof}

\begin{lemma}\label{lem.elem-permutation}
Let $G$ be the group of finite even permutations of $\N$. Denote by $G_1$ the subgroup of finite even permutations of $\N$ that fix $1 \in \N$. Denote by $\cA \subset G_1$ the elements of order $2$ and denote by $\cB \subset G_1$ the elements of order $3$.

Let $\delta$ be a group homomorphism from either $G_1$ or $G$ to $G$. If $\delta(a) = a$ for every $a \in \cA$, then $\delta$ is the identity homomorphism. If $\delta(b) = b^{\pm 1}$ for every $b \in \cB$,  then $\delta$ is the identity homomorphism.
\end{lemma}
\begin{proof}
Denote by $\delta_0$ the restriction of $\delta$ to $G_1$. Since the elements of order $2$ generate $G_1$, if $\delta_0(a) = a$ for all $a \in \cA$, we have that $\delta_0(\si) = \si$ for all $\si \in G_1$.

Next assume that $\delta_0(b) = b^{\pm 1}$ for all $b \in \cB$. Denote by $\cB_0 \subset \cB$ the subset of permutations of the form $(ijk)$ with $2 \leq i < j < k$. We prove that $\delta_0(b) = b$ for all $b \in \cB_0$. Put $b_1 = (234)$. By symmetry, it suffices to prove that $\delta_0(b_1) = b_1$. Assume that this is not the case. Then, $\delta_0(b_1) = b_1^{-1}$. Write $b_2 = (456)$. We have $b_1 b_2 b_1^{-1} = (256)$, while $b_1^{-1} b_2 b_1 = (356)$. It is thus impossible to get $\delta_0(b_2) = b_2^{\pm 1}$. So $\delta_0(b) = b$ for all $b \in \cB_0$. Since $\cB_0$ generates the group $G_1$, we again have that $\delta_0(\si) = \si$ for all $\si \in G_1$.

Every homomorphism $\delta : G \recht G$ whose restriction to $G_1$ is the identity, must itself be the identity. This concludes the proof of the lemma.
\end{proof}

\begin{lemma}\label{lem.first-group-lemma}
When $\Lambda_{a,b} < \Gamma_0 < \Gamma_1 = \Gamma$ are defined as in Example \ref{example-of-assumptions}, then the assumptions in \ref{my-assum} all hold.
\end{lemma}
\begin{proof}
Note that \ref{one} is immediate. For \ref{two}, it suffices to note that $A_\infty$ has no nontrivial finite dimensional unitary representations. This is well known and can be seen as follows: since $A_\infty$ contains infinitely many commuting copies of $A_5$, any finite dimensional representation $\pi$ has the property that $\pi(A_5)$ is abelian and hence trivial, which forces $\pi$ to be trivial on $A_\infty$.

Conditions \ref{three} and \ref{four} follow from the first statement of Lemma \ref{lem.playing-with-words}. To prove \ref{five}, we denote by $G_r$ the second copy of $G$ in the free product $\Gamma = G \ast G_r$. Let $\delta : G_1 \ast G_r \recht G \ast G_r$ be an injective group homomorphism such that $\delta(\Lambda_{a,b}) \prec_\Gamma \Lambda_{a,b}$ for all $(a,b) \in \cA \times \cB$. After replacing $\delta$ by $(\Ad y) \circ \delta$ for some $y \in \Gamma$, by the Kurosh theorem, we can take $L,R \in \{G,G_r\}$, injective group homomorphisms $\delta_1 : G_1 \recht L$ and $\delta_r : G_r \recht R$, and an element $w \in \Gamma$ such that $\delta(g) = \delta_1(g)$ for all $g \in G_1$ and $\delta(h) = w \delta_r(h) w^{-1}$ for all $h \in G_r$. Denote by $R' \in \{G,G_r\}$ the ``other'' group, so that $\{R,R'\} = \{G,G_r\}$. We may assume that either $w = e$ or that $w$ ends with a letter in $R' \setminus \{e\}$.

Given $(a,b) \in \cA \times \cB$, we view $a \in G_1$ and $b \in G_r$. Since $a$ and $b$ are free, also $\delta_1(a)$ and $w \delta_r(b) w^{-1}$ are free. Since $\delta(\Lambda_{a,b}) \prec_\Gamma \Lambda_{a,b}$, we get that a finite index subgroup of $\langle \delta_1(a) \rangle \ast w \langle \delta_r(b) \rangle w^{-1}$ can be conjugated into $\Lambda_{a,b}$.

By the first statement in Lemma \ref{lem.playing-with-words}, we first conclude that $L = G$, $R = G_r$. In particular, either $w = e$ or $w$ ends with a letter in $G \setminus \{e\}$. In the case where $w=e$, Lemma \ref{lem.playing-with-words} says that $\delta_1(a) = a$ and $\delta_r(b) = b^{\pm 1}$ for all $(a,b) \in \cA \times \cB$. By Lemma \ref{lem.elem-permutation}, both $\delta_1$ and $\delta_r$ are the identity map.

In the case where $w \neq e$, we write $w = u_0^{-1} w_0$ where either $w_0 = e$ or $w_0$ is a word starting with $b^{\pm 1}$ and ending with $a$. The latter option can never hold for all $(a,b) \in \cA \times \cB$. In the former case, after replacing $\delta$ by $(\Ad u_0) \circ \delta$, Lemma \ref{lem.playing-with-words} again says that $\delta(a) = a$ and $\delta(b) = b^{\pm 1}$ for all $(a,b) \in \cA \times \cB$. Again, $\delta$ is the identity homomorphism.

So also condition \ref{five} in \ref{my-assum} holds and the lemma is proven.
\end{proof}

\begin{proof}[{Proof of Theorem \ref{thm.all-embed-trivial}}]
The theorem immediately follows from Lemmas \ref{lem.first-group-lemma} and \ref{lem.all-embedding-trivial-general}.
\end{proof}

\section[Partially ordered sets of II$_1$ factors: proof of Thm.\ \ref{thm.family-with-groups} and Cor.\ \ref{cor.examples-orders}]{\boldmath Partially ordered sets of II$_1$ factors: proof of Theorem \ref{thm.family-with-groups} and Corollary \ref{cor.examples-orders}}

Note that the augmentation functor $\Gamma \mapsto H_\Gamma$ is functorial in the following precise sense. To any injective group homomorphism $\Gamma \recht \Lambda$ corresponds a canonical injective group homomorphism $H_\Gamma \recht H_\Lambda$ and thus, a canonical embedding of II$_1$ factors $L(H_\Gamma) \emb L(H_\Lambda)$. Theorem \ref{thm.family-with-groups}, which we now prove, says that if $L(H_\Gamma) \semb L(H_\Lambda)$, then $\Gamma \emb \Lambda$, but does not say that any embedding of $L(H_\Gamma)$ into $L(H_\Lambda)$ is of such a canonical form.

The proof of Theorem \ref{thm.family-with-groups} follows rather directly from Lemmas \ref{lem.step1} and \ref{lem.step2}. We however have to do a few tweaks in the case where the groups $\Gamma$ and $\Lambda$ are uncountable.

\begin{proof}[{Proof of Theorem \ref{thm.family-with-groups}}]
For every infinite group $\Gamma$, we denote by $B_\Gamma = \bigl(\Z/2\Z\bigr)^{((G_\Gamma \times G_\Gamma)/N_\Gamma)}$ the abelian normal subgroup of $H_\Gamma$ such that $H_\Gamma = B_\Gamma \rtimes (G_\Gamma \times G_\Gamma)$. We also write $A_\Gamma = L(B_\Gamma)$.

Note that whenever $\Gamma_0 \subg \Gamma_1$ is a subgroup, we canonically have $H_{\Gamma_0} \subg H_{\Gamma_1}$. Through this inclusion, $A_{\Gamma_0} \subg A_{\Gamma_1}$.

Also note that whenever $\Gamma$ is a countable infinite group, the II$_1$ factor $L(H_\Gamma)$ is of the form studied in Section \ref{sec.embeddings-bernoulli}. Defining
\begin{equation}\label{eq.reinterpret}
D_\Gamma = \bigl(\Z/2\Z\bigr)^{((\Z \ast \Gamma)/\Gamma)} \quad\text{and}\quad G_\Gamma \actson^\al L(D_\Gamma) \quad\text{through $\pi_\Gamma : G_\Gamma \recht \Z \ast \Gamma$,}
\end{equation}
we canonically have $L(H_\Gamma) = M(G_\Gamma,\al)$.

Throughout the proof, we say that a group homomorphism $\pi : G_1 \times G_1 \recht G_2 \times G_2$ is \emph{of product form} if $\pi$ is either of the form $\pi(g,h) = (\eta(g),\delta(h))$ or the form $\pi(g,h) = (\delta(h),\eta(g))$ for group homomorphisms $\delta,\eta : G_1 \recht G_2$.

Let $\Gamma$ and $\Lambda$ be arbitrary infinite groups. Take $d > 0$ and an embedding $\theta : L(H_\Gamma) \recht L(H_\Lambda)^d$. We have to prove that $\Gamma$ is isomorphic with a subgroup of $\Lambda$. Fix a countable infinite subgroup $\Gamma_0 \subg \Gamma$. Since $H_{\Gamma_0}$ is countable, we can take a countable subgroup $\Lambda_0 \subg \Lambda$ such that $\theta(L(H_{\Gamma_0})) \subseteq L(H_{\Lambda_0})^d$. We denote by $\theta_0$ the restriction of $\theta$ to $L(H_{\Gamma_0})$.

As explained above, the embedding $\theta_0 : L(H_{\Gamma_0})) \recht L(H_{\Lambda_0})^d$ fits into the context of Section \ref{sec.embeddings-bernoulli}. By Lemma \ref{lem.step1}, $d \in \N$ and we can unitarily conjugate the original embedding and replace $\Gamma_0$ by a finite index subgroup such that $\theta_0$ becomes the direct sum of embeddings $\theta_i : L(H_{\Gamma_0})) \recht M_{d_i}(\C) \ot L(H_{\Lambda})$ satisfying
\begin{equation}\label{eq.onGamma0}
\theta_i(A_{\Gamma_0}) \subset M_{d_i}(\C) \ot A_{\Lambda} \quad\text{and}\quad \theta_i(u_{(g,h)}) \in \cU(\C^{d_i}) \ot u_{\pi_i(g,h)} \;\;\text{for all $g,h \in G_{\Gamma_0}$,}
\end{equation}
where the $\pi_i : G_{\Gamma_0} \times G_{\Gamma_0} \recht G_{\Lambda} \times G_{\Lambda}$ are injective group homomorphisms of product form. After a further unitary conjugacy and regrouping direct summands, we may assume that for $i \neq j$, the group homomorphisms $\pi_i$ and $\pi_j$ are non conjugate, as homomorphisms from $G_{\Gamma_0} \times G_{\Gamma_0}$ to $G_{\Lambda} \times G_{\Lambda}$. Denote by $p_i \in M_d(\C)$ the projections corresponding to the $i$'th direct summand.

Whenever $\Gamma_1 \subg \Gamma$ and $L_1 \subg G_{\Gamma_1}$ are subgroups and $p \in M_d(\C)$ is a nonzero projection, we say that $\theta$ is \emph{standard} on $(\Gamma_1,L_1,p)$ if $p$ can be written as the sum of nonzero projections $q_j \in M_d(\C)$ such that for all $j$, the projection $q_j \ot 1$ commutes with $\theta(A_{\Gamma_1} \rtimes (L_1 \times L_1))$ and
$$(q_j \ot 1) \theta(A_{\Gamma_1}) \subset q_j M_d(\C) q_j \ot A_{\Lambda} \quad\text{and}\quad (q_j \ot 1)\theta(u_{(g,h)})  \in \cU(p \C^d) \ot u_{\zeta_j(g,h)}$$
for all $g,h \in L_1$ and injective group homomorphisms $\zeta_j : L_1 \times L_1 \recht G_\Lambda \times G_\Lambda$ of product form.

By construction, $\theta$ is standard on $(\Gamma_0,G_{\Gamma_0},1)$. We prove that automatically, without a further unitary conjugacy, $\theta$ is standard on $(\Gamma,L,1)$ for some finite index normal subgroup $L \subg G_\Gamma$.

Let $\Gamma_1 \subg \Gamma$ be an arbitrary countable subgroup with $\Gamma_0 \subg \Gamma_1$. We claim that there exists a finite index subgroup $L_1 \subg G_{\Gamma_1}$ with $[G_{\Gamma_1} : L_1] \leq \exp(d)$ such that
$\theta$ is standard on $(\Gamma_1,L_1,p_i)$ for every $i$.

Applying as above Lemma \ref{lem.step1} to the restriction of $\theta$ to $L(H_{\Gamma_1})$, we find integers $n_j \in \N$ with $\sum_j n_j = d$ and we find a unitary $V \in M_d(\C) \ot L(H_\Lambda)$ such that the restriction of $(\Ad V) \circ \theta$ to $L(H_{\Gamma_1})$ satisfies
\begin{equation}\label{eq.onGamma1}
V\theta(A_{\Gamma_1})V^* \subset \bigoplus_j \bigl( M_{n_j}(\C) \ot A_{\Lambda} \bigr) \quad\text{and}\quad V\theta(u_{(g,h)})V^* \in \bigoplus_j \bigl( \cU(\C^{n_j}) \ot u_{\zeta_j(g,h)}\bigr)
\end{equation}
for all $g,h$ in a finite index subgroup $L_1 \subg G_{\Gamma_1}$, where $\zeta_j : L_1 \times L_1 \recht H_\Lambda \times H_\Lambda$ are injective group homomorphisms of product form and $[G_{\Gamma_1} : L_1] \leq \exp(d)$. Denote by $q_j \in M_d(\C)$ the projection corresponding to the $j$'th direct summand in the above decomposition.

Define $L_0 = G_{\Gamma_0} \cap L_1$. Since $G_{\Gamma_0} \subg G_{\Gamma_1}$, we have that $L_0 \subg G_{\Gamma_0}$ has finite index.

Whenever $(q_j \ot 1)V (p_i \ot 1)$ is nonzero, comparing \eqref{eq.onGamma0} and \eqref{eq.onGamma1}, there must exist $x,y \in G_\Lambda$ such that
$$\{\zeta_j(g,h) \, (x,y) \, \pi_i(g,h)^{-1} \mid g,h \in L_0 \}$$
is a finite set. By the relative icc property of $\pi_i(L_0 \times L_0)$ in $G_\Lambda \times G_\Lambda$, this set must be the singleton $\{(x,y)\}$. Since moreover the action of $\pi_i(L_0 \times L_0)$ on $A_\Lambda$ is weakly mixing, we get that $(q_j \ot 1) V (p_i \ot 1) \in q_j M_d(\C) p_i \ot u_{(x,y)}$. In particular, the restrictions of $\pi_i$ and $\zeta_j$ to $L_0 \times L_0$ are conjugate.

Fix an index $j$. We claim that there is precisely one $i$ such that $(q_j \ot 1)V (p_i \ot 1)$ is nonzero. If $i \neq i'$ and if both $(q_j \ot 1)V (p_i \ot 1)$ and $(q_j \ot 1)V (p_{i'} \ot 1)$ are nonzero, we find that $\pi_i$ and $\pi_{i'}$ are conjugate on the finite index subgroup $L_0 \times L_0$ of $G_{\Gamma_0} \times G_{\Gamma_0}$. By the relative icc property, $\pi_i$ and $\pi_{i'}$ are conjugate on $G_{\Gamma_0} \times G_{\Gamma_0}$, contradicting the choices made above. So, the claim is proven and we can partition the indices $j$ into subsets $J_i$ such that $V^*(q_j \ot 1)V = p_{i,j} \ot 1$ for all $j \in J_i$, with $\sum_{j \in J_i} p_{i,j} = p_i$. Also,
$$(q_j \ot 1) V (p_i \ot 1) = V_{i,j} \ot u_{(x_{i,j},y_{i,j})} \; ,$$
where $V_{i,j}$ is a partial isometry with left support $q_j$ and right support $p_{i,j}$ whenever $j \in J_i$. It now follows from \eqref{eq.onGamma1} that $\theta$ is standard on $(\Gamma_1,L_1,p_i)$ for all $i$, as claimed above.

Since this holds for all countable subgroups $\Gamma_1 \subg \Gamma$ with $\Gamma_0 \subg \Gamma_1$ and since the index of $L_1$ in $G_{\Gamma_1}$ stays bounded by $\exp(d)$, it follows that $\theta$ is standard on $(\Gamma,L,1)$ for a finite index subgroup $L \subg G_\Gamma$, that we may next assume to be normal.

Since $\theta$ is standard on $(\Gamma,L,1)$, we have precisely proven that the conclusion of Lemma \ref{lem.step1} holds for the (potentially nonseparable) embedding $\theta : L(H_\Gamma) \recht L(H_\Lambda)^d$. Next, also the conclusion of Lemma \ref{lem.step2} holds, because the proof of Lemma \ref{lem.step2} was entirely self contained and only exploited the control of relative commutants using the relative icc property, which works equally well in an uncountable setting. The same holds for the first paragraphs of the proof of Proposition \ref{prop.all-embeddings}. Using the notation of \eqref{eq.reinterpret}, there thus exists a trace preserving embedding $\psi : L(D_\Gamma) \recht L(D_\Lambda)$ and an injective group homomorphism $\delta : G_\Gamma \recht G_\Lambda$ such that $\al_{\delta(g)} \circ \psi = \psi \circ \al_g$ for all $g \in L$.

For every $x \in \Z \ast \Gamma$, denote by $E_{x\Gamma} \cong L(\Z/2\Z)$ the natural subalgebra of $L(D_\Gamma)$. We similarly define $E_{y\Lambda} \subset L(D_\Lambda)$ for all $y \in \Z \ast \Lambda$.

If $K \subg G_\Lambda$ and $K \not\prec N_\Lambda$, the action of $K$ on $L(D_\Lambda)$ is ergodic. Since the elements of $E_{e\Lambda}$ are fixed by $\al_g$ for al $g \in N_\Gamma$, it follows that $\delta(L \cap N_\Gamma) \prec N_\Lambda$. Conjugating $\delta$ and replacing $L$ by a smaller finite index normal subgroup of $G_\Gamma$, we may assume that $\delta(L \cap N_\Gamma) \subg N_\Lambda$. Let $y \in G_\Lambda \setminus N_\Lambda$. Note that $\Lambda \subg \Z \ast \Lambda$ is malnormal, so that $N_\Lambda \cap y N_\Lambda y^{-1} = \Ker \pi_\Lambda$. Similarly, $N_\Gamma \cap x N_\Gamma x^{-1} = \Ker \pi_\Gamma$ for all $x \in G_\Gamma \setminus N_\Gamma$.

We claim that $\delta(L \cap N_\Gamma) \not\prec N_\Lambda \cap y N_\Lambda y^{-1} = \Ker \pi_\Lambda$. Otherwise, because $\Ker \pi_\Lambda$ is a normal subgroup of $G_\Lambda$, we find a finite index subgroup $K \subg L \cap N_\Gamma$ with $\delta(K) \subg \Ker \pi_\Lambda$. The equivariance of $\psi$ then implies that $\al_g = \id$ for all $g \in K$, which is absurd. By the claim, the subalgebra of $\delta(L \cap N_\Gamma)$-invariant elements in $D_\Lambda$ equals $E_{e\Lambda}$. So, $\psi(E_{e\Gamma}) \subseteq E_{e\Lambda}$. Since both are two-dimensional and $\psi$ is a trace preserving embedding, we have $\psi(E_{e\Gamma}) = E_{e\Lambda}$

Fix $a \in \Z \ast \Gamma \setminus \Gamma$. Choosing $x \in G_\Gamma$ with $\pi_\Gamma(x) = a$ and making a similar reasoning for the fixed points under $L \cap x N_\Gamma x^{-1}$, we find $b \in \Z \ast \Lambda$ such that $\psi(E_{a\Gamma}) = E_{b\Lambda}$. Since $\psi$ is injective, we must have that $b \in \Z \ast \Lambda \setminus \Lambda$. Take $y \in G_\Lambda$ with $\pi_\Lambda(y) = b$.

Since $\delta(L \cap N_\Gamma) \subg N_\Lambda$, we find that $\pi_\Lambda(\delta(L \cap N_\Gamma)) \subg \Gamma$ is a finite index normal subgroup of $\pi_\Lambda(\delta(N_\Gamma))$. Since $\Gamma \subg \Z \ast \Lambda$ is malnormal, it follows that $\pi_\Lambda(\delta(N_\Gamma)) \subg \Lambda$. Thus, $\delta(N_\Gamma) \subg N_\Lambda$.

Since $L \cap x N_\Gamma x^{-1}$ fixes $E_{a \Gamma}$ pointwise, we have that $\delta(L \cap x N_\Gamma x^{-1})$ fixes $E_{b \Lambda}$ pointwise. Therefore, $\delta(L \cap x N_\Gamma x^{-1}) \subg y N_\Lambda y^{-1}$. Reasoning as above, it follows that $\delta(x N_\Gamma x^{-1}) \subg y N_\Lambda y^{-1}$. Since $\delta$ is injective, it follows that
$$\delta(\Ker \pi_\Gamma) = \delta(N_\Gamma \cap x N_\Gamma x^{-1}) = \delta(N_\Gamma) \cap \delta(x N_\Gamma x^{-1}) \subg N_\Lambda \cap y N_\Lambda y^{-1} = \Ker \pi_\Lambda \; .$$
Conversely, $L \cap \delta^{-1}(N_\Lambda)$ must fix $E_{e\Gamma}$ pointwise, so that $L \cap \delta^{-1}(N_\Lambda) \subg N_\Gamma$. Since $L \cap \delta^{-1}(N_\Lambda)$ is a normal finite index subgroup of $\delta^{-1}(N_\Lambda)$, malnormality again implies that $\delta^{-1}(N_\Lambda) \subg N_\Gamma$. Similarly, $\delta^{-1}(y N_\Lambda y^{-1}) \subg x N_\Gamma x^{-1}$. Taking intersections, we conclude that $\delta^{-1}(\Ker \pi_\Lambda) \subg \Ker \pi_\Gamma$. Altogether, we have proven that $g \in G_\Gamma$ satisfies $\delta(g) \in \Ker \pi_\Lambda$ if and only if $g \in \Ker \pi_\Gamma$.

There is thus a unique, well defined, injective group homomorphism $\deltabar : \Z \ast \Gamma \recht \Z \ast \Lambda$ such that $\pi_\Lambda \circ \delta = \deltabar \circ \pi_\Gamma$. Since $\delta(N_\Gamma) \subg N_\Lambda$, we have $\deltabar(\Gamma) \subg \Lambda$. So, $\Gamma$ is isomorphic with a subgroup of $\Lambda$.

Finally note that if the embedding $\theta : L(H_\Gamma) \recht L(H_\Lambda)^d$ at the start of the proof has an image that is of finite index in $L(H_\Lambda)^d$,
then $\delta$ must be surjective and it follows that $\Gamma \cong \Lambda$. This ends the proof of the theorem.
\end{proof}

\begin{proof}[{Proof of Corollary \ref{cor.examples-orders}}]
Let $(I,\leq)$ be a partially ordered set with density character $\kappa$. Take a sup-dense subset $I_0 \subseteq I$ of cardinality at most $\kappa$. Then the map $\vphi : I \recht 2^{I_0} : \vphi(i) = \{j \in I_0 \mid j \leq i\}$ is injective and has the property that $\vphi(i) \subseteq \vphi(j)$ if and only if $i \leq j$. It thus suffices to realize the partially ordered set $(2^{I_0},\subseteq)$ in the class of II$_1$ factors with density character at most $\kappa$ and relation $\semb$ or $\emb$.

By Theorem \ref{thm.family-with-groups}, it suffices to realize $(2^{I_0},\subseteq)$ in the class of infinite groups with cardinality at most $\kappa$ and with relation ``is isomorphic with a subgroup of'' that we denote as $\emb$. Formally add to $I_0$ two extra elements $a,b$ and put $I_1 = I_0 \sqcup \{a,b\}$. Assume that $(\Lambda_i)_{i \in I_1}$ is a family of nontrivial groups of cardinality at most $\kappa$ having the following properties: $\Lambda_i \not\cong \Z$, $\Lambda_i$ is freely indecomposable and if $i \neq j$, then $\Lambda_i \not\prec \Lambda_j$. Whenever $J \subseteq I_0$, define $\Gamma_J$ as the free product of the groups $\Lambda_i$, $i \in I_0 \cup \{a,b\}$. For every $J \subseteq I_0$, the group $\Gamma_J$ is infinite (since it contains $\Lambda_a \ast \Lambda_b$), of cardinality at most $\kappa$, and by the Kurosh theorem, $\Gamma_J \emb \Gamma_{J'}$ if and only if $J \subseteq J'$.

When $I_0$ is countable, it is very easy to give an explicit family of such mutually non embeddable groups $(\Lambda_n)_{n \in \N}$. It suffices to enumerate the prime numbers $(p_n)_{n \in \N}$ and take $\Lambda_n = \Z/p_n \Z$.

When $\kappa$ is any cardinal number, the argument is obviously less explicit. By \cite{FK77}, we can choose a family $(K_i)_{i \in I_1}$ of infinite fields of cardinality at most $\kappa$ with the property that for $i \neq j$, $K_i$ is not isomorphic to a subfield of $K_j$. Then the groups $\Lambda_i = K_i \rtimes K_i^*$ are freely indecomposable (even amenable) and for $i \neq j$, we have that $\Lambda_i \not\emb \Lambda_j$.

Corollary \ref{cor.family-with-a} already provides the concrete chain of separable II$_1$ factors $(M_t)_{t \in \R}$ of order type $(\R,\leq)$, both in $(\IIone,\emb)$ or $(\IIone,\mathord{\semb})$.

We concretely realize as follows a chain of separable II$_1$ factors of order type $\om_1$, the first uncountable ordinal. Fix a prime number $p$. For every countable ordinal $\lambda < \om_1$, recall from \cite[Theorem 10]{Ric71} the standard construction of a countable abelian $p$-group $B_\lambda$ of length $\lambda$ as the abelian group with generators $a(i_1, \ldots, i_n)$ indexed by strictly increasing sequences $0 \leq i_1 < \cdots i_n < \lambda$ of any finite length $n \in \N$ subject to the relations
$$p \, a(i) = 0 \quad\text{and}\quad p \, a(i_1,i_2,\ldots,i_n) = a(i_2,\ldots,i_n)$$
for all $0 \leq i < \lambda$ and $n \geq 2$, $0 \leq i_1 < \cdots i_n < \lambda$. For all $\mu,\lambda < \om_1$, we have that $B_\lambda \emb B_\mu$ iff $\lambda \leq \mu$. By Theorem \ref{thm.family-with-groups}, the II$_1$ factors $M_\lambda = M(\Z \times \Lambda_\lambda)$ form a chain of separable II$_1$ factors $(M_\lambda)_{\lambda < \omega_1}$ of order type $\om_1$, both in $(\IIone,\emb)$ and $(\IIone,\mathord{\semb})$.
\end{proof}

\section[Complete intervals of II$_1$ factors: proof of Theorem \ref{thm.complete-intervals}]{\boldmath Complete intervals of II$_1$ factors: proof of Theorem \ref{thm.complete-intervals}}\label{sec.lattices-of-II1-factors}

We prove Theorem \ref{thm.complete-intervals} as a consequence of a general construction of II$_1$ factors $M_1 \subseteq M$ where we have a complete control over \emph{all} II$_1$ factors $N$ such that $M_1 \semb N$ and $N \semb M$.

The construction is a variant of the construction in Theorem \ref{thm.all-embed-trivial}. So we again denote by $G$ the group of finite even permutations of $\N$ and let $G_1 \subg G$ be the subgroup of permutations that fix $1 \in \N$. Denote by $\cA \subset G_1$ the elements of order $2$ and denote by $\cB \subset G_1$ the elements of order $3$.

We then introduce an extra variable in the construction: let $K$ be any countable group such that $\Gamma = G \ast K$ belongs to the family $\cC$ of Definition \ref{def.family-C}. In our concrete applications, $K$ will be a free product of amenable groups, so that indeed $\Gamma \in \cC$.

Define the normal subgroup $\Gamma_1 \lhd \Gamma$ as the kernel of the natural homomorphism $\pi : \Gamma \recht K$. Put $\Gamma_0 = \Gamma_1 \cap (G_1 \ast K)$. For all $a \in \cA$, $b \in \cB$ and $k \in K \setminus \{e\}$, consider the subgroup $\Lambda_{a,b,k} = \langle a \rangle \ast k \langle b \rangle k^{-1}$ of $\Gamma_0$. Choose distinct probability measures $\mu_{a,b,k}$ on $\{0,1\}$ with $\mu_{a,b,k}(0) \in (0,1/2)$. Denote by $\Delta : \Gamma_0 \recht \Gamma_0 \times \Gamma_0$ the diagonal embedding. Consider the generalized Bernoulli action
\begin{equation}\label{eq.my-new-action}
\Gamma_0 \times \Gamma \actson (X,\mu) = \prod_{a \in \cA, b \in \cB, k \in K \setminus \{e\}} (\{0,1\},\mu_{a,b,k})^{(\Gamma_0 \times \Gamma)/\Delta(\Lambda_{a,b,k})} \; .
\end{equation}
Write $A = L^\infty(X,\mu)$ and define the II$_1$ factors $M = A \rtimes (\Gamma_0 \times \Gamma)$ and $M_1 = A \rtimes (\Gamma_0 \times \Gamma_1)$. To describe all II$_1$ factors $N$ such that $M_1 \semb N$ and $N \semb M$, we need the following notation.

\begin{notation}\label{not.H2fin}
Given a countable group $L$, we denote by $H^2_\fin(L,\T)$ the space of scalar $2$-cocycles $\om : L \times L \recht \T$ (up to coboundaries) that arise as the so-called obstruction $2$-cocycle of a \emph{finite dimensional} projective unitary representation $\pi : L \recht \cU(\C^n) : \pi(g) \pi(h) = \om(g,h) \pi(gh)$.

We always implicitly assume that $\om$ is normalized, meaning that $\om(g,e) = \om(e,g) = 1$ and $\pi(e) = 1$.
\end{notation}

\begin{theorem}\label{thm.complete-intervals-construction}
Let $\Gamma_0,\Gamma$ and $K$ be groups as above. Consider the action $\Gamma_0 \times \Gamma \actson (X,\mu)$ defined by \eqref{eq.my-new-action} and define $M_1 \subset M$ as above. Whenever $L \subg K$ is a subgroup, define $\Gamma_L = \pi^{-1}(L)$. When $\om \in H^2_\fin(L,\T)$, define $\omtil \in H^2_\fin(\Gamma_0 \times \Gamma_L)$ by $\omtil = 1 \times (\om \circ \pi)$. Put $M(L,\om) = A \rtimes_{\omtil} (\Gamma_0 \times \Gamma_L)$.
\begin{enumlist}
\item A II$_1$ factor $N$ satisfies $M_1 \semb N$ and $N \semb M$ if and only if $N$ is stably isomorphic with $M(L,\om)$ for a subgroup $L \subg K$ and an $\om \in H^2_\fin(L,\T)$.

\item Let $L_i \subg K$ and $\om_i \in H^2_\fin(L_i,\T)$ for $i = 1,2$. We have $M(L_1,\om_1) \semb M(L_2,\om_2)$ if and only if there exists a $g \in K$ such that $L_1 \cap g L_2 g^{-1} \subg L_1$ has finite index.

\item Let $t > 0$, $L_i \subg K$ and $\om_i \in H^2_\fin(L_i,\T)$ for $i = 1,2$. We have $M(L_1,\om_1) \cong M(L_2,\om_2)^t$ if and only if $t=1$ and there exists a $g \in K$ such that $g^{-1} L_1 g = L_2$ and $\om_1 = \om_2 \circ \Ad g$ in $H^2_\fin(L_1,\T)$ (i.e.\ equal up to a coboundary).

\item Let $L_i \subg K$ and $\om_i \in H^2_\fin(L_i,\T)$ for $i = 1,2$. The II$_1$ factors $M(L_1,\om_1)$ and $M(L_2,\om_2)$ are virtually isomorphic if and only if there exists a $g \in K$ such that $g^{-1} L_1 g \cap L_2$ has finite index in both $g^{-1} L_1 g$ and $L_2$.
\end{enumlist}
\end{theorem}
\begin{proof}
Note that $\Gamma$ naturally is the semidirect product $\Gamma = \Gamma_1 \rtimes K$. In this way, we obtain a canonical outer action $\al$ of $K$ on $M_1$ such that $M = M_1 \rtimes K$. Under this identification, we have that $M(L,\om) = M_1 \rtimes_\om L$ whenever $L \subg K$ and $\om \in H^2_\fin(L,\T)$.

Also note that by Lemma \ref{lem.another-group-lemma}, all hypotheses of Lemma \ref{lem.all-embedding-trivial-general} are satisfied, so that all embeddings of $M_1$ into $M^t$ are trivial.

1. This is now a consequence of Lemma \ref{lem.intermediate-subfactors} below.

2. Assume that $t > 0$ and $M(L_1,\om_1) \emb M(L_2,\om_2)^t$. Write $N_i = M(L_i,\om_i)$. We encode the embedding $N_1 \emb N_2^t$ as a bimodule $\bim{N_1}{\cH}{N_2}$ with $\dim_{-N_2}(\cH) = t < \infty$. Take a $d$-dimensional projective representation $\zeta : L_2 \recht \cU(\C^d)$ with associated $2$-cocycle $\om_2$. Denote by $\rho : N_2 \recht M_d(\C) \ot M$ the corresponding canonical embedding given by $\rho(a) = 1 \ot a$ for all $a \in M_1$ and $\rho(u_g) = \zeta(g) \ot u_g$ for all $g \in L_2$. Define $\bim{N_2}{\cK}{M}$ as the bimodule corresponding to $\rho$.

Since $\dim_{-M}(\cK) = d$, the $N_1$-$M$-bimodule $\cH \ot_{N_2} \cK$ has finite right dimension $t d$. Restricting this bimodule to an $M_1$-$M$-bimodule, Lemma \ref{lem.all-embedding-trivial-general} says that it must be a direct sum of the trivial inclusion bimodule $\bim{M_1}{L^2(M)}{M}$. We conclude that the $M_1$-$N_2$-bimodule $\cH \ot_{N_2} \cK \ot_M \overline{\cK}$ is isomorphic with a multiple of $\bim{M_1}{\overline{\cK}}{N_2}$.

Note that $\cK \ot_M \overline{\cK} = M_d(\C) \ot L^2(M)$ so that $\rho$ provides an embedding of the trivial $N_2$-$N_2$-bimodule into $\cK \ot_M \overline{\cK} = M_d(\C) \ot L^2(M)$. We conclude that $\bim{M_1}{\cH}{N_2}$ is contained in a multiple of $\bim{M_1}{\overline{\cK}}{N_2}$. Restricting both to $M_1$-$M_1$-bimodules, it follows that $\bim{M_1}{\cH}{M_1}$ is contained in a multiple of $\bim{M_1}{\overline{\cK}}{M_1}$.

For every $g \in K$, define $\bim{M_1}{\cL^g}{M_1}$ as the irreducible bimodule encoding the automorphism $\al_g \in \Aut(M_1)$, i.e.\ $\cL^g = L^2(M_1)$ and $a \cdot \xi \cdot b = a\xi \al_g(b)$. Then, $\bim{M_1}{\overline{\cK}}{M_1}$ is a direct sum of bimodules of the form $\bim{M_1}{\cL^g}{M_1}$. We have thus shown that $\bim{M_1}{\cH}{M_1}$ is isomorphic with a direct sum of bimodules of the form $\bim{M_1}{\cL^g}{M_1}$. So we uniquely find a subset $I \subset K$ and a decomposition $\cH = \bigoplus_{g \in I} \cH^g$ of $\cH$ into $M_1$-$M_1$-subbimodules such that $\bim{M_1}{\cH^g}{M_1}$ is isomorphic with a direct sum of $d_g \in \N \cup \{+\infty\}$ copies of $\cL^g$.

We now consider the left action by $u_s$, $s \in L_1$. Since $u_s$ implements the automorphism $\al_s$ on $M_1$, we must have $u_s \cdot \cH^g = \cH^{sg}$ for every $g \in I$. Similarly, $\cH^g \cdot u_r = \cH^{gr}$ for all $r \in L_2$. Since $d_g = \dim_{-M_1}(\cH^g) = \dim_{M_1-}(\cH^g)$, it follows that $d_{sgr} = d_g$ for all $g \in I$, $s \in L_1$ and $r \in L_2$. In particular, $I = L_1 I L_2$. It also follows that
$$t = \dim_{-N_2}(\cH) = \sum_{gL_2 \in I/L_2} d_g \; .$$
In particular, $t \in \N$, $d_g < +\infty$ for all $g \in I$ and $|I/L_2| < \infty$. Fix any $g \in I$. Since $I/L_2$ is a finite set on which $L_1$ acts, the stabilizer of $g L_2$ has finite index in $L_1$. This stabilizer equals $L_1 \cap g L_2 g^{-1}$.

To prove the converse statement, write for brevity $M(L)$ instead of $M(L,1)$ with respect to the trivial $2$-cocycle $1 \in H^2_\fin(L,\T)$. When $L = L_1 \cap g L_2 g^{-1}$ has finite index in $L_1$, note that $M(L_1) \semb M(L)$ and, via $\Ad u_g^*$, also $M(L) \emb M(L_2)$. We also have $M(L_1,\om_1) \semb M(L_1)$ and $M(L_2) \semb M(L_2,\om_2)$. Thus, $M(L_1,\om_1) \semb M(L_2,\om_2)$.

3. Assume that $t > 0$ and $M(L_1,\om_1) \cong M(L_2,\om_2)^t$. Again write $N_i = M(L_i,\om_i)$ and encode the isomorphism $N_1 \cong N_2^t$ as a bimodule $\bim{N_1}{\cH}{N_2}$ with $\dim_{-N_2}(\cH) = t$ and $\dim_{N_1-}(\cH) = 1/t$. Make the decomposition $\cH = \bigoplus_{g \in I} \cH^g$ as in the proof of 2. Since
$$t = \dim_{-N_2}(\cH) = \sum_{gL_2 \in I/L_2} d_g \quad\text{and}\quad 1/t = \dim_{N_1-}(\cH) = \sum_{L_1 g \in L_1 \backslash I} d_g \; ,$$
it follows that $t = 1$, $d_g = 1$ for all $g \in I$ and both $I/L_2$ and $L_1 \backslash I$ are singletons. So, we find $g_0 \in I$ such that $I = g_0 L_2 = L_1 g_0$. Thus, $g_0^{-1} L_1 g_0 = L_2$.

Since $d_g = 1$ for all $g \in I$, we can identify $\cH = L^2(M_1) \ot \ell^2(I)$ such that the right action by $u_r$, $r \in L_2$, is given by
$$(a \ot \delta_{g_0 g}) \cdot u_r = \om_2(g,r) \, (a \ot \delta_{g_0 gr}) \quad\text{for all $a \in M_1$, $g,r \in L_2$.}$$
Define $\gamma : L_1 \recht \T$ such that the left action by $u_s$, $s \in L_1$, satisfies
$$u_s \cdot (a \ot \delta_{g_0}) = \gamma(s) \, (\al_s(a) \ot \delta_{s g_0}) \quad\text{for all $a \in M_1$, $s \in L_1$.}$$
Since $u_s u_{s'} = \om_1(s,s') u_{ss'}$ for all $s,s' \in L_1$, we get that
$$u_s \cdot (a \ot \delta_{s' g_0}) = \overline{\gamma(s')} \, \om_1(s,s') \, \gamma(ss') \, (\al_s(a) \ot \delta_{ss'g_0}) \quad\text{for all $a \in M_1$, $s \in L_1$.}$$
Since the left and right actions commute, we also get that
$$u_s \cdot (a \ot \delta_{s' g_0}) = \gamma(s) \, \om_2(g_0^{-1}sg_0,g_0^{-1}s' g_0) \, (\al_s(a) \ot \delta_{ss'g_0}) \quad\text{for all $a \in M_1$, $s \in L_1$.}$$
Comparing both, it follows that $\om_1$ is cohomologous to $\om_2 \circ \Ad g_0^{-1}$.

The converse statement is trivial.

4. Assume that $N_1 = M(L_1,\om_1)$ and $N_2 = M(L_2,\om_2)$ are virtually isomorphic, through a bimodule $\bim{N_1}{\cH}{N_2}$ with $\dim_{-N_2}(\cH) < +\infty$ and $\dim_{N_1-}(\cH) < +\infty$. Making the decomposition $\cH = \bigoplus_{g \in I} \cH^g$ as in the proof of 2, we get that $L_1 \backslash I$ and $I / L_2$ are finite sets. Taking any $g \in I$, we find that $g^{-1} L_1 g \cap L_2$ has finite index in both $g^{-1} L_1 g$ and $L_2$.

Conversely, first note that $M(L_i,\om_i)$ and $M(L_i,1)$ are virtually isomorphic. If $g^{-1} L_1 g \cap L_2$ has finite index in both $g^{-1} L_1 g$ and $L_2$, it follows immediately that $M(L_1,1)$ and $M(L_2,1)$ are virtually isomorphic.
\end{proof}

We can now prove Theorem \ref{thm.complete-intervals} as stated in the introduction. After that, we formulate a seemingly weaker result in Proposition \ref{prop.complete-interval-concrete-for-completion}, but with a much more concrete construction of the family of II$_1$ factors.

\begin{proof}[{Proof of Theorem \ref{thm.complete-intervals}}]
Applying Theorem \ref{thm.realizing-lattices} to the countable group $\Lambda = A_\infty$ of all finite even permutations of $\N$, we find a group $K$ with subgroup $L \subg K$ such that the following properties hold.
\begin{enumlist}
\item The intermediate subgroup lattice $\{L_1 \mid L \subg L_1 \subg K\}$ is isomorphic with $(I,\leq)$.
\item Every intermediate subgroup $L \subg L_1 \subg K$ is freely generated by isomorphic copies of $A_\infty$.
\item For every $g \in K$, we have that $g \in L \vee gLg^{-1}$.
\end{enumlist}
Since $K$ is a free product of amenable groups, Theorem \ref{thm.complete-intervals-construction} applies. By property~2, the intermediate subgroups $L \subg L_1 \subg K$ have no nontrivial finite dimensional unitary representations. In particular, these groups do not have proper finite index subgroups. In Theorem \ref{thm.complete-intervals-construction}, the embedding relations $\emb$ and $\semb$ then amount to the relation $g L_1 g^{-1} \subg L_2$ for some $g \in K$. If this relation holds, property~3 and the fact that $L \subg L_1$ and $L \subg L_2$ imply that
$$g \in L \vee g L g^{-1} \subg L \vee g L_1 g^{-1} \subg L \vee L_2 = L_2 \; ,$$
so that $L_1 \subg L_2$. This concludes the proof.
\end{proof}

\begin{proposition}\label{prop.complete-interval-concrete-for-completion}
Let $(I,\leq)$ be any countable partially ordered set and denote by $(\Ibar,\leq)$ its completion given by all downward closed subsets of $(I,\leq)$. There is a concrete construction of II$_1$ factors $(M_i)_{i \in \Ibar}$ with separable predual satisfying the conclusions of Theorem \ref{thm.complete-intervals}.
\end{proposition}
\begin{proof}
Choose any countably infinite locally finite field $\cK$. Define the $\cK$-vector space $V = (\cK^2)^{(I)}$ as the direct sum of copies of $\cK^2$ indexed by $I$. Note that $V$ is countable. For every $a \in I$, we denote by $V_a \subseteq V$ the natural direct summand in position $a \in I$ and we denote by $V'_a$ the direct sum of all $V_b$ with $b \neq a$, so that $V = V_a \oplus V'_a$. Whenever $v \in \cK^2$ and $a \in I$, we have the natural vector $v_a \in V_a$.

For every $a \in I$ and $A \in \SL_2(\cK)$, define $\pi_a(A) \in \GL(V)$ by $\pi_a(A)(v_a) = (A(v))_a$ for all $v \in \cK^2$ and $\pi_a(A)(w) = w$ for all $w \in V'_a$. Define the group $T \subg \SL_4(\cK)$ by
$$T = \bigl\{ \bigl(\begin{smallmatrix} A & X \\ 0 & B\end{smallmatrix}\bigr) \bigm| A,B \in \SL_2(\cK), X \in M_2(\cK) \bigr\} \; .$$
Whenever $a,b \in I$ and $a \neq b$, denote by $\pi_{a,b} : T \recht \GL(V)$ the natural representation acting in coordinates $a$ and $b$, given by
$$\pi_{a,b}\bigl(\begin{smallmatrix} A & X \\ 0 & B\end{smallmatrix}\bigr) = S \quad\text{where}\quad S(u) = u \;\;\text{if $u \in V'_a \cap V'_b$,}\;\; S(v_a + w_b) = (A(v) + X(w))_a + (B(w))_b \; .$$
Note that when $X = 0$, then $S = \pi_a(A) \, \pi_b(B)$. We define the countable subgroup $K_0 \subg \GL(V)$ as the subgroup generated by all $\pi_{a,b}(T)$ with $a,b \in I$ and $a < b$. Note that $\pi_a(\SL_2(\cK)) \subg K_0$ for all $a \in I$.
We view $V$ as a countable abelian group and $K_0 \actson V$ acting by automorphisms. Put $K = V \rtimes K_0$.

For every subset $J \subseteq I$, denote by $V(J) \subseteq V$ the direct sum of all $V_a$ with $a \in J$. We claim that the intermediate subgroups $K_0 \subg L \subg K$ are precisely given by $V(J) \rtimes K_0$ where $J \subseteq I$ is a downward closed set. To prove this claim, fix such an intermediate subgroup $L$. Then, $L = V_0 \rtimes K_0$ where $V_0 \subg V$ is an additive subgroup that is globally invariant under the action of $K_0$.

Note that if $v \in \cK^2$ is any nonzero vector, then the additive subgroup of $\cK^2$ generated by $\{A(v) - v \mid A \in \SL_2(\cK)\}$ equals $\cK^2$. Indeed, let $w \in \cK^2$ be an arbitrary vector and take one of the two standard basis vectors $e \in \cK^2$ such that $w+e \neq 0$. Then choose $A,B \in \SL_2(\cK)$ such that $A(e) = w+e$ and $B(v) = e$. We get that
$$(AB(v) - v) - (B(v) - v) = AB(v) - B(v) = (w+e) - e = w \; .$$

Denote by $p_a : V \recht \cK^2$ the natural projection map, so that $v = \sum_{a \in I} (p_a(v))_a$ for all $v \in V$. If $v \in V_0$ and $a \in I$, we have
$$V_0 \ni \pi_a(A)(v) - v = (A(p_a(v)) - p_a(v))_a$$
for all $A \in \SL_2(\cK)$. The observation in the previous paragraph thus implies that $V_a \subset V_0$ whenever $p_a(V_0) \neq \{0\}$. Defining $J \subseteq I$ as the set of all $a \in I$ such that $p_a(V_0) \neq \{0\}$, we get that $V_0 = V(J)$.

When $a \leq b$, we have $\pi_{a,b}(T)(V_b) = V_a$. It follows that $J$ is downward closed. So the claim is proven.

Note that $K$ is a locally finite group. In particular, $K$ is amenable. Since $\SL_2(\cK)$ has no nontrivial finite dimensional unitary representations, the same holds for the group $T$. Then, every finite dimensional unitary representation of $K_0$ has to be the identity on each copy $\pi_{a,b}(T)$ of $T$, and hence on $K_0$. It then follows that also the intermediate subgroups $V_0 \rtimes K_0$ have no nontrivial finite dimensional unitary representations.

For every downward closed set $J \subseteq I$, write $L(J) = V(J) \rtimes K_0$. We finally prove that $g L(J_1) g^{-1} \subg L(J_2)$ for some $g \in K$ if and only if $J_1 \subseteq J_2$. One implication being trivial, assume that $g L(J_1) g^{-1} \subg L(J_2)$. Since $V(J_1) = g V(J_1) g^{-1} \subg g L(J_1) g^{-1} \subg L(J_2)$, we get that $V(J_1) \subg V(J_2)$ and thus, $J_1 \subseteq J_2$.

We can apply Theorem \ref{thm.complete-intervals-construction} and the proof is complete.
\end{proof}

\begin{lemma}\label{lem.another-group-lemma}
The groups $\Lambda_{a,b,k} < \Gamma_0 < \Gamma_1 < \Gamma$ introduced in Theorem \ref{thm.complete-intervals-construction} satisfy all the assumptions in \ref{my-assum}.
\end{lemma}
\begin{proof}
Condition \ref{one} is obvious. We can view $\Gamma_1$ as the free product of the subgroups $k G k^{-1}$ with $k \in K$. Similarly, $\Gamma_0$ is the free product of the subgroups $k G_1 k^{-1}$ with $k \in K$. Therefore, $\Gamma_0$ and $\Gamma_1$ have no nontrivial finite dimensional unitary representations, confirming \ref{two}.

The second statement of Lemma \ref{lem.playing-with-words} implies that conditions \ref{three} and \ref{four} hold.

To prove \ref{five}, let $\delta : \Gamma_0 \recht \Gamma$ be an injective group homomorphism such that $\delta(\Lambda_{a,b,k}) \prec_\Gamma \Lambda_{a,b,k}$ for all $a \in \cA$, $b \in \cB$, $k \in K \setminus \{e\}$. After composing $\delta$ with an inner automorphism of $\Gamma$ and using the Kurosh theorem, we find for all $k \in K$, $T_k \in \{G,K\}$, injective group homomorphisms $\delta_k : G_1 \recht T_k$ and elements $w_k \in \Gamma$ such that
$$\delta(g) = \delta_e(g) \;\;\text{for all}\;\; g \in G \quad\text{and}\quad \delta(k g k^{-1}) = w_k \delta_k(g) w_k^{-1} \;\;\text{for all $k \in K \setminus \{e\}$ and $g \in G_1$.}$$
Denoting by $T_k' \in \{G,K\}$ the ``other'' group so that $\{T_k,T'_k\} = \{G,K\}$, we may furthermore assume that for every $k \in K \setminus \{e\}$, either $w_k = e$, or $w_k$ ends with a letter from $T'_k \setminus \{e\}$.

So, for every $a \in \cA$, $b \in \cB$ and $k \in K \setminus \{e\}$, we get that $\delta(\Lambda_{a,b,k})$ equals the free product $\langle \delta_e(a) \rangle \ast w_k \langle \delta_k(b) \rangle w_k^{-1}$. Since $\delta(\Lambda_{a,b,k}) \prec_\Gamma \Lambda_{a,b,k}$, the second statement of Lemma \ref{lem.playing-with-words} implies that $T_e = T_k = G$ and that, because $w_k$ ends with a letter from $K \setminus \{e\}$, we can write $w_k$ as the reduced word $w_k = u_k^{-1} v_k k$, with $u_k \in G$ and $v_k \in \Lambda_{a,b,k}$. Given $k$, we must thus have $v_k \in \Lambda_{a,b,k}$ for all $a \in \cA$, $b \in \cB$. It follows that $v_k = e$. By Lemma \ref{lem.playing-with-words}, we also have that $u_k \delta_e(a) u_k^{-1} = a$ and $\delta_k(b) = b^{\pm 1}$.

For a fixed $k \in K \setminus \{e\}$, we thus get that $\delta_k(b) = b^{\pm 1}$ for all $b \in \cB$. By Lemma \ref{lem.elem-permutation}, $\delta_k$ is the identity homomorphism from $G_1$ to $G$. It also follows that, $\delta_e(g) = u_k^{-1} g u_k$ for all $g \in G$ and $k \in K \setminus \{e\}$. This forces all $u_k$ to be equal to a single $u \in G$. We have now proven that $(\Ad u) \circ \delta$ is the identity homomorphism from $\Gamma_0$ to $\Gamma$.

When $\delta : \Gamma_1 \recht \Gamma$ is an injective group homomorphism such that $\delta(\Lambda_{a,b,k}) \prec_\Gamma \Lambda_{a,b,k}$ for all $a \in \cA$, $b \in \cB$, $k \in K \setminus \{e\}$, it follows from the previous paragraph that after a conjugacy, $\delta$ is the identity on $\Gamma_0$. In particular, $\delta(k G_1 k^{-1}) = k G_1 k^{-1}$ for all $k \in K$. This forces $\delta(k G k^{-1}) \subg k G k^{-1}$. A homomorphism $G \recht G$ that is the identity on $G_1$ must be the identity on $G$. We conclude that $\delta$ is the identity on $k G k^{-1}$ for all $k \in K$. So, $\delta$ is the identity and also condition \ref{five} holds.
\end{proof}

\begin{lemma}\label{lem.intermediate-subfactors}
Let $K$ be a countable group, $P$ a II$_1$ factor and $K \actson P$ an outer action. Let $N$ be a II$_1$ factor. Then the following statements are equivalent.
\begin{enumlist}
\item There exists a $d \in \N$ such that $N$ is stably isomorphic with an intermediate subfactor of $1 \ot P \subseteq M_d(\C) \ot (P \rtimes K)$.
\item There exists a subgroup $L \subg K$ and $\om \in H^2_\fin(L,\T)$ (see Notation \ref{not.H2fin}) such that $N$ is stably isomorphic with $P \rtimes_\om L$.
\end{enumlist}
\end{lemma}
\begin{proof}
2 $\Rightarrow$ 1. If $L \subg K$ is a subgroup and $\om \in H^2_\fin(L,\T)$, choose a finite dimensional projective representation $\pi : L \recht \cU(\C^d)$ such that $\pi(g) \pi(h) = \om(g,h) \pi(gh)$ for all $g,h \in L$. Define the embedding
$$\theta : P \rtimes_\om L \recht M_d(\C) \ot (P \rtimes K) : \theta(a u_g) = \pi(g) \ot a u_g \quad\text{for all $a \in P$, $g \in L$.}$$
Then, $\theta(P \rtimes_\om L)$ is an intermediate subfactor for $1 \ot P \subseteq M_d(\C) \ot (P \rtimes K)$.

1 $\Rightarrow$ 2. Assume that $1 \ot P \subseteq M \subseteq M_d(\C) \ot (P \rtimes K)$ is an intermediate subfactor. We prove that $M$ is stably isomorphic with $P \rtimes_\om L$ for some subgroup $L \subg K$ and $\om \in H^2_\fin(L,\T)$.

Write $N = M_d(\C) \ot (P \rtimes K)$ and $D = M_d(\C) \ot P$. We can view $N$ as the crossed product $N = D \rtimes K$. Denote by $E_D : N \recht D$ the unique trace preserving conditional expectation. For every $g \in K$, define
$$\phi_g : N \recht D u_g : \phi_g(x) = E_D(xu_g^*) u_g \; .$$
We claim that for every $g \in K$, there exists a vector subspace $V_g \subset M_d(\C)$ such that
\begin{equation}\label{eq.my-inclusion-Vg}
\phi_g(M) \subseteq V_g \ot P u_g \subseteq M \; .
\end{equation}
To prove this claim, define
$$V_g = \lspan \bigl\{ (\id \ot \tau)\bigl(E_D(x u_g^*) (1 \ot y^*)\bigr) \bigm| x \in M, y \in P \bigr\} \; .$$
By definition, $\phi_g(M) \subseteq V_g \ot P u_g$. We have to prove that $V_g \ot P u_g \subseteq M$.

Fix $x \in M$, $y \in P$ and write $a = (\id \ot \tau)\bigl(E_D(x u_g^*) (1 \ot y^*)\bigr)$. Then,
$$a = (\id \ot \tau)\bigl(E_D(x (1 \ot \al_{g^{-1}}(y^*)) u_g^*)\bigr) = (\id \ot \tau)(x (1 \ot \al_{g^{-1}}(y^*))u_g^*)\;.$$
Since $P \subseteq P \rtimes K$ is irreducible, it follows that $a \ot u_g$ belongs to the $\|\,\cdot\,\|_2$-closure of the convex hull of the elements
$$(1 \ot b) x (1 \ot \al_{g^{-1}}(y^*b^*)) \;\;\text{where $b \in \cU(P)$.}$$
Since $1 \ot P \subseteq M$, all these elements belong to $M$ and we conclude that $V_g \ot u_g \subseteq M$. Then also $V_g \ot P u_g \subseteq M$. So \eqref{eq.my-inclusion-Vg} is proven.

Since $\phi_g(x) = x$ for all $x \in M_d(\C) \ot P u_g$, it follows from \eqref{eq.my-inclusion-Vg} that
\begin{equation}\label{eq.better-descr-Vg}
\phi_g(M) = V_g \ot P u_g = M \cap (M_d(\C) \ot P u_g) \; .
\end{equation}
Write $B = V_e$. By \eqref{eq.better-descr-Vg}, $B \subseteq M_d(\C)$ is a unital $*$-subalgebra. It also follows from \eqref{eq.better-descr-Vg} that $V_g V_g^* \subseteq B$, $V_g^* V_g \subseteq B$ and $B V_g B = V_g$.

Let $p \in B$ be a minimal projection. Since $p \ot 1 \in M$, we can replace $M$ by the stably isomorphic $(p \ot 1)M(p \ot 1)$ and we may thus assume that $B = \C 1$. Define $L = \{g \in K \mid V_g \neq \{0\}\}$. It follows that for every $g \in L$, $V_g = \C \pi(g)$ where $\pi(g) \in \cU(\C^d)$. Since $V_g V_h \subseteq V_{gh}$, it follows that $\pi$ is a projective representation. Denote by $\om \in H^2_\fin(L,\om)$ the associated $2$-cocycle. We have proven that $M$ is generated by the elements $\pi(g) \ot P u_g$, $g \in L$. This precisely means that $M \cong P \rtimes_\om L$.
\end{proof}

\section[II$_1$ factors with prescribed semiring of self-embeddings]{\boldmath II$_1$ factors with prescribed semiring of self-embeddings}\label{sec.self-embed}

Given a II$_1$ factor $M$, we consider all possible stable self-embeddings $M \semb M$ and identify two embeddings when they are unitarily conjugate. Since we can take direct sums and compositions of embeddings, we obtain the semiring $\Embs(M)$ that we call the embeddings semiring of $M$. Note that $\Embs(M)$ can be identified with the space of isomorphism classes of $M$-$M$-bimodules $\bim{M}{H}{M}$ having finite right dimension.

The invariant $\Embs(M)$ is extremely rich, encoding at the same time the outer automorphism group of $M$, the fundamental group of $M$ and the fusion ring of finite index $M$-$M$-bimodules.
In this section, we deduce from Theorems \ref{thm.all-embed-trivial} and \ref{thm.complete-intervals-construction} several concrete computations of $\Embs(M)$. As a consequence, we can prove the following result. Recall that a semigroup $\cF$ is called left cancellative if $gh = gk$ implies $h=k$.

\begin{theorem}\label{thm.prescribed-embeddings-semiring}
Let $\cS$ be one of the following semigroups.
\begin{enumlist}
\item A countable, unital, left cancellative semigroup.
\item The semigroup $\Emb(\cG)$ of self-embeddings of a countable structure (like a countable group, or a countable field).
\item The semigroup $\cI(N) = \{v \in N \mid v^* v = 1\}$ of isometries in a von Neumann algebra $N$ with separable predual.
\end{enumlist}
There then exists a II$_1$ factor $M$ with separable predual such that $\Embs(M) \cong \N[\cS]$, the semiring of formal sums of elements in $\cS$.
\end{theorem}

Point~1 of Theorem \ref{thm.prescribed-embeddings-semiring} was already stated in \cite{Dep13}. There is however a gap in the proof of \cite[Lemma 6.2]{Dep13}. In Remark \ref{rem.gap}, we provide a more detailed discussion on the relation between our results and \cite{Dep13}.

For point~2, recall that a countable structure $\cG$ consists of a countable set $I$ together with a countable family of subsets $R_k \subseteq I^{n_k}$, $n_k \in \N$. Given such a countable structure, its semigroup of self-embeddings, denoted $\Emb(\cG)$, is defined as the set of all injective maps $\vphi : I \recht I$ with the property that for every $k$, we have
$$(\vphi(i_1),\ldots,\vphi(i_{n_k})) \in R_k \quad\text{if and only if}\quad (i_1,\ldots,i_{n_k}) \in R_k \; .$$
Note that this semigroup can be smaller than the semigroup of monomorphisms, where only the implication $\Leftarrow$ would be required.

For a countable group $\Gamma$, we get the semigroup of injective group homomorphisms $\Gamma \recht \Gamma$. For a countable field $K$, we get the semigroup of injective field homomorphisms $K \recht K$. For a countable graph $(V,E)$, we get the semigroup of graph embeddings, i.e.\ injective maps $\vphi : V \recht V$ such that $x,y$ are joined by an edge if and only if $\vphi(x),\vphi(y)$ are joined by an edge.

Of course, Theorem \ref{thm.prescribed-embeddings-semiring} provides new classes of Polish groups that can be realized as the outer automorphism group $\Out(M)$ of a full II$_1$ factor $M$ and actually is the first systematic and explicit realization result for non locally compact outer automorphism groups. For completeness, we explicitly state this as a corollary.

Our method provides in particular a new method to realize countable groups, as well as compact groups, as the outer automorphism group of a II$_1$ factor with separable predual. This was first proven in \cite{IPP05} for compact abelian groups, in \cite[Theorem 7.12]{PV06} for finitely presented countable groups, in \cite[Theorem 2.13]{Vae07} for arbitrary countable groups and in \cite{FV07} for arbitrary compact groups. The construction in the proof of Theorem \ref{thm.prescribed-embeddings-semiring} is substantially simpler and deals with all these groups at once.

\begin{corollary}\label{cor.outer-automorphism-groups}
Each of the following groups arises as the outer automorphism group of a full II$_1$ factor with separable predual.
\begin{enumlist}
\item Permutation groups, i.e.\ closed subgroups of the group of all permutations of a countable set equipped with the topology of pointwise convergence.
\item The unitary groups $\cU(N)$ of von Neumann algebras with separable predual.
\item Second countable compact groups.
\item Second countable locally compact groups that are totally disconnected.
\item Direct sums of the groups in 1, 2, 3, 4.
\end{enumlist}
\end{corollary}

To make the connection with Theorem \ref{thm.prescribed-embeddings-semiring}, recall that the class of permutation groups in point~1 coincides with the class of automorphism groups of countable structures.

Before embarking on the proof of Theorem \ref{thm.prescribed-embeddings-semiring}, we first note that we can easily describe the embeddings semiring for the II$_1$ factors appearing in Theorem \ref{thm.complete-intervals-construction}. Using the same notation, assume that $L \subg K$ is any subgroup without nontrivial finite dimensional unitary representations. Consider the II$_1$ factor $N = L^\infty(X,\mu) \rtimes (\Gamma_0 \times \Gamma_L)$ as in Theorem \ref{thm.complete-intervals-construction}. The following result then follows immediately.

\begin{proposition}\label{prop.already-this}
Define $\cF_1 = \{g \in K \mid gLg^{-1} \subg L\}$ and denote by $\cF = L \backslash \cF_1$ the (well defined) quotient semigroup. There is a canonical isomorphism $\Embs(N) \cong \N[\cF]$.
\end{proposition}
\begin{proof}
In the proof of Theorem \ref{thm.complete-intervals-construction}, we described all $N$-$N$-bimodules $\bim{N}{H}{N}$ with finite right dimension as direct sums of embeddings induced by $g \in \cF_1$. Two such embeddings, induced by $g_1,g_2 \in \cF_1$, are unitarily conjugate if and only if $L g_1 = L g_2$.
\end{proof}

\begin{remark}
Already Proposition \ref{prop.already-this} can be used to realize several semirings $\N[\cS]$ as the embeddings semiring of a II$_1$ factor. Indeed, assume that $\cS$ is a countable unital semigroup that can be embedded into a group $T$ that is either amenable or a free product of amenable groups. Let $G$ be the group of finite even permutations of $\N$. In the construction of Proposition \ref{prop.already-this}, we can then take $K = G \ast T$ with the subgroup $L \subg K$ defined as the free product of the conjugates $h G h^{-1}$, $h \in \cS$. Note that $G \ast K = G \ast G \ast T$ belongs to the family $\cC$ of Definition \ref{def.family-C}. Since $G$ has no nontrivial finite dimensional unitary representations, also $L$ does not have them. So we can apply Proposition \ref{prop.already-this}. Define $\cF_1 = \{g \in K \mid gLg^{-1} \subset L\}$. Then, $\cF_1 = L \cS$ and $L\backslash \cF_1 = \cS$. We conclude that $\Embs(N) \cong \N[\cS]$.
\end{remark}

To prove Theorem \ref{thm.prescribed-embeddings-semiring}, we need more flexibility and proceed in two steps. Assume that $(A_0,\tau)$ is an amenable tracial von Neumann algebra and let $\cG \subg \Aut(A_0,\tau)$ be any closed subgroup of the Polish group of trace preserving automorphisms. We then define the semigroup $\End(A_0,\tau,\cG)$ of all unital normal trace preserving $*$-homomorphisms $\psi : A_0 \recht A_0$ that commute with $\cG$, i.e.\ $\theta \circ \psi = \psi \circ \theta$ for all $\theta \in \cG$. We similarly denote by $\Aut(A_0,\tau,\cG)$ the subgroup of invertible elements, i.e.\ the trace preserving $*$-automorphisms $\psi \in \Aut(A_0,\tau)$ that commute with $\cG$.

As a first step, we prove in Proposition \ref{prop.the-factor-semiring} that an appropriate modification of the construction in Theorem \ref{thm.all-embed-trivial} allows to realize $\N[\End(A_0,\tau,\cG)]$ as the embeddings semiring of a full II$_1$ factor $M$, with $\Aut(A_0,\tau,\cG)$ corresponding to its outer automorphism group $\Out(M)$. Such a result was actually already proven in \cite[Theorem 8.5]{Dep10} for $\Out(M)$, and in \cite[Theorem 4.1]{Dep13} for $\Embs(M)$. For completeness, we nevertheless include our construction here, since it is simpler and a direct consequence of the results in Section \ref{sec.embeddings-bernoulli}.

As a second and new step, we provide in Propositions \ref{prop.this-is-all}, \ref{prop.all-closed-perm} and \ref{prop.bicentralizer-gaussian-von-neumann} several results on which semigroups arise as $\End(A_0,\tau,\cG)$, leading to the proof of Theorem \ref{thm.prescribed-embeddings-semiring} and Corollary \ref{cor.outer-automorphism-groups}.

\begin{proposition}\label{prop.the-factor-semiring}
Let $(A_0,\tau)$ be an amenable tracial von Neumann algebra with separable predual. Let $\cG \subg \Aut(A_0,\tau)$ be a closed subgroup and write $\cS = \End(A_0,\tau,\cG)$. Then, the semiring $\N[\cS]$ arises as the embeddings semiring $\Embs(M)$ of a full II$_1$ factor $M$ with separable predual. We have $\Out(M) \cong \Aut(A_0,\tau,\cG)$ as Polish groups.
\end{proposition}
\begin{proof}
Let $\Gamma_0 \times \Gamma \actson (X,\mu)$ be the group action defined in Theorem \ref{thm.all-embed-trivial}. In that context, $\Gamma_0 = G_1 \ast G \subg G \ast G = \Gamma$. Let $c \in G_1$ and $d \in G$ be elements of order $5$. View $c$ in the first free product factor and view $d$ in de second free product factor. Define $\Lambda_1 \subg G_1$ as $\Lambda_1 = \langle c \rangle \ast \langle d \rangle \cong \Z/5\Z \ast \Z / 5\Z$.
Choose $\F_\infty \cong \Lambda_0 \subg \Lambda_1$ such that $\Lambda_0 \subg \Lambda_1$ is malnormal: $g \Lambda_0 g^{-1} \cap \Lambda_0 = \{e\}$ for all $g \in \Lambda_1 \setminus \Lambda_0$. Denote by $\Delta : \Gamma_0 \recht \Gamma_0 \times \Gamma_0$ the diagonal embedding.

We prove the following properties of $\Lambda_0 \subg \Gamma$.
\begin{enumlist}
\item For all $g \in \Gamma$ and $(a,b) \in \cA \times \cB$, we have that $g \Lambda_0 g^{-1} \cap \Lambda_{a,b} = \{e\}$.
\item For all $g \in \Gamma \setminus \Lambda_0$, we have that $g \Lambda_0 g^{-1} \cap \Lambda_0 = \{e\}$.
\item If $\delta : \Gamma_0 \recht \Gamma$ is an injective group homomorphism and $(a,b) \in \cA \times \cB$, we have that $\delta(\Lambda_{a,b}) \not\prec_{\Gamma} \Lambda_0$.
\end{enumlist}
To prove 1, note that as in Lemma \ref{lem.playing-with-words}, we have that $g \Lambda_1 g^{-1} \cap \Lambda_{a,b} = \{e\}$ for  all $(a,b) \in \cA \times \cB$ and $g \in \Gamma$.
Similarly, $g \Lambda_1 g^{-1} \cap \Lambda_1$ is finite for every $g \in \Gamma \setminus \Lambda_1$. Since $\Lambda_0 \subg \Lambda_1$ and $\Lambda_0$ is torsion free, we get that $g \Lambda_1 g^{-1} \cap \Lambda_1 = \{e\}$ for every $g \in \Gamma \setminus \Lambda_1$. Since $\Lambda_0 \subg \Lambda_1$ is malnormal, statement~2 follows. To prove 3, let $\delta : \Gamma_0 \recht \Gamma$ be an injective group homomorphism and $(a,b) \in \cA \times \cB$. By the Kurosh theorem, we can compose $\delta$ with an inner automorphism of $\Gamma$ such that $\delta(G_1 \ast e)$ is a subgroup of either $G_1 \ast e$ or $e \ast G$, and such that $\delta(e \ast G)$ can be conjugated into either $G_1 \ast e$ or $e \ast G$. So, we may assume that $\delta(a) = a_1$ and $\delta(b) = w b_1 w^{-1}$ where $a_1,b_1$ are elements of order $2$ and $3$ in either $G_1 \ast e$ or $e \ast G$, and $w \in \Gamma$. As in Lemma \ref{lem.playing-with-words}, because of incompatible orders, no infinite subgroup of $\langle a_1 \rangle \ast w \langle b_1 \rangle w^{-1}$ can be conjugated into $\Lambda_1 = \langle c \rangle \ast \langle d \rangle$. A fortiori, statement~3 holds.

Choose a group homomorphism $\al : \Lambda_0 \cong \F_\infty \recht \cG$ such that $\al(\Lambda_0) \subg \cG$ is dense and such that $\Ker \al$ is infinite. We can view $\al$ as an action $\Lambda_0 \actson^\al (A_0,\tau)$. We get the coinduced action
$$\Gamma_0 \times \Gamma \actson (A_1,\tau) = (A_0,\tau_0)^{(\Gamma_0 \times \Gamma)/\Delta(\Lambda_0)} \; .$$
Using the diagonal action, we define
$$M = (A_1 \ovt L^\infty(X,\mu)) \rtimes (\Gamma_0 \times \Gamma) \; .$$
We prove that the embeddings semiring $\Embs(M)$ is given by $\N[\End(A_0,\tau,\cG)]$. To prove this statement, we have to repeat the proof of Lemma \ref{lem.all-embedding-trivial-general} in this slightly broader context. Write $A = A_1 \ovt L^\infty(X,\mu)$.

Let $d > 0$ and $\theta : M \recht M^d$ a normal unital $*$-homomorphism. As in the proof of Lemma \ref{lem.all-embedding-trivial-general}, the crossed product II$_1$ factor $M$ fits into the framework of Section \ref{sec.embeddings-bernoulli}. Thus, by Lemma \ref{lem.step1}, we find that $d \in \N$ and that $\theta$ is a direct sum of embeddings of a special form. It suffices to analyze each of these direct summands separately and may thus assume that
$$\theta : M \recht M_d(\C) \ot M \quad\text{with}\quad \theta(A) \subseteq M_d(\C) \ot A \quad\text{and}\quad \theta(u_{(g,h)}) = 1 \ot u_{\pi(g,h)} \; ,$$
where $\pi : \Gamma_0 \times \Gamma \recht \Gamma_0 \times \Gamma$ is an injective group homomorphism that is either of the form $\pi(g,h) = (\eta(g),\delta(h))$ or $\pi(g,h) = (\delta(h),\eta(g))$.

Fix $(a,b) \in \cA \times \cB$. By property~3 above, we have $\pi(\Delta(\Lambda_{a,b})) \not\prec_{\Gamma_0 \times \Gamma} \Lambda_0$. As in the proof of Lemma \ref{lem.all-embedding-trivial-general}, we then find that $\pi(\Delta(\Lambda_{a,b})) \prec_{\Gamma_0 \times \Gamma} \Lambda_{a,b}$ for all $(a,b) \in \cA \times \cB$. By Lemma \ref{lem.first-group-lemma}, after a conjugacy, we may assume that $\pi(g,h) = (g,h)$ for all $(g,h) \in \Gamma_0 \times \Gamma$. As in Lemma \ref{lem.all-embedding-trivial-general}, it then also follows that $\theta(1 \ot b) = 1 \ot (1 \ot b)$ for all $b \in L^\infty(X,\mu)$.

Identify $A_0$ with its copy in $A_1$ sitting in the coordinate position $(e,e)\Delta(\Lambda_0)$. Properties 1 and 2 above imply that $A \cap L(\Delta(\Ker \al))' = A_0 \ot 1$. Therefore, $\theta(a \ot 1) = \psi(a) \ot 1$ where $\psi : A_0 \recht M_d(\C) \ot A_0$. Let $g \in \Gamma \setminus \{e\}$. Since $\theta(A_0 \ot 1)$ commutes with $\theta(u_{(e,g)}(A_0 \ot 1)u_{(e,g)}^*)$, we find that $\psi(A_0) \subseteq D \ot A_0$ where $D \subset M_d(\C)$ is an abelian von Neumann subalgebra. It follows that $\theta(M) \subseteq D \ot M$. We can thus further decompose $\theta$ as a direct sum of embeddings $\theta : M \recht M$ such that
\begin{equation}\label{eq.canonical}
\begin{split}
& \theta(a \ot 1) = \psi(a) \ot 1 \;\;\text{for all $a \in A_0$,}\;\; \theta(1 \ot b) = 1 \ot b \;\;\text{for all $b \in L^\infty(X)$,}\\
& \theta(u_{(g,h)}) = u_{(g,h)} \;\;\text{for all $(g,h) \in \Gamma_0 \times \Gamma$,}
\end{split}
\end{equation}
where $\psi : A_0 \recht A_0$ is a trace preserving inclusion. We find that $\psi \in \End(A_0,\tau,\cG)$. Conversely, whenever $\psi \in \End(A_0,\tau,\cG)$, there is a unique embedding $\theta : M \recht M$ satisfying \eqref{eq.canonical}. Since the relative commutant of $L(\Gamma_0 \times \Gamma)$ in $M$ is trivial, distinct elements of $\End(A_0,\tau,\cG)$ give rise to embeddings that are not unitarily conjugate. We also have that the relative commutant of $L(\Gamma_0 \times \Gamma)$ in the ultrapower $M^\omega$ is trivial, so that $M$ is full. This concludes the proof of the proposition.
\end{proof}

The following lemma is the key technical result to realize concrete semigroups as $\End(A_0,\tau,\cG)$. It was already stated as \cite[Step 1 of Lemma 6.2]{Dep13}, but the proof there has a gap that we repair here. We refer to Remark \ref{rem.gap} for a more detailed discussion.

\begin{lemma}\label{lem.factor-product}
Let $(X_0,\mu_0)$ be a standard nonatomic probability space and put $(X,\mu) = (X_0,\mu_0)^\N$. For every measure preserving automorphism $\Delta \in \Aut(X_0,\mu_0)$, consider the diagonal automorphism
$$\be_\Delta \in \Aut(X,\mu) : (\be_\Delta(x))_n = \Delta(x_n) \; .$$
Let $\psi : (X,\mu) \recht (X,\mu)$ be a nonsingular factor map satisfying $\psi \circ \be_\Delta = \be_\Delta \circ \psi$ for all $\Delta  \in \Aut(X_0,\mu_0)$. Then $\psi$ is measure preserving and there exists an injective map $\sigma : \N \recht \N$ such that $(\psi(x))_n = x_{\si(n)}$ for all $n \in \N$ and a.e.\ $x \in X$.
\end{lemma}
\begin{proof}
Choosing a mixing probability measure preserving action $\Lambda \actson (X_0,\mu_0)$, its diagonal product $\Lambda \actson (X,\mu)$ is ergodic and measure preserving. It follows that $\psi_*(\mu)$ is $\Lambda$-invariant and hence equal to $\mu$. So, $\psi$ is measure preserving.

For notational convenience, we write $A_0 = L^\infty(X_0,\mu_0)$ and view $(A,\tau)$ as the infinite tensor product $(A_0,\tau)^\N$. We denote by $\pi_n : A_0 \recht A$ the embedding in the $n$'th tensor factor. Whenever $p \in A_0$ is a projection, write $A_0(p) = \C p + \C (1-p)$. We view $\psi$ as a trace preserving unital $*$-homomorphism $\psi : A \recht A$ that commutes with all $\be_\Delta$.

{\bf Claim 1.} For every projection $p \in A_0$, we have that $\psi\bigl((A_0(p))^\N\bigr) \subseteq (A_0(p))^\N$.

When $p$ equals $0$ or $1$, there is nothing to prove. So assume that $0 < p < 1$. Choose a trace preserving action $\Lambda \actson (A_0,\tau)$ such that $p$ is invariant and such that the actions $\Lambda \actson A_0 p$ and $\Lambda \actson A_0(1-p)$ are mixing. This can be done by taking the disjoint union of two Bernoulli shifts. Then $(A_0(p))^\N$ is precisely the fixed point algebra under the diagonal action $\Lambda \actson A$. Since $\psi$ commutes with this diagonal action, claim~1 follows.

{\bf Claim 2.} For every $n,m \in \N$ and every $a \in A_0$, we have that
$$\psi(\pi_n(a)) \in \pi_m(\C 1 + \C a) \, A_0^{\N \setminus \{m\}} \; .$$
Write $A_1 = A_0^{\N \setminus \{m\}}$ and view $A$ as the tensor product $A = \pi_m(A_0) \ovt A_1$. Let $\om$ be an arbitrary normal state on $A_1$. Define the unital normal positive map
$$\theta : A_0 \recht A_0 : \theta(a) = \pi_m^{-1}((\id \ot \om)\psi(\pi_n(a))) \; .$$
It suffices to prove that $\theta(a) \in \C 1 + \C a$ for all $a \in A_0$.

Choose an increasing sequence of finite dimensional unital $*$-subalgebras $B_k \subset A_0$ such that $\bigcup_k B_k$ is weakly dense in $A_0$. Fix $k$ and denote by $p_1,\ldots,p_N$ the minimal projections of $B_k$. From claim~1, we know that $\theta(p_i) \in \C p_i + \C (1-p_i)$ for all $i$. We can thus write
$$\theta(p_i) = \al_i \, p_i + \be_i \, (1-p_i) \quad\text{with $\al_i,\be_i \in [0,1]$.}$$
Define $\eta_k \in (B_k)^*_+$ by $\eta_k(p_i) = \be_i$. Define $b_k \in B_k$ by $b_k = \sum_i (\al_i - \be_i) p_i$. Note that $-1 \leq b_k \leq 1$. Also,
$$\theta(a) = b_k a + \eta_k(a) 1 \quad\text{for all $a \in B_k$.}$$
Since $\theta$ is unital, we have in particular that $\eta_k(1) 1 = 1 - b_k$. Therefore, $\eta_k(1) \leq 2$ and $b_k = (1-\eta_k(1)) \, 1$. Since $\eta_k$ is positive, $\|\eta_k\| \leq 2$. We conclude that
\begin{equation}\label{eq.theta-on-Bk}
\theta(a) = \eta_k(a) 1 + (1-\eta_k(1)) a \quad\text{for all $a \in B_k$.}
\end{equation}
Let $a \in A_0$ be an arbitrary element. By the Kaplansky density theorem, we can choose $k_n \recht \infty$ and $a_n \in B_{k_n}$ such that $\|a_n\| \leq \|a\|$ for all $n$ and $a_n \recht a$ strongly. After passage to a subsequence, we may assume that the bounded sequences $\eta_{k_n}(a_n)$ and $\eta_{k_n}(1)$ are convergent, say to $\al,\be \in \R$. It follows from \eqref{eq.theta-on-Bk} that $\theta(a_n) \recht \al 1 + \be a$ strongly. On the other hand, $\theta(a_n) \recht \theta(a)$ strongly, by normality of $\theta$. So, claim~2 is proven.

It follows from claim~2 that for all $n \in \N$, $\cF \subset \N$ finite and $a \in A_0$
\begin{equation}\label{eq.better-claim}
\psi(\pi_n(a)) \in (\C 1 + \C a)^{\otimes \cF} \, A_0^{\N \setminus \cF} \; .
\end{equation}

Fix $n \in \N$. We prove that there exists an $m \in \N$ such that $\psi(\pi_n(a)) = \pi_m(a)$ for all $a \in A_0$. Fix a generating Haar unitary $u \in (A_0,\tau)$.

For every finite subset $\cF \subset \N$ and every unitary $v \in A_0$, write
$$\pi_\cF(v) = \prod_{m \in \cF} \pi_m(v) \; .$$
We use the convention that $\pi_\emptyset(v) = 1$. For every unitary $v \in A_0$ with $\tau(v) = 0$, we denote by $H(v) \subset L^2(A,\tau)$ the $\|\,\cdot\,\|_2$-closed linear span of $\{\pi_\cF(v) \mid \cF \subset \N \;\text{finite}\}$. Note that the vectors $\pi_\cF(v)$, $\cF \subset \N$ finite, form an orthonormal basis of $H(v)$. For every finite subset $\cF \subset \N$, denote by $E_\cF$ the unique trace preserving conditional expectation of $A$ onto $A_0^\cF$. By \eqref{eq.better-claim}, we have that $E_\cF(\psi(\pi_n(v))) \in H(v)$ for every $v \in \cU(A_0)$ with $\tau(v)=0$, and every finite $\cF \subset \N$. Letting $\cF \recht \N$, we have $E_\cF(\psi(\pi_n(v))) \recht \psi(\pi_n(v))$ in $\|\,\cdot\,\|_2$, and we conclude that $\psi(\pi_n(v)) \in H(v)$ for every unitary $v \in A_0$ with $\tau(v) = 0$.

Applying this to our fixed Haar unitary $u \in A_0$, we can uniquely write
$$\psi(\pi_n(u)) = \sum_{\cF \subset \N \, \text{finite}} \al_\cF \, \pi_\cF(u) \; ,$$
with convergence in $\|\,\cdot\,\|_2$. We now use that $\psi(\pi_n(u^2)) = \psi(\pi_n(u))^2$. When $\cF \neq \cG$ are different finite subsets of $\N$, we have $\pi_\cF(u) \, \pi_\cG(u) \perp H(u^2)$. Since $\psi(\pi_n(u^2)) \in H(u^2)$, it follows that
$$\psi(\pi_n(u^2)) = P_{H(u^2)}\bigl(\psi(\pi_n(u^2))\bigr) = P_{H(u^2)}\bigl(\psi(\pi_n(u)) \, \psi(\pi_n(u))\bigr) = \sum_{\cF \subset \N \; \text{finite}} \al_\cF^2 \, \pi_\cF(u^2) \; .$$
Since $\psi$ is trace preserving, both $\psi(\pi_n(u))$ and $\psi(\pi_n(u^2))$ have $\|\,\cdot\,\|_2$ equal to $1$. We conclude that
$$\sum_{\cF \subset \N \;\text{finite}} |\al_\cF|^2 = 1 = \sum_{\cF \subset \N \;\text{finite}} |\al_\cF|^4 \; .$$
So there is precisely one finite subset $\cF \subset \N$ with $|\al_\cF|=1$ and $\al_\cG = 0$ for all $\cG \neq \cF$.

If $\cF = \emptyset$, we have $\psi(\pi_n(u)) \in \C 1$ and thus $\psi(\pi_n(A_0)) \subseteq \C 1$, which is absurd because $\psi(\pi_n(A_0))$ is a diffuse von Neumann algebra. If $\cF$ has two or more elements, write $a = u + u^2$. Then, $\psi(\pi_n(a)) = \al_\cF \pi_\cF(u) + \al_\cF^2 \pi_\cF(u^2)$. Take $m \in \cF$. Since $\cF \setminus \{m\}$ is nonempty, it follows that the linear span of $(\id \ot \om)\psi(\pi_n(a))$, $\om \in (A_0^{\N \setminus \{m\}})_*$, contains both $\pi_m(u)$ and $\pi_m(u^2)$. By claim~2, this linear span should be contained in $\C 1 + \C \pi_m(a)$. We reached a contradiction and conclude that $\cF = \{m\}$. Write $\al = \al_\cF$. It then follows from claim~2 that $\al \pi_m(u) + \al^2 \pi_m(u^2)$ belongs to $\C1 + \C \pi_m(u+u^2)$. This implies that $\al = 1$. So, $\psi(\pi_n(u)) = \pi_m(u)$. It follows that $\psi(\pi_n(a)) = \pi_m(a)$ for all $a \in A_0$.

Obviously, this $m \in \N$ is unique. We denote $m = \si(n)$. Since $\psi$ is an injective homomorphism, $\si : \N \recht \N$ must be an injective map. We have proven that $\psi(\pi_n(a)) = \pi_{\si(n)}(a)$ for all $n \in \N$ and $a \in A_0$. This concludes the proof of the lemma.
\end{proof}

\begin{definition}\label{def.bicentralizer}
Let $(A_0,\tau)$ be an amenable tracial von Neumann algebra with separable predual. Let $\End(A_0,\tau)$ be the semigroup of unital trace preserving $*$-homomorphisms $A_0 \recht A_0$. We say that a unital subsemigroup $\cS \subseteq \End(A_0,\tau)$ has the \emph{bicentralizer property} if the following two properties hold.
\begin{enumlist}
\item The (automorphic) centralizer $C(\cS) = \{\be \in \Aut(A_0,\tau) \mid \al \circ \be = \be \circ \al \;\;\text{for all $\al \in \cS$}\}$ acts ergodically on $(A_0,\tau)$.
\item The bicentralizer equals $\cS$~: if $\psi \in \End(A_0,\tau)$ and $\psi \circ \be = \be \circ \psi$ for all $\be \in C(\cS)$, then $\psi \in \cS$.
\end{enumlist}
\end{definition}

Clearly, a subsemigroup with the bicentralizer property must be closed in the usual Polish topology on $\End(A_0,\tau)$. Also, if $\cS \subseteq \End(A_0,\tau)$ has the bicentralizer property, then automatically $\cS \cap \Aut(A_0,\tau)$ equals the group $\cS_\inv$ of invertible elements of $\cS$ and with the notation of Definition \ref{def.bicentralizer}, $C(C(\cS_\inv)) = \cS_\inv$. Finally note that if $\cS \subseteq \End(A_0,\tau)$ has the bicentralizer property and $\psi : A_0 \recht A_0$ is any unital, normal, injective $*$-homomorphism that commutes with $C(\cS)$, the ergodicity of $C(\cS) \actson (A_0,\tau)$ implies that $\psi$ is automatically trace preserving and thus, $\psi \in \cS$.

When $A_0 = L^\infty(X,\mu)$ is abelian, we identify $\End(A_0,\tau)$ with the semigroup $\Factor(X,\mu)$ of measure preserving factor maps $(X,\mu) \recht (X,\mu)$, even though this identification is an anti-isomorphism.

It is easy to realize any unital, left cancellative semigroup $\cS$ as the semigroup of self-embeddings of a countable structure. Therefore, point~1 of Theorem \ref{thm.prescribed-embeddings-semiring} is a special case of point~2. But in this special case, we can give the following more straightforward construction, which has an independent interest (see Remark \ref{rem.double-centralizer-Bernoulli}).

\begin{proposition}\label{prop.this-is-all}
Let $\cS$ be a countable, unital, left cancellative semigroup. Let $(X_0,\mu_0)$ be a standard nonatomic probability space. Define $(X,\mu) = (X_0,\mu_0)^\cS$. For every $g \in \cS$, define the measure preserving factor map $\psi_g : (X,\mu) \recht (X,\mu) : (\psi_g(x))_h = x_{gh}$ for all $g,h \in \cS$ and a.e.\ $x \in X$.

Then, the semigroup $\{\psi_g \mid g \in \cS\} \subg \Factor(X,\mu)$ has the bicentralizer property (in the sense of Definition \ref{def.bicentralizer}).
\end{proposition}

Note that we need the left cancellation property to ensure that $\psi_g$ is a well defined factor map.

\begin{remark}\label{rem.double-centralizer-Bernoulli}
It follows in particular from Proposition \ref{prop.this-is-all} that the bicentralizer of the Bernoulli action $\Gamma \subg \Aut((X_0,\mu_0)^\Gamma)$ equals $\Gamma$ whenever $(X_0,\mu_0)$ is a standard \emph{nonatomic} probability space and $\Gamma$ is any countable group. In \cite{Rud77}, this bicentralizer property is proven for the Bernoulli action of $\Gamma = \Z$ and an arbitrary, \emph{possibly atomic}, standard base space $(X_0,\mu_0)$. We do not know whether this property holds for arbitrary countable groups $\Gamma$.
\end{remark}

\begin{proof}[{Proof of Proposition \ref{prop.this-is-all}}]
Fix a measure preserving factor map $\psi : (X,\mu) \recht (X,\mu)$ that commutes with $C(\cS)$. For every measure preserving automorphism $\Delta \in \Aut(X_0,\mu_0)$, consider the diagonal automorphism
$$\be_\Delta \in \Aut(X,\mu) : (\be_\Delta(x))_n = \Delta(x_n) \; .$$
Note that $\be_\Delta \in C(\cS)$. Since we can choose a mixing $\Delta$, it follows that $C(\cS)$ acts ergodically on $(X,\mu)$. By Lemma \ref{lem.factor-product}, we can take an injective map $\sigma : \cS \recht \cS$ such that $\psi(x)_h = x_{\sigma(h)}$ for all $h \in \cS$ and a.e.\ $x \in X$. Denote by $e \in \cS$ the unit element. It remains to prove that $\si(h) = \si(e) h$ for all $h \in \cS$, since it then follows that $\psi = \psi_g$ with $g = \si(e)$.

To prove this, we copy the proof of \cite[Step 2 in Lemma 6.2]{Dep13}. Fix a Borel set $\cU \subset X_0$ with $0 < \mu_0(\cU) < 1$. Fix a measure preserving automorphism $\Delta \in \Aut(X_0,\mu_0)$ that is free, i.e.\ $\Delta(x) \neq x$ for all $x \in X_0$, and that satisfies $\Delta(\cU) = \cU$. Fix $k \in \cS$. Define the measure preserving automorphism $\be \in \Aut(X,\mu)$ by
$$\be : X \recht X : (\be(x))_h = \begin{cases} \Delta(x_h) &\;\;\text{if $x_{hk} \in \cU$,}\\ x_h &\;\;\text{if $x_{hk} \not\in \cU$.}\end{cases}$$
We have $\be \in C(\cS)$. Therefore, $\psi$ commutes with $\be$. We have
$$(\be(\psi(x)))_h = \begin{cases} \Delta(x_{\si(h)}) &\;\;\text{if $x_{\si(hk)} \in \cU$,}\\ x_{\si(h)} &\;\;\text{if $x_{\si(hk)} \not\in \cU$,}\end{cases}\quad\text{and}\quad
(\psi(\be(x)))_h = \begin{cases} \Delta(x_{\si(h)}) &\;\;\text{if $x_{\si(h)k} \in \cU$,}\\ x_{\si(h)} &\;\;\text{if $x_{\si(h)k} \not\in \cU$.}\end{cases}$$
If for some $h \in \cS$, we have $\si(hk) \neq \si(h) k$, the set $\cV = \{x \in X \mid x_{\si(hk)} \in \cU , x_{\si(h)k} \not\in \cU\}$ has positive measure. By the freeness of $\Delta$, we have $\be(\psi(x)) \neq \psi(\be(x))$ for all $x \in \cV$. So, we conclude that $\si(hk) = \si(h) k$ for all $h,k \in \cS$. Writing $g = \si(e)$, this means that $\si(k) = g k$ for all $k \in \cS$.
\end{proof}

\begin{remark}\label{rem.gap}
We use the notation of Proposition \ref{prop.the-factor-semiring}. As mentioned above, in \cite[Theorem 4.1]{Dep13}, it was already proven that $\N[\End(A_0,\tau_0,\cG)]$ arises as the embeddings semiring $\Embs(M)$ of a II$_1$ factor $M$, whenever
$\cG$ is an ergodic group of measure preserving transformations of an abelian $(A_0,\tau)$. This is very similar to our result in Proposition \ref{prop.the-factor-semiring}, although we do not require $A_0$ to be abelian or the action $\cG \actson A_0$ to be ergodic.

Next, the statement of \cite[Lemma 6.2]{Dep13} is identical to the statement of our Proposition \ref{prop.this-is-all}. The proof of \cite[Lemma 6.2]{Dep13} consists of two steps. In step~1, it is implicitly used that $(\{0,1\},\mu_0)^\N$ is atomic when $\mu_0$ is a nontrivial probability measure on $\{0,1\}$, which is of course false. We repair this step~1 here in Lemma \ref{lem.factor-product}. Step~2 of the proof of \cite[Lemma 6.2]{Dep13} is correct and is repeated here in our proof of Proposition \ref{prop.this-is-all}.
\end{remark}

Next assume that $\cG$ is a countable structure, given by a countable set $I$ and a countable family of subsets $R_k \subseteq I^{n_k}$, $k \geq 1$. Let $J$ be the disjoint union of $I$ and the sets $I^{n_k}$. Let $\Emb(\cG)$ act on $J$ by defining $g \cdot (i_1,\ldots,i_{n_k}) = (g \cdot i_1,\ldots,g \cdot i_{n_k})$ for all $g \in \Emb(\cG)$ and $(i_1,\ldots,i_{n_k}) \in I^{n_k}$. We identify in this way $\Emb(\cG)$ with a closed subsemigroup of $\Inj(J)$, the injective maps from $J$ to $J$.

\begin{proposition}\label{prop.all-closed-perm}
Let $\cG$ be a countable structure and write $\cS = \Emb(\cG)$. Define $\cS \subseteq \Inj(J)$ as above. Let $(X_0,\mu_0)$ be a standard nonatomic probability space. Define $(X,\mu) = (X_0,\mu_0)^J$. For every $g \in \cS$, define the measure preserving factor map $\psi_g : (X,\mu) \recht (X,\mu) : (\psi_g(x))_j = x_{g\cdot j}$ for all $g \in \cS$, $j \in J$ and a.e.\ $x \in X$.

Then, the semigroup $\{\psi_g \mid g \in \cS\} \subg \Factor(X,\mu)$ has the bicentralizer property (in the sense of Definition \ref{def.bicentralizer}).
\end{proposition}

\begin{proof}
For every $k \in \N$, write $J_k = I^{n_k}$ and write $J_0 = I$. We then view $J$ as the disjoint union of the subsets $J_k$, $k \geq 0$. For every $\Delta \in \Aut(X_0,\mu_0)$ and every $k \geq 0$, consider $\beta_{\Delta,k} \in \Aut(X,\mu)$ defined by
$$\bigl(\beta_{\Delta,k}(x)\bigr)_j = \begin{cases} x_j &\;\;\text{if $j \in J \setminus J_k$,}\\
\Delta(x_j) &\;\;\text{if $j \in J_k$.}\end{cases}$$
By construction, $\beta_{\Delta,k} \in C(\cS)$ for all $\Delta \in \Aut(X_0,\mu_0)$ and $k \geq 0$. Using a mixing $\Delta$ and varying $k$, it already follows that $C(\cS)$ acts ergodically.

Let $\psi : (X,\mu) \recht (X,\mu)$ be a measure preserving factor map that commutes with $C(\cS)$. For every $k \geq 0$, denote by $\pi_k : (X,\mu) \recht (X_0,\mu_0)^{J_k}$ the natural factor map. Then $\pi_k \circ \psi$ is invariant under $\beta_{\Delta,m}$ for all $\Delta$ and $m \neq k$. By ergodicity, it follows that $\pi_k \circ \psi = \psi_k \circ \pi_k$ where $\psi_k : (X_0,\mu_0)^{J_k} \recht (X_0,\mu_0)^{J_k}$ is a measure preserving factor map that commutes with all $\beta_{\Delta,k}$. It then follows from Lemma \ref{lem.factor-product} that there exist injective maps $\si_k : J_k \recht J_k$ such that
\begin{equation}\label{eq.nice-form}
(\psi(x))_j = x_{\si_k(j)} \quad\text{for all $k \in \N$, $j \in J_k$ and a.e.\ $x \in X$.}
\end{equation}

We write $\si = \si_0$, which is an injective map from $I$ to $I$. We will prove that $\si_k(i_1,\ldots,i_{n_k}) = (\si(i_1),\ldots,\si(i_{n_k}))$ for all $k \geq 1$ and $(i_1,\ldots,i_{n_k}) \in J_k$. To prove this statement, we may assume that $X_0 = \T$ and that $\mu_0$ is the Lebesgue measure on $\T$. Fix $k \geq 1$ and fix $1 \leq l \leq n_k$. For every $i \in J_k = I^{n_k}$, denote by $i_l \in I$ its $l$'th coordinate and recall that $I = J_0$. We will prove that $(\si_k(i))_l = \si(i_l)$ for all $i \in J_k$.

Define $\gamma_{l,k} \in \Aut(X,\mu)$ by
$$(\gamma_{l,k}(x))_i = \begin{cases} x_i &\;\;\text{if $i \in J \setminus J_k$,}\\
x_{i_l} x_i &\;\;\text{if $i \in J_k$, with $l$'th coordinate $i_l \in I = J_0$.}\end{cases}$$
By construction, $\gamma_{l,k} \in C(\cS)$. So, $\psi$ commutes with $\gamma_{l,k}$. Using \eqref{eq.nice-form}, it follows that
$$x_{\si(i_l)} \, x_{\si_k(i)} = x_{(\si_k(i))_l} \, x_{\si_k(i)} \quad\text{for a.e.\ $x \in X$.}$$
Hence, $\si(i_l) = (\si_k(i))_l$. This holds for all $k \geq 1$, $1 \leq l \leq n_k$ and $i \in J_k$. We have thus proven that $\si_k(i_1,\ldots,i_{n_k}) = (\si(i_1),\ldots,\si(i_{n_k}))$.

Recall that the countable structure $\cG$ is defined by the sets $R_k \subseteq I^{n_k} = J_k$. For every $\Delta \in \Aut(X_0,\mu_0)$ and every $k \geq 1$, define $\eta_{\Delta,k} \in \Aut(X,\mu)$ by
$$(\eta_{\Delta,k}(x))_i = \begin{cases} x_i &\;\;\text{if $i \in J \setminus J_k$ or $i \in J_k \setminus R_k$,}\\
\Delta(x_i) &\;\;\text{if $i \in R_k$.}\end{cases}$$
By construction, $\eta_{\Delta,k} \in C(\cS)$. Expressing that $\psi$ commutes with $\eta_{\Delta,k}$ for all $\Delta$ and $k$ says that $(\si(i_1),\ldots,\si(i_{n_k})) \in R_k$ if and only if $(i_1,\ldots,i_{n_k}) \in R_k$. We have thus proven that $\si \in \Emb(\cG)$, meaning that $\psi = \psi_g$ for some $g \in \cS$.
\end{proof}

For the formulation of the next result, we make use of Gaussian probability spaces. To every real Hilbert space $H_\R$ is associated a canonical standard probability space $(X,\mu)$ and a generating family of random variables $X \recht \R : x \mapsto \langle x , \xi \rangle$ for every $\xi \in H_\R$ with a Gaussian distribution with mean zero and variance $\|\xi\|^2$, such that $\xi \mapsto \langle \, \cdot \, , \xi \rangle$ is real linear.

We thus have the Gaussian Hilbert space $H_\R \subset L^2_\R(X,\mu)$. The construction is functorial, in the sense that to any real linear isometric $v : H_\R \recht H_\R$ corresponds an essentially unique measure preserving factor map $\psi_v : (X,\mu) \recht (X,\mu)$ such that $\langle \psi_v(x),\xi \rangle = \langle x, v \xi \rangle$ for all $\xi \in H_\R$ and a.e.\ $x \in X$.

The von Neumann bicommutant theorem can then be translated as follows to a bicentralizer property.

\begin{proposition}\label{prop.bicentralizer-gaussian-von-neumann}
Let $N \subseteq B(H)$ be a von Neumann algebra acting on a separable Hilbert space $H$. Assume that $p H$ has infinite dimension for every nonzero projection $p \in N$. Define $\cI(N) = \{v \in N \mid v^* v = 1\}$, the semigroup of isometries in $N$.

View $H$ as a real Hilbert space by forgetting multiplication by $i \in \C$ and using the scalar product $\Re \langle \, \cdot \, , \, \cdot \, \rangle$. Then the semigroup $\{\psi_v \mid v \in \cI(N)\} \subseteq \Factor(X,\mu)$ has the bicentralizer property in the sense of Definition \ref{def.bicentralizer}.
\end{proposition}

Of course, the only situation where $p H$ is finite dimensional arises when $N$ admits a type~I factor $B(K)$ as direct summand represented on a finite direct sum of $K$. We can always realize $p H$ to be infinite dimensional by choosing a representation of $N$ with infinite multiplicity.

\begin{proof}
Write $\cS = \{\psi_v \mid v \in \cI(N)\}$. Denote $P = N' \cap B(H)$. By construction, $\psi_u \in C(\cS)$ for every $u \in \cU(P)$. By our assumption, the representation of $\cU(P)$ on $H$ has no nonzero finite dimensional invariant subspace, because they would be of the form $p H$ for a nonzero projection $p \in N$. Thus, $C(\cS)$ acts ergodically on $(X,\mu)$.

Take $\psi \in \Factor(X,\mu)$ such that $\psi$ commutes with all $\psi_u$, $u \in \cU(P)$. We prove that $\psi \in \cS$. Define the real linear isometry
$$\theta : L^2_\R(X,\mu) \recht L^2_\R(X,\mu) : \theta(F) = F \circ \psi \; .$$
We write $H_\R$ instead of $H$ whenever we want to stress that we view $H$ as a real Hilbert space. Then, $\cU(H) \subset \cO(H_\R)$ and this is a proper inclusion. For every $u \in \cO(H_\R)$, we denote by $\theta_u$ the orthogonal transformation of $L^2_\R(X,\mu)$ given by $\theta_u(F) = F \circ \psi_u$.

The real Hilbert space $L^2_\R(X,\mu)$ has a canonical direct sum decomposition into $\R 1$ and $H_k \subset L^2_\R(X,\mu)$, $k \geq 1$, with $H_1 = H_\R$ and $H_k$ being the $k$-fold symmetric real tensor product of $H_\R$. For every $u \in \cO(H_\R)$, the subspaces $H_k$ are invariant subspaces of the orthogonal transformation $\theta_u$ and the restriction of $\theta_u$ to $H_k$ is given by the $k$-fold tensor product $u \ot \cdots \ot u$. So whenever $u_n \in \cO(H_\R)$ is a sequence that converges weakly to $t 1$ with $0 < t < 1$, it follows that $\theta_{u_n}(F) \recht t^k F$ weakly for every $F \in H_k$ and $k \geq 1$.

By our assumption, the von Neumann algebra $P$ does not admit a matrix algebra as direct summand. Therefore, $P$ contains a diffuse abelian von Neumann subalgebra. In particular, we can choose $u_n \in \cU(P)$ such that $u_n \recht t 1$ weakly, with $t = 1/2$. By the discussion in the previous paragraph, $\theta_{u_n}(F) \recht t^k F$ weakly for every $F \in H_k$ and $k \geq 1$. Define the real linear operator $T$ on $L^2_\R(X,\mu)$ by $T(1) = 1$ and $T(F) =t^k F$ for every $F \in H_k$ and $k \geq 1$. Since $\theta_{u_n} \recht T$ weakly and since $\theta$ commutes with all $\theta_{u_n}$, we find that $\theta$ commutes with $T$. Since $H_\R = H_1$ is the spectral subspace of $T$ with eigenvalue $t$, we conclude that $\theta(H_\R) \subseteq H_\R$. We thus find a real linear isometry $v : H_\R \recht H_\R$ such that the restriction of $\theta$ to $H_\R$ equals $v$.

Since $\theta$ commutes with $\theta_u$ for all $u \in \cU(P)$, it follows that $v$ commutes with all $u \in \cU(P) \subset \cO(H_\R)$. In particular, $v$ commutes with the multiplication by $i \in \C$, so that $v : H \recht H$ is actually complex linear and $v \in P'$. By the bicommutant theorem, $v \in \cI(N)$.

By construction, $\langle \psi(x),\xi\rangle = \langle \psi_v(x),\xi\rangle$ for a.e.\ $x \in X$. So, $\psi = \psi_v \in \cS$.
\end{proof}

We have now gathered enough material to prove Theorem \ref{thm.prescribed-embeddings-semiring} and Corollary \ref{cor.outer-automorphism-groups}.

\begin{proof}[{Proof of Theorem \ref{thm.prescribed-embeddings-semiring}}]
The theorem is an immediate consequence of Proposition \ref{prop.the-factor-semiring} and Proposition \ref{prop.this-is-all} (for point~1), Proposition \ref{prop.all-closed-perm} (for point~2) and Proposition \ref{prop.bicentralizer-gaussian-von-neumann} (for point~3, by taking an infinite multiplicity representation of $N$).
\end{proof}

\begin{proof}[{Proof of Corollary \ref{cor.outer-automorphism-groups}}]
We first recall that the closed subgroups of the group of all permutations of a countable set can be characterized in two ways. These are precisely the automorphism groups of countable structures. And they are precisely the Polish groups for which the neutral element admits a neighborhood basis consisting of open subgroups. In particular, second countable, locally compact, totally disconnected groups belong to this class, so that point~4 is a special case of point~1.

By Propositions \ref{prop.all-closed-perm} and \ref{prop.bicentralizer-gaussian-von-neumann} and the remarks in the previous paragraph, each of the groups in points~1, 2 and 4 can be realized as closed subgroups $\cG \subg \Aut(X,\mu)$ with $C(\cG)$ acting ergodically and $\cG = C(C(\cG))$. If $K$ is a second countable, compact group, using the left translation action $K \actson (K,\text{\rm Haar})$, the same holds for the groups in point~3.

We finally make the following easy observation: if $\cG_i \subg \Aut(A_i,\tau_i)$ are closed subgroups such that $C(\cG_i)$ acts ergodically and $\cG_i = C(C(\cG_i))$, then $\cG_1 \times \cG_2 \subg \Aut(A_1 \ovt A_2,\tau_1 \ot \tau_2)$ has the same property. Indeed, if $\psi \in \Aut(A_1 \ovt A_2,\tau_1 \ot \tau_2)$ commutes with $C(\cG_1) \times C(\cG_2)$, the ergodicity of $C(\cG_i)$ implies that $\psi(A_1 \ot 1) = A_1 \ot 1$ and $\psi(1 \ot A_2) = 1 \ot A_2$. Thus, $\psi = \psi_1 \ot \psi_2$ with $\psi_i \in C(C(\cG_i)) = \cG_i$.

Therefore, we can make direct products of the groups in points~1, 2, 3 and 4, so that the proof of the corollary is complete.
\end{proof}

\section{Realizing algebraic lattices as intermediate subgroup lattices}

Let $(I,\leq)$ be a partially ordered set. We say that a subset $J \subseteq I$ admits a supremum if the set $\{b \in I \mid a \leq b \;\;\text{for all $a \in J$}\}$ is nonempty and admits a least element, which is then unique. We similarly define the notion of infimum. We call $(I,\leq)$ a lattice of for all $a,b \in I$, the set $\{a,b\}$ admits a supremum, denoted $a \vee b$, and an infimum, denoted $a \wedge b$.

A lattice $(I,\leq)$ is called complete if every subset admits an infimum and a supremum. Note that $(I,\leq)$ then admits a least element, denoted as $0 \in I$, and a greatest element.
Let $(I,\leq)$ be a complete lattice. An element $a \in I$ is called compact if the following property holds: whenever $J \subseteq I$ is a nonempty subset such that $a \leq \sup J$, there exists a finite subset $J_0 \subseteq J$ such that $a \leq \sup J_0$. An algebraic lattice is a complete lattice in which every element can be written as the supremum of a set of compact elements.

Whenever $K$ is a group and $L \subg K$, the set of intermediate subgroups $\{L_1 \mid L \subg L_1 \subg K\}$ partially ordered by inclusion, is a complete lattice. It is actually algebraic: the compact elements are precisely the intermediate subgroups that can be generated by $L$ and a finite subset $\cF \subseteq K$.

Conversely, it was proven in \cite{Tum86} that every algebraic lattice arises in this way as an intermediate subgroup lattice of $L \subg K$. To prove Theorem \ref{thm.complete-intervals}, we need a certain control over $L$, $K$ and the intermediate subgroups. We want $K$ to be a free product of amenable groups and we want that the intermediate subgroups have no nontrivial finite dimensional unitary representations. We also want a control over conjugacy between intermediate subgroups. All this can be realized at once by adapting the proof of \cite{Rep04}. We thus prove the following result.

\begin{theorem}\label{thm.realizing-lattices}
Let $(I,\leq)$ be an algebraic lattice and denote by $I_0 \subseteq I$ the set of compact elements. Let $\kappa$ be the cardinal number $\max\{|I_0|,|\N|\}$. Let $\Lambda$ be any nontrivial countable group. Define $K = \Lambda^{\ast \kappa}$ as the free product of $\kappa$ copies of $\Lambda$. There exists a subgroup $L \subg K$ with the following properties.
\begin{enumlist}
\item The intermediate subgroup lattice $\{L_1 \mid L \subg L_1 \subg K\}$ is isomorphic with $(I,\leq)$.
\item Every intermediate subgroup $L \subg L_1 \subg K$ is freely generated by conjugates of the copies of $\Lambda$ in $K$.
\item For every $g \in K$, we have that $g \in L \vee gLg^{-1}$.
\end{enumlist}
\end{theorem}

We prove Theorem \ref{thm.realizing-lattices} by straightforwardly adapting the proof of \cite{Rep04}. For completeness, we include the details here. In this paper, we only use Theorem \ref{thm.realizing-lattices} in the case where $I_0$ is countable. But again for completeness, we prove the general version here. The only place where this level of generality influences the proof is in Lemma \ref{lem.step2-valuation}, where we need a transfinite induction rather than a standard inductive argument.

As in \cite{Rep04}, we use a dual point of view. Fix an algebraic lattice $(I,\leq)$ and denote by $I_0 \subseteq I$ the set of compact elements. Note that the least element $0$ belongs to $I_0$ and that $a \vee b \in I_0$ when $a,b \in I_0$. In technical terms, $(I_0,\leq)$ is a join semilattice with $0$. A subset $J \subseteq I_0$ is called an ideal if for all $a,b \in J$, we have $a \vee b \in J$ and $x \in J$ whenever $x \in I_0$ and $x \leq a$. The map sending $a \in I$ to the ideal $\{x \in I_0 \mid x \leq a\}$ is an isomorphism between $(I,\leq)$ and the lattice of ideals in $(I_0,\leq)$.

Let $G$ be a group and let $(I_0,\leq)$ be a join semilattice with $0$. As in \cite{Rep04}, a map $\delta : G \recht I_0$ is called a valuation if $\delta(e) = 0$, $\delta(g) = \delta(g^{-1})$ and $\delta(gh) \leq \delta(g) \vee \delta(h)$ for all $g,h \in G$. We denote $G_\delta = \{g \in G \mid \delta(g) = 0\}$. More generally, for every $a \in I_0$, we define the subgroup $G_\delta(a) = \{g \in G \mid \delta(g) \leq a\}$. So, $G_\delta = G_\delta(0)$.

\begin{lemma}\label{lem.step1-valuation}
Let $G$ be a group and let $(I_0,\leq)$ be a join semilattice with $0$. Let $\delta : G \recht I_0$ be a valuation. Fix $g_0 \in G$ and write $a = \delta(g_0)$. Let $K_1,K_2$ be any nontrivial groups and define $\cG = G \ast K_1 \ast K_2$.

There exists a valuation $\delta_1 : \cG \recht I_0$ with the following properties.
\begin{enumlist}
\item For every $h \in G$, we have $\delta_1(h) = \delta(h)$.
\item For every $h \in G$ with $\delta(h) \leq a$, we have that $h \in \cG_{\delta_1} \vee g_0 \cG_{\delta_1} g_0^{-1}$.
\item For every $b \in I_0$, the subgroup $\cG_{\delta_1}(b) \subg \cG$ is a freely generated by conjugates of $G_\delta(a)$, $G_\delta(b)$, $K_1$ and $K_2$ (which need not all occur).
\end{enumlist}
\end{lemma}
\begin{proof}
Note that for any group $K$, we can extend $\delta$ to the valuation $\delta \ast 0$ on $G \ast K$, which is defined as follows. Whenever $x = u_0 v_1 u_1 \cdots v_n u_n$ is a reduced word in $u_i \in G$, $v_i \in K$, we define $(\delta \ast 0)(x) = \delta(u_0) \vee \delta(u_1) \vee \cdots \vee \delta(u_n)$. It is easy to check that $\delta \ast 0$ is indeed a valuation. The level groups of $\delta \ast 0$ are precisely $G_\delta(b) \ast K$.

If $g_0 = e$, we take $\delta_1 = \delta \ast 0 \ast 0$. For the rest of the proof, assume that $g_0 \neq e$. Pick nontrivial elements $k_i \in K_i \setminus \{e\}$. Write $S = G \ast K_1$ and define the valuation $\delta' = \delta \ast 0$ on $S$. Define the subgroup $G_1 \subg S$ by $G_1 = k_1 G_\delta(a) k_1^{-1}$.

We view $\cG$ as the free product $\cG = S \ast K_2$. Define $h_1 \in S$ by $h_1 = g_0 k_1 g_0^{-1}$. Whenever $b \in I_0$ satisfies $a \not\leq b$, define the subgroup $\cG(b) \subg \cG$ by
$$\cG(b) = S_{\delta'}(b) \vee k_2 G_1 k_2^{-1} \vee h_1 K_2 h_1^{-1} \; .$$
We will construct the valuation $\delta_1 : \cG \recht I_0$ such that $\cG_{\delta_1}(b) = \cG(b)$ whenever $a \not\leq b$ and $\cG_{\delta_1}(b) = S_{\delta'}(b) \ast K_2$ if $a \leq b$. We will define $\delta_1$ by considering words in three types of ``letters'':
$$L_1 = \{s \in S \mid s \neq e, \delta'(s) \neq a\} \quad , \quad L_2 = k_2 (G_1 \setminus \{e\}) k_2^{-1} \quad\text{and}\quad L_3 = h_1 (K_2 \setminus \{e\}) h_1^{-1} \; .$$
We claim that if two alternating words $v_1 \cdots v_n$ and $w_1 \cdots w_m$ in the letters $L_1$, $L_2$, $L_3$ define the same element $x \in \cG$, then $n=m$ and $v_i = w_i$ for all $i$. By induction, it suffices to prove that $v_1 = w_1$. View $x$ as an element of the free product $S \ast K_2$. The only reductions that happen in $v_1 \cdots v_n$ and $w_1 \cdots w_m$ arise when a letter $s \in L_1$ is preceded and/or followed by a letter in $L_3$, leading to one of the following elements in $S$: $s$, $s h_1$, $h_1^{-1} s$ or $h_1^{-1} s h_1$. Since $\delta'(s) \neq a$ and $\delta'(h_1) = a$, the elements $s h_1$ and $h_1^{-1} s$ differ from $e$. Since $s \neq e$, also the elements $s$ and $h_1^{-1} s h_1$ differ from $e$. So no further reductions happen.

Note here, as a side remark for later use, that this already means that if $a \not\leq b$, then $\cG(b)$ is the free product of the three groups $S_{\delta'}(b)$, $k_2 G_1 k_2^{-1}$ and $h_1 K_2 h_1^{-1}$.

If $v_1 \neq w_1$, up to exchanging the $v$'s and $w$'s if needed, we must have $v_1 \in L_1$, $v_2 \in L_3$, $w_1 \in L_1$, $w_2 \in L_2$. The reduced word defining $x \in S \ast K_2$ then starts as
$$(v_1 h_1) \, (K_2 \setminus \{e\}) \, (h_1^{-1}s) \, \cdots = w_1 \, k_2 \, (G_1 \setminus \{e\}) \, k_2^{-1} \, \cdots$$
where depending on the subsequent $v_i$, $i \geq 3$, we have $h_1^{-1} s \in \{h_1^{-1}, h_1^{-1} L_1 , h_1^{-1} L_1 h_1\}$. This set does not intersect $G_1 \setminus \{e\}$, because $h_1 \not\in G_1$ and because all elements in $h_1 (G_1 \setminus \{e\}) \cup h_1 (G_1 \setminus \{e\}) h_1^{-1}$ have $\delta'$-valuation equal to $a$, which is the case because $h_1$ and $G_1$ are free inside $S$. So, the claim is proven.

Define the valuation $\delta\dpr : \cG \recht I_0$ by $\delta\dpr(e)=0$ and
$$\delta\dpr(s_0 z_1 s_1 \cdots z_n s_n) = a \vee \delta'(s_0) \vee \cdots \vee \delta'(s_n)$$
when $s_0 z_1 s_1 \cdots z_n s_n$ is a reduced word with $s_0,s_n \in S$, $s_i \in S \setminus \{e\}$ for all $1 \leq i \leq n-1$ and $z_j \in K_2 \setminus \{e\}$ for all $j$. It is easy to check that $\delta\dpr$ is indeed a valuation.

Then define the map $\delta_1 : \cG \recht I_0$ as follows. If $x \in \cG$ can be written as an alternating word with letters from $L_1$, $L_2$ and $L_3$, and if $s_1,\cdots,s_n$ are the letters from $L_1$ that occur in this word, put
$$\delta_1(x) = \delta'(s_1) \vee \cdots \vee \delta'(s_n) \; .$$
If $x$ cannot be written in such a way, then put $\delta_1(x) = \delta\dpr(x)$. By the claim above, $\delta_1$ is well defined.

We prove that $\delta_1$ is indeed a valuation. Since $\delta\dpr(v) \leq a$ for all $x \in L_2 \cup L_3$, we have for all $x \in \cG \setminus \{e\}$ that $a \vee \delta_1(x) = \delta\dpr(x)$. In particular, $\delta_1(x) \leq \delta\dpr(x)$ for all $x \in \cG$.

Fix $x,x' \in \cG \setminus \{e\}$. We have to prove that $\delta_1(xx') \leq \delta_1(x) \vee \delta_1(x')$. If at least one of the $x,x'$ cannot be written as an alternating word with letters from $L_1$, $L_2$ and $L_3$, we have that $\delta_1(x) \vee \delta_1(x') \geq a$ and thus,
$$\delta_1(x) \vee \delta_1(x') = a \vee \delta_1(x) \vee \delta_1(x') = \delta\dpr(x) \vee \delta\dpr(x') \geq \delta\dpr(xx') \geq \delta_1(xx') \; .$$
If both $x,x'$ can be written as an alternating word with letters from $L_1$, $L_2$ and $L_3$, we denote the letters from $L_1$ as $s_1,\ldots,s_n$ and $s'_1,\ldots,s'_m$, respectively. Put
$$b = \delta'(s_1) \vee \cdots \vee \delta'(s_n) \vee \delta'(s'_1) \vee \cdots \vee \delta'(s'_m) = \delta_1(x) \vee \delta_1(x') \; .$$
If $a \not\leq b$, we have that $x,x' \in \cG(b)$ and thus, $xx' \in \cG(b)$. This implies that $\delta_1(xx') \leq b = \delta_1(x) \vee \delta_1(x')$. If $a \leq b$, we get that
$$\delta_1(x) \vee \delta_1(x') = b = a \vee b = a \vee \delta_1(x) \vee \delta_1(x') = \delta\dpr(x) \vee \delta\dpr(x') \geq \delta\dpr(xx') \geq \delta_1(xx') \; .$$
So, $\delta_1$ is a valuation.

By construction, if $a \not\leq b$, we have $\cG_{\delta_1}(b) = \cG(b)$. If $a \leq b$, we have $\cG_{\delta_1}(b) = S_{\delta'}(b) \ast K_2 = G_{\delta}(b) \ast K_1 \ast K_2$.

So, $\delta_1$ satisfies properties 1 and 3 of the lemma. If $h \in G$ and $\delta(h) \leq a$, we have that $h \in G_{\delta}(a)$. So, $h \in k_1^{-1} G_1 k_1 \subg G_1 \vee \cG_{\delta_1}$ because $\delta_1(k_1) = \delta'(k_1) = 0$. By construction, $h_1 K_2 h_1^{-1} \subg \cG_{\delta_1}$, so that $k_2 \in h_1^{-1} \cG_{\delta_1} h_1$. Since $k_1 \in \cG_{\delta_1}$, we have $h_1 \in g_0 \cG_{\delta_1} g_0^{-1}$. We conclude that $k_2 \in \cG_{\delta_1} \vee g_0 \cG_{\delta_1} g_0^{-1}$. But also by construction, $k_2 G_1 k_2^{-1} \subg \cG_{\delta_1}$. Thus, $G_1 \subg k_2^{-1} \cG_{\delta_1} k_2 \subg \cG_{\delta_1} \vee g_0 \cG_{\delta_1} g_0^{-1}$. We proved above that $h \in G_1 \vee \cG_{\delta_1}$. We thus conclude that $h \in \cG_{\delta_1} \vee g_0 \cG_{\delta_1} g_0^{-1}$ and the lemma is proven.
\end{proof}

\begin{lemma}\label{lem.step2-valuation}
Let $G$ be a group and let $(I_0,\leq)$ be a join semilattice with $0$. Let $\delta : G \recht I_0$ be a valuation. Let $\kappa$ be the cardinal number $\max\{|G|,|\N|\}$. Let $\Lambda$ be any nontrivial group. Define $K = G \ast \Lambda^{\ast \kappa}$. There exists a valuation $\delta_1 : K \recht I_0$ with the following properties.
\begin{enumlist}
\item For every $h \in G$, we have $\delta_1(h) = \delta(h)$.
\item If $h,g \in K$ and $\delta_1(h) \leq \delta_1(g)$, then $h \in K_{\delta_1} \vee g K_{\delta_1} g^{-1}$.
\item For every $b \in I_0$, the subgroup $K_{\delta_1}(b)$ is freely generated by conjugates of copies of $\Lambda \subg K$ and subgroups of the form $G_\delta(c) \subg G$, $c \in I_0$.
\end{enumlist}
\end{lemma}

\begin{proof}
Put $K_1 = G \ast \Lambda^{\ast \kappa}$. We first prove that there exists a valuation $\delta_1 : K_1 \recht I_0$ such that 1 and 3 hold, and such that for all $h,g \in G$ with $\delta(h) \leq \delta(g)$, we have $h \in (K_1)_{\delta_1} \vee g (K_1)_{\delta_1} g^{-1}$.

To prove this statement, let $\kappa_0$ be the smallest ordinal with $|\kappa_0| = |G|$ and let $(g_\mu)_{\mu < \kappa_0}$ be a well ordering of $G$. We define the inductive system of extensions $(G_\mu,\delta_\mu)$ by transfinite induction. Put $(G_0,\delta_0) = (G,\delta)$. For a successor ordinal $\mu+1$, we define $(G_{\mu+1},\delta_{\mu+1})$ by applying Lemma \ref{lem.step1-valuation} to $(G_\mu,\delta_\mu)$, the element $g_\mu \in G \subg G_\mu$ and the nontrivial groups $K_1 = K_2 = \Lambda$. In particular, $G_{\mu+1} = G_\mu \ast \Lambda \ast \Lambda$ and whenever $h \in G$ with $\delta(h) \leq \delta(g_\mu)$, we have that
\begin{equation}\label{eq.needed-conj}
h \in (G_{\mu+1})_{\delta_{\mu+1}} \vee g_\mu  (G_{\mu+1})_{\delta_{\mu+1}}  g_\mu^{-1} \; .
\end{equation}
When $\mu$ is a limit ordinal, we define $(G_\mu,\delta_\mu)$ as the direct limit of the inductive system $(G_\lambda,\delta_\lambda)_{\lambda < \mu}$.

We put $(K_1,\delta_1) = (G_{\kappa_0},\delta_{\kappa_0})$. If $h,g \in G$ and $\delta(h) \leq \delta(g)$, we have that $g = g_\mu$ for some $\mu < \kappa_0$. By \eqref{eq.needed-conj}, we have that $h \in (K_1)_{\delta_1} \vee g (K_1)_{\delta_1} g^{-1}$.

We now inductively apply the previous construction, defining an inductive system of extensions $(K_n,\delta_n)$ such that properties 1 and 3 hold, and such that for all $h,g \in K_n$ with $\delta(h) \leq \delta(g)$, we have that $h \in (K_{n+1})_{\delta_{n+1}} \vee g (K_{n+1})_{\delta_{n+1}} g^{-1}$. Defining $(K,\delta)$ as the direct limit of $(K_n,\delta_n)$, the lemma is proven.
\end{proof}

We are now ready to prove Theorem \ref{thm.realizing-lattices}.

\begin{proof}[{Proof of  Theorem \ref{thm.realizing-lattices}}]
Define $G_0 = \Lambda^{\ast I_0}$ as the free product of copies of $\Lambda$ indexed by $I_0$. Denote by $\Lambda_a \subg G_0$ the copy of $\Lambda$ in position $a \in I_0$. Every nontrivial element $g \in G_0 \setminus \{e\}$ can be uniquely written as a reduced word with letters from $\Lambda_a \setminus \{e\}$, $a \in \cF$, where $\cF \subset I_0$ is the finite subset of those types of letters that occur in the reduced expression of $g$. We define $\delta_0(g) = \sup \cF$. It is easy to check that $\delta_0 : G_0 \recht I_0$ is a well defined valuation.

By construction, $\delta_0(h) \leq a$ if and only if $h$ belongs to the free product of $\Lambda_b$, $b \leq a$. We also note that $\delta_0 : G_0 \recht I_0$ is surjective, because $\delta_0(g) = a$ whenever $g \in \Lambda_a \setminus \{e\}$. We finally claim that for all $a,b \in I_0$, there exist $g,h \in G_0$ with $\delta_0(g) = a$, $\delta_0(h) = b$ and $\delta_0(gh) = a \vee b$. When $a=b$ and $a \neq 0$, we pick $g \in \Lambda_a \setminus \{e\}$ and $k \in \Lambda_0 \setminus \{e\}$. Defining $h = k g$, the required properties hold. When $a=b=0$, we take $g=h=e$. When $a \neq b$, we take $g \in \Lambda_a \setminus \{e\}$ and $h \in \Lambda_b \setminus \{e\}$ any nontrivial elements.

We now apply Lemma \ref{lem.step2-valuation} to $(G_0,\delta_0)$. Since $G_0 = \Lambda^{\ast I_0}$ and $\Lambda$ is countable, we have that $|G_0| = \kappa$. We thus find an extension of $(G_0,\delta_0)$ to $(K,\delta)$ such that for all $h,g \in K$ with $\delta(h) \leq \delta(g)$, we have $h \in K_\delta \vee g K_\delta g^{-1}$, and such that the subgroups $K_\delta(b)$ are freely generated by conjugates of $\Lambda$. Define the subgroup $L \subg K$ by $L = K_\delta$.

The original algebraic lattice $(I,\leq)$ is identified with the set of ideals in $(I_0,\leq)$, partially ordered by inclusion. For every ideal $J \subseteq I_0$, define the intermediate subgroup $L \subg K(J) \subg K$ by
$$K(J) = \{g \in K \mid \delta(g) \in J\} \; .$$
Since $\delta_0 : G_0 \recht I_0$ is surjective, the extension $\delta : K \recht I_0$ is surjective. The map $J \mapsto K(J)$ is thus injective and $K(J_1) \subg K(J_2)$ if and only if $J_1 \subseteq J_2$.

We prove that every intermediate subgroup $L \subg S \subg K$ is of the form $S = K(J)$ for an ideal $J \subseteq I_0$. Fix $L \subg S \subg K$. Define $J = \delta(S)$. We have to prove that $J$ is an ideal and that $S = K(J)$.

Let $a \in J$ and let $b \in I_0$ with $b \leq a$. We have to prove that $b \in J$. Take $g \in S$ with $\delta(g) = a$. Since $\delta$ is surjective, take any $h \in G$ with $\delta(h) = b$. Since $\delta(h) \leq \delta(g)$, we have that $h \in L \vee g L g^{-1}$. Since $L \subg S$, we get that $h \in S$. So, $b \in J$. Next, take $a,b \in J$. We have to prove that $a \vee b \in J$. Take $g,h \in S$ such that $\delta(g) = a$ and $\delta(h) = b$. By the properties of $\delta_0 : G_0 \recht I_0$, we can take $g_0,h_0 \in G_0$ such that $\delta(g_0) = a$, $\delta(h_0) = b$ and $\delta(g_0 h_0) = a \vee b$. Since $\delta(g_0) \leq \delta(g)$ and $g \in S$, it follows as in the previous point that $g_0 \in S$. Similarly, $h_0 \in S$. Since $\delta(g_0 h_0) = a \vee b$, we get that $a \vee b \in J$. So, $J \subseteq I_0$ is an ideal.

By definition, $S \subg K(J)$. Take $h \in K(J)$. We have to prove that $h \in S$. Since $\delta(h) \in J$, we can take $g \in S$ with $\delta(g) = \delta(h)$. In particular, $\delta(h) \leq \delta(g)$, so that $h \in L \vee g L g^{-1}$. Because $L \subg S$, we get that $h \in S$.

Statements 1 and 2 of the theorem are now proven. Statement 3 follows because $\delta(g) \leq \delta(g)$.
\end{proof}

\end{document}